\documentclass[11pt, a4paper, reqno]{article}

\usepackage{SdS}
\usepackage[sortcites,
backend=biber,
date=year,
style=numeric-comp,
doi=false,
isbn=false,
url=false,
eprint=true,
maxbibnames=5]{biblatex}

\DeclareNameAlias{author}{family-given}
\AtEveryBibitem{
  \clearfield{month}
  \clearfield{url}
  \clearfield{urlyear}
  \clearfield{urlmonth}
  \clearfield{eprintclass}
  \clearfield{primaryclass}
  \ifentrytype{online}{
    \clearfield{year}}  
  {
    \clearfield{eprint}}
}
\renewbibmacro*{volume+number+eid}{%
  \printfield{volume}%
  \setunit*{\addnbspace}% NEW (optional); there's also \addnbthinspace
  \printfield{number}%
  \setunit{\addcomma\space}%
  \printfield{eid}}
\DeclareFieldFormat[article]{volume}{\textbf{#1}}
\DeclareFieldFormat[article]{number}{\mkbibparens{#1}}
\DeclareFieldFormat*{title}{\textit{#1}}
\DeclareFieldFormat{journaltitle}{#1\isdot}
\renewbibmacro{in:}{}
\numberwithin{equation}{section}

\title{Nonlinear stability of the slowly-rotating \KdS{} family}
\author{Allen Juntao Fang\thanks{Sorbonne
    Universit\'{e}, CNRS, Laboratoire
    Jacques-Louis Lions (LJLL), F-75005 Paris, France,
    fanga@ljll.math.upmc.fr}}
\date{}

\begin{document}
\maketitle

\begin{abstract}
  In this paper, we provide a new proof of nonlinear stability of the
  slowly-rotating \KdS{} family of black holes as a family of
  solutions to the Einstein vacuum equations with cosmological
  constant $\Lambda>0$, originally established by Hintz and Vasy in
  their seminal work \cite{hintz_global_2018}. Using the linear theory
  developed in the companion paper \cite{fang_linear_2021}, we prove
  the nonlinear stability of slowly-rotating \KdS{} using a bootstrap
  argument, avoiding the need for a Nash-Moser argument, and requiring
  initial data small only in the $H^6$ norm.
\end{abstract}

% \normalsize

% \tableofcontents

\section{Introduction}

The aim of this paper is to use the linear theory developed in
\cite{fang_linear_2021} in order to provide a new proof of the
nonlinear stability of the slowly-rotating \KdS{} family of black hole
solutions to Einstein's vacuum equations.

\subsection{The black hole stability problem}

In this paper, we consider the Einstein vacuum equations (EVE) which
govern Einstein's theory of relativity under the assumption of vacuum
and are given by
\begin{equation}
  \label{nonlinear:eq:intro:EVE}
  \Ric(g) - \Lambda g = 0,
\end{equation}
where $g$ is a Lorentzian metric with signature $(-,+,+,+)$ on a
$3+1$ dimensional manifold $\mathcal{M}$, $\Ric$ denotes its Ricci
tensor, and $\Lambda$ is the cosmological constant.

Of particular interest are black hole solutions to
\eqref{nonlinear:eq:intro:EVE}. An explicit two-parameter family of black hole
solutions to EVE with $\Lambda=0$ is given by the asymptotically flat
Kerr black hole solutions. The Kerr family represents a family of
rotating, uncharged black hole solution, and is of particular interest
due to the final state conjecture, which states that all rotating
non-charged black holes are asymptotically Kerr.

A central conjecture in the field is the stability of Kerr
black holes, which surmises that small perturbations of a Kerr
spacetime asymptote to a nearby member of the Kerr family\footnote{For
  black hole solutions, the question of stability is posed at the
  level of a given family of black hole solutions rather than an
  individual solution. This is in line with the physical expectation
  that nontrivial perturbations of a black hole alter its mass and
  angular momentum.}. There has been extensive work
done towards addressing this question. The most simple member of the
Kerr family is in fact the Minkowski spacetime, which does not itself
contain a black hole. The stability of Minkowski was first shown in
the breakthrough result of Christodoulou-Klainerman in
\cite{christodoulou_global_1993}. Following substantial developments
in the field, nonlinear stability of Schwarzschild was then shown by
Klainerman-Szeftel in \cite{klainerman_global_2020}, and by
Dafermos-Holzegel-Rodnianski in \cite{dafermos_non-linear_2021}. Most
recently, Klainerman-Szeftel were able to obtain a nonlinear stability
result for the slowly-rotating Kerr family
\cite{klainerman_kerr_2021}. Behind these nonlinear results, there
was of course substantial work done on the stability of linear
problems on black hole backgrounds. For an overview of the literature
on linear waves on black hole backgrounds and linearized stability of
Einstein's equations, we refer the reader to the introduction of 
\cite{fang_linear_2021}. 

A close relative of the Kerr
family of solutions are the \KdS{} family of black hole solutions,
which is the main focus of the present paper. Like their Kerr cousins,
the \KdS{} black hole spacetimes are a 2-parameter family of black
hole solutions to EVE representing rotating, uncharged black
holes. The key difference between the two is that members of the
\KdS{} family solve EVE with $\Lambda>0$, while members of the Kerr
family solve EVE with $\Lambda=0$.

The \KdS{} family features many of the same inherent geometric
obstacles to stability as the Kerr family. Both families feature
trapped null geodesics which do not escape to either the black hole
horizon or null infinity (the cosmological horizon in the case of
\KdS).  Ergoregions are also present in members of both families,
reflecting a lack of a global timelike Killing vectorfield in both
Kerr and \KdS{} metrics. However, it is generally expected that
proving stability of \KdS{} is substantially easier than proving
stability for Kerr, owing to the spatially compact\footnote{The part
  of \KdS{} treated both in the current paper and in
  \cite{hintz_global_2018} is called the \emph{stationary} region. One
  can also look at the stability of the complementary causal region,
  which is often referred to as the \emph{cosmological (or expanding)
    region}. The cosmological region is itself not spatially compact,
  and features rather different behavior compared to the stationary
  region cf. \cite{schlue_decay_2021}.}  nature of the domain of outer
communication of members of the \KdS{} family. For this reason,
studying the stability of \KdS{} is both an interesting question in
and of itself, and a potential way of gaining insight to further study
of the stability of the Kerr family.

\subsection{Statement of the main theorem}

We now state the main theorem of this paper. 

\begin{theorem} [Nonlinear stability of the slowly-rotating \KdS{}
  family, first version]
  \label{nonlinear:thm:intro:NL-stab:informal}
  Fix $\Lambda>0$. Assume that the initial data of a solution $g$ of
  \eqref{nonlinear:eq:intro:EVE} are sufficiently close to a slowly-rotating
  \KdS{} metric $g_{b^0}$ with black hole parameters
  $b^0 = (M^0, a^0)$. Then $g$ exists globally and moreover, there
  exist black hole parameters $b_\infty = (M_{\infty}, a_\infty)$
  close to $(M^0, a^0)$ such that
  \begin{equation*}
    g - g_{b_\infty} = O(e^{-\SpectralGap \tStar}),\quad \tStar\to \infty,
  \end{equation*}
  for some $\SpectralGap>0$ constant that is independent of the
  initial data. 
\end{theorem}
We refer the reader to Theorem \ref{nonlinear:thm:main} for the precise
statement of the main theorem.  This result was first proven by Hintz
and Vasy in their seminal work \cite{hintz_global_2018} using harmonic
coordinates to treat EVE as a system of quasilinear wave equations in
terms of the metric coefficients. Their proof relies on the framework
of the method of scattering resonances, using microlocal analysis and
the $b$-calculus to prove extremely strong linear stability results,
which they then used to solve the nonlinear problem with a Nash-Moser
scheme.

As expected, the exponential decay we proved in
\cite{fang_linear_2021} at the linear level allows for a relatively
short nonlinear analysis in the present paper. We provide some
comparisons between the proof presented here and the original proof of
Hintz and Vasy in \cite{hintz_global_2018}.
\begin{enumerate}
\item Like in the original work of Hintz and Vasy
  \cite{hintz_global_2018}, we rely on wave coordinates, and the main
  estimates used are linear estimates on exact \KdS{} (see
  \cite{fang_linear_2021}, or Chapter 5 of \cite{hintz_global_2018}).
\item Instead of employing an iterative Nash-Moser scheme to prove
  nonlinear stability, we use a standard  bootstrap
  argument. In particular, we solve the true nonlinear Einstein
  equations up to a finite time, rather than solving a sequence of
  linearized Einstein equations as in \cite{hintz_global_2018}. Doing so
  preserves the structure of EVE and in particular allows us to avoid
  the need for constraint damping in contrast with the approach taken
  by Hintz and Vasy.
\item Using a bootstrap argument instead of Nash-Moser, we close the
  argument with a much lower level of regularity imposed on the
  initial data. Using our method, we expect to actually be able to
  improve the regularity requirements to require only
  $\frac{9}{2}+\delta$, $\delta>0$ derivatives (see Remark
  \ref{nonlinear:remark:reg-threshold}). However, for the sake of avoiding
  fractional Sobolev spaces, we prove the theorem here with initial
  data in $H^6$ instead. Regardless, this represents an improvement
  over the 21 derivatives used in \cite{hintz_global_2018}.
\end{enumerate}

\subsection{Strategy of the proof of Theorem \ref{nonlinear:thm:intro:NL-stab:informal}}
We make the following three bootstrap
assumptions (see Section \ref{nonlinear:sec:BA-assumptions:specific}):
\begin{enumerate}
\item low-regularity exponential decay;
\item high-regularity integrated slow exponential growth; and,
\item that there exists a suitably small gauge choice such that we
  have a certain orthogonality condition. This argument is similar in
  spirit to the last slice argument in other black hole stability
  proofs. See for example
  \cite{klainerman_kerr_2021,klainerman_effective_2019}.
\end{enumerate}
Each of these bootstrap assumptions will be improved in Section
\ref{nonlinear:sec:bs}.
\begin{enumerate}
\item Improving the low-regularity exponential decay bootstrap
  assumption is done in Section \ref{nonlinear:sec:bs-decay}, and is a direct
  consequence of the linear theory in \cite{fang_linear_2021}.
\item Improving the high-regularity exponential growth bootstrap
  assumption is done in Section \ref{nonlinear:sec:bs-energy}. This is done
  using a weak Morawetz estimate on a perturbation of \KdS{} which only
  yields exponential growth, but does not lose derivatives.
\item Improving the gauge bootstrap assumption is done in Section
  \ref{nonlinear:sec:bs-gauge} by relying on the implicit function
  theorem. 
\end{enumerate}
These improvements of the bootstrap assumptions allow us to extend the
bootstrap time (see Section \ref{nonlinear:sec:bs-ext}). This then allows us to
use a continuity argument in \ref{nonlinear:sec:closing-proof} to conclude the
proof of Theorem \ref{nonlinear:thm:intro:NL-stab:informal}.

The previously mentioned weak Morawetz estimate needed to improve the
high-regularity exponential growth bootstrap assumption is a
high-frequency Morawetz estimate on exponentially decaying
perturbations of a slowly-rotating \KdS{} background. This Morawetz
estimate will be proven by perturbing the high-frequency Morawetz 
estimate proven in the linear theory (Theorem
\ref*{linear:thm:resolvent-estimate:main} in
\cite{fang_linear_2021}). The perturbed high-frequency Morawetz
estimate will feature arbitrarily slow exponential growth rather than
small exponential decay but will no longer lose derivatives.  This is
critical in improving the high regularity bootstrap
assumptions. 

\subsection{Outline of the paper}

In Section \ref{nonlinear:sec:setup} we set up the main geometric framework. In
particular, we construct regular coordinate systems on the \KdS{}
family and introduce some crucial vectorfields. In Section
\ref{nonlinear:sec:EVE}, we define the choice of wave gauge for Einstein's
vacuum equations. In Section \ref{nonlinear:sec:main-thm}, we state the precise
form of the main theorem and provide an outline of the proof. In
Section \ref{nonlinear:sec:lin-theory} we review the main points of the linear
theory in \cite{fang_linear_2021} which will be
needed in the nonlinear scheme. 

In Section \ref{nonlinear:sec:BA-Assume}, we construct the semi-global extension
of EVE, and state out bootstrap assumptions for solutions of the
aforementioned semi-global extension. Bootstrap assumptions will be
made at the level of high-regularity growth, low-regularity decay, and
gauge. Each of these will in turn be improved and extended in Section
\ref{nonlinear:sec:bs}, leading to a proof of Theorem \ref{nonlinear:thm:intro:NL-stab:informal}
using a standard continuity argument.

Finally, in Section \ref{nonlinear:sec:energy-perturb-kds}, we prove a
series of energy estimates on exponentially decaying perturbations of
a slowly-rotating \KdS{} background. These estimates are direct
perturbations of estimates in \cite{fang_linear_2021} and are needed
to improve the high-regularity exponential growth bootstrap
assumption.

\subsection{Acknowledgments}

The author would like to acknowledge J\'{e}r\'{e}mie Szeftel for his
support, encouragement, and many helpful discussions. This work is
supported by the ERC grant ERC-2016 CoG 725589 EPGR.

\section{Geometric set up}
\label{nonlinear:sec:setup}

In this section, we define key geometric objects that
we will make use of later.

\subsection{Notational conventions}

Throughout the paper, Greek indices will be used to indicate the
spacetime indices $\{0,1,2,3\}$, lower-case Latin indices will be used
to represent the spatial indices $\{1,2,3\}$, and upper-case Latin
indices will be used to represent angular indices. We also use
$X^\flat$ and $\omega^\sharp$ to denote the canonical one-form associated
to the vectorfield $X$ and the canonical vectorfield associated to the
one-form $\omega$ respectively.

We will use $\cdot$ to indicate the natural contraction between two
objects. For two vectors, $v_1, v_2$, $v_1\cdot v_2$ will indicate
their dot product, while for tensors $u, w$, $u\cdot w$ will indicate
tensor contraction.

We will use $\nabla$ to indicate the full space-time derivative. 

\subsection{The \KdS{} family}
\label{nonlinear:sec:KdS}

The \KdS{} family of black holes, which will be presented explicitly
in what follows, are a family of stationary black hole solutions to
the Einstein vacuum equations (EVE) under the assumption of a positive
cosmological constant $\Lambda>0$. The two-parameter family is parameterized by
\begin{enumerate}
\item the mass of the black hole $M$ and,
\item the angular momentum of the black hole $a = |\AngularMomentum|$,
  where $\frac{\AngularMomentum}{|\AngularMomentum|}$ is the axis of
  symmetry of the black hole. 
\end{enumerate}
We will denote by $B$ the set of black-hole parameters $b=(M, a)$. 

In this paper, we will not deal with the full \KdS{} family, but
instead are primarily concerned only a subfamily characterized by two
features: first, that the mass of the black hole is subextremal
and satisfies $1-9\Lambda M^2>0$; and that the black hole is slowly
rotating, $\abs*{a}\ll M, \Lambda$. The subextremality of the mass ensures
that the event horizon and the cosmological horizon remain physically
separated, and the slow rotation ensures that the trapped set remains
physically separated from both the cosmological and the black-hole
ergoregions.

\subsubsection{The \SdS{} metric}

Given a cosmological constant $\Lambda$, and a black hole mass $M>0$
such that
\begin{equation}
  \label{nonlinear:eq:SdS:non-degeneracy-condition}
  1 - 9\Lambda M^2>0,
\end{equation}
we denote by
\begin{equation*}
  b_0 = (M, \mathbf{0})
\end{equation*}
the black hole parameters for a subextremal \SdS{} black
hole. The \SdS{} family represents a family of spherically symmetric, non-rotating
black hole solutions to Einstein's equations with positive
cosmological constant. On the \emph{domain of outer communication} (also
known in the literature as the \emph{static region}),
$\StaticRegion = \Real_t\times(r_{b_0,\EventHorizonFuture}, r_{b_0,
  \CosmologicalHorizonFuture})\times \Sphere^2$, with
$r_{b_0,\EventHorizonFuture}$, $r_{b_0, \CosmologicalHorizonFuture}$
defined below, the \SdS{} metric a can be expressed in standard
Boyer-Lindquist coordinates by
\begin{equation}
  \label{nonlinear:eq:SdS:metric-def:BL}
  g_{b_0} = - \mu_{b_0}dt^2 + \mu_{b_0}^{-1}dr^2 + r^2\UnitSphereMetric,
\end{equation}
where
\begin{equation*}
  \mu_{b_0}(r) = 1 - \frac{2M}{r}- \frac{\Lambda r^2}{3},
\end{equation*}
and $\UnitSphereMetric$ denotes the standard metric on
$\UnitSphere^2$. The subextremal mass restriction in
(\ref{nonlinear:eq:SdS:non-degeneracy-condition}) guarantees that $\mu_{b_0}(r)$
has three roots: two positive simple roots, and one negative simple
root, 
\begin{equation*}
  r_-<0<r_{b_0,\EventHorizonFuture}<r_{b_0,\CosmologicalHorizonFuture}<\infty.
\end{equation*}
The hypersurfaces defined by
$\{r = r_{b_0, \EventHorizonFuture}\}, \{r = r_{b_0,
  \CosmologicalHorizonFuture}\}$ are the \textit{(future) event
  horizon} and the \textit{(future) cosmological horizon},
respectively, and bound the domain of outer communications, on which
the form of the metric in (\ref{nonlinear:eq:SdS:metric-def:BL}) is
valid. Because we will need to prove estimates which extend slightly
into the black hole interior region and the cosmological region, we
define, fixing $b_0$ and
$0<\varepsilon_{\StaticRegionWithExtension}\ll 1$,
\begin{align}  
  \StaticRegionWithExtension
  &:=
    \Real_{\tStar}
    \times (r_{b_0, \EventHorizonFuture} - \varepsilon_{\StaticRegionWithExtension}, r_{b_0, \CosmologicalHorizonFuture} + \varepsilon_{\StaticRegionWithExtension})
    \times \Sphere^2\label{nonlinear:eq:extended-region-def},\\
  \Sigma
  &:= (r_{b_0, \EventHorizonFuture} - \varepsilon_{\StaticRegionWithExtension}, r_{b_0, \CosmologicalHorizonFuture} + \varepsilon_{\StaticRegionWithExtension}) \label{nonlinear:eq:Sigma-def}.
\end{align}
We also define
\begin{equation}
  \label{nonlinear:eq:extended-horizon-def}
  \EventHorizonFuture_-:=\curlyBrace*{r_{b_0, \EventHorizonFuture} - \varepsilon_{\StaticRegionWithExtension}},\qquad
  \CosmologicalHorizonFuture_+:=\curlyBrace*{r_{b_0, \CosmologicalHorizonFuture} + \varepsilon_{\StaticRegionWithExtension}}.
\end{equation}

\subsubsection{The \KdS{} metric}

More general than the \SdS{} family, the \KdS{} family represents a
stationary, \textit{axi-symmetric}, family of black-hole solutions to
Einstein's equations, of which the \SdS{} family is clearly a
sub-family. In this section, we will detail various useful coordinate
systems that we will use subsequently. We remark that throughout the
paper, we are mainly interested in \KdS{} metrics that are
slowly-rotating, i.e $a = \abs{\AngularMomentum}$ is small relative to
$M, \Lambda$, and thus, close to a \SdS{} relative.

\begin{definition}
  In the Boyer-Lindquist
  $(t, r, \theta, \varphi)\in \Real\times (r_{b, \EventHorizonFuture},
  r_{b, \CosmologicalHorizonFuture})\times (0,\pi)\times
  \Sphere^1_\varphi$ coordinates (with $r_{b, \EventHorizonFuture}$,
  and $r_{b, \CosmologicalHorizonFuture}$ defined below), the \KdS{}
  metric $g_b:=g(M,a)$, and inverse metric
  $G_b := G(M, a)$ take the form:
  \begin{equation}
    \label{nonlinear:eq:KdS-metric:static}
    \begin{split}
      g_b
      &= \rho_b^2\left(\frac{dr^2}{\Delta_b}
        +  \frac{d\theta^2}{\varkappa_b}\right)
      + \frac{\varkappa_b\sin^2\theta}{(1+\lambda_b)^2\rho_b^2 }\left(a\,dt- (r^2+a^2)\,d\varphi\right)^2
      - \frac{\Delta_b}{(1+\lambda_b)^2\rho_b^2}(dt-a\sin^2\theta\,d\varphi)^2,\\
      G_{b} &=  \frac{1}{\rho_b^2}\left(\Delta_{b}\p_r^2 + \varkappa_b\p_\theta^2\right)
      + \frac{(1+\lambda_b)^2}{\rho_b^2\varkappa_b\sin^2\theta}\left(a\sin^2\theta \p_{t} + \p_\varphi\right)^2
      - \frac{(1+\lambda_b)^2}{\Delta_b\rho_b^2}\left(
        \left(r^2+a^2\right)\p_t + a\p_\varphi
      \right)^2,
    \end{split}  
  \end{equation}
  where
  \begin{gather*}
    \Delta_b(r) := (r^2+a^2)\left(1-\frac{1}{3}\Lambda r^2\right) - 2M r,\qquad
    \rho_b^2 := r^2+a^2\cos^2\theta,\\
    \lambda_b :=\frac{1}{3}\Lambda a^2,\qquad
    \varkappa_b:=1+\lambda_b \cos^2\theta. 
  \end{gather*}
\end{definition}

\begin{remark}
  It is easy to observe that the metric reduces to the \SdS{} metric
  $g_{b_0}$ expressed in Boyer-Lindquist coordinates in
  \eqref{nonlinear:eq:SdS:metric-def:BL} when $b=b_0$ ($a=0$). When $a\neq 0$, the
  spherical coordinates $(\theta, \varphi)$ are chosen so that
  $\frac{\AngularMomentum}{\abs{\AngularMomentum}}\in \Sphere^2$ is
  defined by $\theta=0$, and $\p_{\varphi}$ generates counter-clockwise
  rotation around the axis of rotation.  
\end{remark}

\begin{definition}
  As in the \SdS{} case, we define the \emph{event horizon} and the
  \emph{cosmological horizon} of $g_{b}$, denoted by $\EventHorizonFuture$,
  $\CosmologicalHorizonFuture$ to be the $r$-constant hypersurfaces
  \begin{equation*}
    \EventHorizonFuture:=\{r=r_{b, \EventHorizonFuture}\},\qquad \CosmologicalHorizonFuture:=\{r=r_{b, \CosmologicalHorizonFuture}\}
  \end{equation*}
  respectively, where
  $r_{b, \EventHorizonFuture} < r_{b,\CosmologicalHorizonFuture}$ are
  the two largest distinct positive roots of $\Delta_b$.
\end{definition}

A consequence of the implicit function theorem is that these roots
depend smoothly on the black hole parameters $b=(M, a)$, and that for
any $\varepsilon_{\StaticRegionWithExtension}>0$, there exists $a$ sufficiently
small such that $\abs*{r_{b,\EventHorizonFuture} - r_{b_0,
    \EventHorizonFuture}} + \abs*{r_{b,\CosmologicalHorizonFuture} - r_{b_0,
    \CosmologicalHorizonFuture}} <
\varepsilon_{\StaticRegionWithExtension}$. Moreover, since these
two horizons are null, the domain of outer communications, bounded
by $\EventHorizonFuture, \CosmologicalHorizonFuture$ is
a causal domain that is foliated by compact space-like hypersurfaces.

Both the \SdS{} metric \eqref{nonlinear:eq:SdS:metric-def:BL} and the \KdS{}
metric (\ref{nonlinear:eq:KdS-metric:static}) in the Boyer-Lindquist coordinates
have singularities at the event horizon and the cosmological
horizon. Fortunately, this is only a coordinate singularity and can be
smoothed out by an appropriate change of coordinates.

\subsubsection{Regular coordinates on \KdS{} spacetimes}

It is clear upon inspection that the \SdS{} and \KdS{} metrics in
Boyer-Lindquist coordinates as expressed in
\eqref{nonlinear:eq:SdS:metric-def:BL} and \eqref{nonlinear:eq:KdS-metric:static} have
degeneracies at the event and cosmological horizons. However, these
are merely coordinate singularities, and we can construct a smooth
coordinate system that extends beyond both the event and the
cosmological horizon.

On \SdS{}, we introduce the
function
\begin{equation*}
  \tStar := t - F_{b_0}(r).
\end{equation*}
In the coordinates
$(\tStar , r, \omega)$, the \SdS{} metric, $g_{b_0}$, and the inverse
metric, $G_{b_0}$, take the form
\begin{equation}
  \label{nonlinear:eq:SdS:regular}
  \begin{split}
    g_{b_0} &= -\mu_{b_0} \,d\tStar^2
    + 2F_{b_0}'(r)\mu \,d\tStar dr
    +(\mu_{b_0} F_{b_0}'(r)^2 +\mu_{b_0}^{-1})\,dr^2
    + r^2\UnitSphereMetric,\\
    G_{b_0} &= -\left(\mu_{b_0}^{-1} +\mu_{b_0} F_{b_0}'(r)^2\right)\p_{\tStar}^2
    + 2F_{b_0}'(r)\mu_{b_0}\,\p_{\tStar}\p_r
    + \mu_{b_0}\p_r^2 + r^{-2}\UnitSphereInvMetric,
  \end{split}
\end{equation}
which can be required to fulfill certain useful properties.
\begin{lemma}
  \label{nonlinear:lemma:SdS:Kerr-star-regular-coordinates}
  Fix some interval
  $\Interval_{b_0}\subset (r_{b_0,\EventHorizonFuture},
  r_{b_0,\CosmologicalHorizonFuture})$. Then there exists a choice of
  $F_{b_0}(r)$ such that
  \begin{enumerate}
  \item the $\tStar$-constant hypersurfaces are space-like. That is,
    that
    \begin{equation*}
      -\frac{1}{\mu_{b_0}} + F_{b_0}'(r)^2\mu_{b_0} < 0;
    \end{equation*}
  \item $F_{b_0}(r)$ satisfies that
    \begin{equation*}
    F_{b_0}(r) \ge 0,\qquad  r\in(r_{b_0,\EventHorizonFuture},r_{b_0,\CosmologicalHorizonFuture}),
  \end{equation*}
  with equality for $r\in \Interval_{b_0}$.
  \end{enumerate}
\end{lemma}
\begin{proof}
  See Lemma \ref*{linear:lemma:SdS:Kerr-star-regular-coordinates} in \cite{fang_linear_2021}.
\end{proof}

We construct such a new, Kerr-star, coordinate system explicitly (see
similar constructions in Section 5.5 of \cite{dafermos_lectures_2008},
Section 4 of \cite{tataru_local_2010}, and Section 3.2 of
\cite{hintz_global_2018}). First define the new variables
\begin{equation}
  \label{nonlinear:eq:KdS-regular-coordinates:Hintz-Vasy-regular-def}
  \tStar = t - F_b(r),\quad \phiStar = \varphi - \Phi_b(r),
\end{equation}
where $F_b$ and $\Phi_b$ are smooth functions on
$(r_{b_0, \EventHorizonFuture} + \varepsilon_{\StaticRegionWithExtension},
r_{\CosmologicalHorizonFuture} -
\varepsilon_{\StaticRegionWithExtension})$.
We can then compute that the metric takes the form
\begin{equation}
  \label{nonlinear:eq:KdS:regular}
  \begin{split}
    g_b =&{}  \frac{\varkappa \sin^2\theta}{(1+\lambda_b)^2\rho_b^2}
    \left(
      a(d\tStar+ F_b'\,dr) - (r^2+a^2)(d\phiStar + \Phi_b'\,dr)
    \right)^2\\
    & - \frac{\Delta_b}{(1+\lambda_b)\rho_b^2}\left(
      d\tStar + F_b'\,dr
      - a \sin^2\theta(d\phiStar +  \Phi_b'\,dr)
      - \frac{(1+\lambda_b)\rho_b^2}{\Delta_b} \,dr
    \right)\\
    & + \frac{2}{1+\lambda_b}\left(
      d\tStar + F_b'\,dr
      - a \sin^2\theta(d\phiStar + \Phi_b'\,dr)
      - \frac{(1+\lambda_b)\rho_b^2}{\Delta_b} \,dr
    \right)dr
    +\frac{\rho_b^2}{\varkappa_b}\,d\theta^2.  
  \end{split}  
\end{equation}

We pick the $F_b, \Phi_b$ so that the $(\tStar,r,\theta,\phiStar)$
coordinate system expends smoothly beyond the horizons, is identical
to the Boyer-Lindquist coordinates on a small neighborhood of $r=3M$,
and such that the $\tStar$-constant hypersurfaces are spacelike.
\begin{lemma}
  \label{nonlinear:lemma:KdS:Kerr-star-regular-coordinates}
  Fix an interval $\Interval_b := (r_1, r_2)$ such that
  $r_{b_0,\EventHorizonFuture}+\varepsilon_{\StaticRegionWithExtension}<r_1<r_2<r_{b_0,\CosmologicalHorizonFuture}
  - \varepsilon_{\StaticRegionWithExtension}$. Then we can pick
  $F_b, \Psi_b$ so that
  \begin{enumerate}
  \item the choice extends the choice of regular coordinates for
    \SdS{} in \eqref{nonlinear:eq:SdS:regular} in the sense that when
    $b=b_0$,
    \begin{equation*}
      F_{b} = F_{b_0},\qquad \Phi_b = 0,
    \end{equation*}
    where $F_{b_0}$ is that of Lemma \ref{nonlinear:lemma:SdS:Kerr-star-regular-coordinates};
  \item $F_b(r) \ge 0$ for $r\in (r_{b, \EventHorizonFuture}, r_{b,
      \CosmologicalHorizonFuture})$ with equality for $r\in
    \Interval_b$;
  \item the $\tStar$-constant hypersurfaces are space-like, and in
    particular, defining
    \begin{equation}
      \label{nonlinear:eq:GInvdtdt-def}
      \GInvdtdt_b:= - \frac{1}{G_b(d\tStar, d\tStar)},
    \end{equation}
    we have that
    \begin{equation*}
      1\lesssim \GInvdtdt_b \lesssim 1
    \end{equation*}
    uniformly on $\StaticRegionWithExtension$;
  \item the metric $g_b$ is smooth on $\StaticRegionWithExtension$. 
  \end{enumerate}
\end{lemma}
\begin{proof}
  See Lemma \ref*{linear:lemma:KdS:Kerr-star-regular-coordinates} in \cite{fang_linear_2021}.
\end{proof}

\subsection{The Killing vectorfields $\KillT$, $\KillPhi$}
\label{nonlinear:sec:setup:killing}

\begin{definition}
  Define the vectorfields $\KillT = \p_{\tStar}$,
  $\KillPhi=\p_{\phiStar}$ using the Kerr-star coordinates in
  \eqref{nonlinear:eq:KdS:regular}. From the fact that the expression of $g_b$
  in the Kerr-star coordinates is independent of $\tStar$ and
  $\phiStar$, we immediately have that $\KillT, \KillPhi$ are Killing
  vectorfields.
\end{definition}

\begin{definition}
  On \KdS{} spacetimes, the \emph{ergoregion} is defined by
  \begin{equation*}
    \Ergoregion := \{(t,x): g_b(\KillT,\KillT)(t,x) > 0\}.
  \end{equation*}
  We define its boundary, the set of points where $\KillT$ is null, by
  the \emph{ergosphere}. 
\end{definition}
\begin{remark}
  Observe that for the \SdS{} sub-family, the ergosphere is exactly
  the event horizon and the cosmological horizon, and $\KillT$ is
  timelike on the whole of the interior of the domain of outer
  communication. 
\end{remark}

\subsection{Energy momentum tensor and divergence formulas}
\label{nonlinear:sec:EMT-and-div-thm}
In this section, we define the energy momentum tensor and some basic
divergence properties. Given a complex matrix function  $h$, let us
denote its complex conjugate by $\overline{h}$. Moreover, for any
2-tensor $h_{\mu\nu}$, we denote its symmetrization
\begin{equation*}
  h_{(\mu\nu)} = \frac{1}{2}\left(
    h_{\mu\nu} + h_{\nu\mu}
  \right).
\end{equation*}
\begin{definition}
  Let
  \begin{equation*}
    h: \StaticRegionWithExtension\to \Complex^D
  \end{equation*}
  be a complex-valued matrix function.  We define the
  \emph{energy-momentum tensor} to be the symmetric 2-tensor:
  \begin{equation*}
    \EMTensor_{\mu\nu}[h]
    = \nabla_{(\mu}\overline{h}\cdot\nabla_{\nu)}h
    - \frac{1}{2}g_{\mu\nu}\nabla_\alpha \overline{h}\cdot  \nabla^\alpha h,
  \end{equation*}
  where $\cdot$ here is the dot product between matrices. 
\end{definition}
The energy-momentum tensor satisfies the following divergence
property:
\begin{equation}
  \label{nonlinear:eq:EMTensor:divergence-property}
  \nabla_\mu \tensor[]{\EMTensor}{^\mu_\nu}[h]
  = \Re\left(\nabla_\nu\overline{h} \cdot \ScalarWaveOp[g]h\right),
\end{equation}
where we denote by 
\begin{equation*}
  \ScalarWaveOp[g] = \nabla^\alpha \partial_\alpha
\end{equation*}
the scalar wave operator\footnote{We will use
  $\nabla^\alpha\nabla_\alpha$ to denote the vectorial or tensorial
  wave operator.}.

This property will be the key to producing the various divergence
equations we use to derive the relevant energy estimates in the
subsequent sections.
\begin{definition}
  Let $X$ be a smooth vectorfield on $\StaticRegionWithExtension$, $m$
  be a smooth one-form on $\StaticRegionWithExtension$, and
  $q$ be a smooth function on $\StaticRegionWithExtension$.  We will
  refer to $X$ as the \emph{(vectorfield) multiplier}, to $m$ as the
  \emph{auxiliary zero-order corrector}, and to $q$ as
  the \emph{Lagrangian corrector}. Then define
  \begin{equation}
    \label{nonlinear:eq:J-K-currents:def}
    \begin{split}
      \JCurrent{X,q,m}_\mu[{h}]
      &= X^\nu\tensor[]{\EMTensor}{_{\mu\nu}}[{h}] 
      + \frac{1}{2}q\nabla_\mu (\abs*{h}^2)
      - \frac{1}{2}\nabla_\mu q\abs*{{h}}^2
      + \frac{1}{2}m_\mu \abs*{h}^2,\\
      \KCurrent{X,q,m}[{h}]
      &= \DeformationTensor{X}{}\cdot \EMTensor[h]
      + q \nabla^\alpha{h}\cdot \nabla_\alpha\overline{{h}}
      + \frac{1}{2}\nabla_{m^\sharp}\abs{h}^2
      + \frac{1}{2}\nabla^\alpha (m_\alpha - \p_\alpha q) \abs*{{h}}^2,    
    \end{split}
  \end{equation}
  where
  \begin{equation*}
    \DeformationTensor{X}{_{\mu\nu}} := \frac{1}{2}\left(\nabla_\mu X_\nu + \nabla_\nu X_\mu\right)
  \end{equation*}
  denotes the \emph{deformation tensor} of $X$. 
\end{definition}

\begin{prop}
  \label{nonlinear:prop:energy-estimates:spacetime-divergence-prop}
  Let $X$ denote a sufficiently regular vectorfield on
  $\StaticRegionWithExtension$, a \KdS{} black hole spacetime, and
  $\DomainOfIntegration$ denote the spacetime region bounded by
  $\Sigma_{t_1}, \Sigma_{t_2}, \EventHorizonFuture_-$, and
  $\CosmologicalHorizonFuture_+$. Moreover, denote
  \begin{equation*}
    \EventHorizonFuture_{t_1,t_2} := \EventHorizonFuture_-\bigcap \{t_1\le \tStar\le t_2\},\qquad
    \CosmologicalHorizonFuture_{t_1,t_2} := \CosmologicalHorizonFuture_+\bigcap \{t_1\le \tStar\le t_2\}.
  \end{equation*}
  where we recall the definitions of $\EventHorizonFuture_-,
  \CosmologicalHorizonFuture_+$ in \eqref{nonlinear:eq:extended-horizon-def}.
  
  Then we have the following divergence
  property: 
  \begin{equation*}
    \begin{split}
      -\int_{\DomainOfIntegration} \nabla_g\cdot X ={}\int_{\Sigma_{t_2}}X\cdot n_{\Sigma_{t_2}}  -
      \int_{\Sigma_{t_1}}X\cdot n_{\Sigma_{t_1}} +
      \int_{\EventHorizonFuture_{t_1,t_2}}
      X \cdot n_{\EventHorizonFuture_-}  +
      \int_{\CosmologicalHorizonFuture_{t_1,t_2}} X\cdot n_{\CosmologicalHorizonFuture_+}
      .
    \end{split}
  \end{equation*}
  Here, $n_{\Sigma_t}$ is the future-directed unit normal on
  $\Sigma_t$, and $n_{\Horizon}$ denotes the future-directed timelike
  normal of $\Horizon$. 
\end{prop}

In proving the relevant energy estimates, we will typically apply the
divergence formulas in Proposition
\ref{nonlinear:prop:energy-estimates:spacetime-divergence-prop} using the divergence property
\begin{equation}
  \label{nonlinear:eq:div-them:J-K-currents}
  \nabla_g\cdot \JCurrent{X,q,m}[{h}] = \Re\squareBrace*{
    (X+q)\overline{{h}} \cdot  \ScalarWaveOp[g]{h} 
  }
  + \KCurrent{X,q,m}[{h}].
\end{equation}
We will refer to $X$ and $q$ as the \emph{vectorfield (multiplier)},
and the \emph{Lagrangian correction} respectively. When we want to
emphasize the background on which $\JCurrent{X,q,m}$ or
$\KCurrent{X,q,m}$ are defined, we will use the notation
$\JCurrent{X,q,m}_g$ or $\KCurrent{X,q,m}_g$, where $g$ is the
background metric.

\subsection{The redshift vectorfields $\RedShiftN$}

In this subsection, we recall the construction of the redshift
vectorfield $\RedShiftN$.

\begin{prop}
  \label{nonlinear:prop:redshift:N-construction}
  Let $b=(M, a)$, $|a|\ll M, \Lambda$ be the black hole parameters for
  a \KdS{} black hole $g_b$, and let $\Sigma_{\tStar}$ be a
  $\tStar$-constant uniformly spacelike hypersurface. There exist
  positive constants $c_{\RedShiftN}$ and $C_{\RedShiftN}$, parameters
  $r_{\EventHorizonFuture}< r_0<r_1<R_1<R_0<
  r_{\CosmologicalHorizonFuture}$, and a stationary uniformly
  time-like vectorfield $\RedShiftN$ such that
  \begin{enumerate}
  \item $\KCurrent{\RedShiftN, 0, 0}_{g_b}[{h}] \ge c_{\RedShiftN} \JCurrent{\RedShiftN, 0, 0}_{g_b}[{h}]\cdot
    n_{\Sigma_{\tStar}}$ for
    $r_{\EventHorizonFuture}\le r\le r_0$, and for
    $R_0\le r\le r_{\CosmologicalHorizonFuture}$.
  \item
    $-\KCurrent{\RedShiftN, 0, 0}_{g_b}[{h}] \le C_{\RedShiftN}
    \JCurrent{\RedShiftN, 0, 0}_{g_b}[{h}]\cdot n_{\Sigma_{\tStar}}$ for
    $r_0\le r\le R_0$.
  \item $\RedShiftN=\KillT$ for $r_1\le r\le R_1$.
  \item There exists some $\delta>0$ such that
    \begin{equation}
      \label{nonlinear:eq:RedShiftN-timelike}
      g_b(\RedShiftN, \RedShiftN) < -\delta.
    \end{equation}
  \end{enumerate}
\end{prop}

\begin{proof}
  See Proposition \ref*{linear:prop:redshift:N-construction} in \cite{fang_linear_2021}.
\end{proof}

We use the redshift vectorfield $\RedShiftN$ defined above to
construct higher-order Sobolev spaces. To do so, we will require the
following technical lemma (see Lemma 3.11 of
\cite{warnick_quasinormal_2015} for the anti-de Sitter equivalent)
which constructs the $\RedShiftK_i$ vectorfields.

\begin{lemma}
  \label{nonlinear:lemma:enhanced-redshift:Ka-construction}
  There exists a finite collection of vectorfields
  $\curlyBrace*{\RedShiftK_i}_{i=1}^N$
  with the following properties:
  \begin{enumerate}
  \item $\RedShiftK_i$ are stationary, smooth vectorfields on $\StaticRegionWithExtension$.
  \item Near $\EventHorizonFuture$, $\RedShiftK_1$ is future-oriented null with
    $g(\RedShiftK_1,\HorizonGen_{\EventHorizonFuture}) = -1$, and near
    $\CosmologicalHorizonFuture$, $\RedShiftK_1$ is future-oriented null with
    $g(\RedShiftK_1,\HorizonGen_{\CosmologicalHorizonFuture}) = -1$.
  \item $\RedShiftK_i$ are tangent to both $\EventHorizonFuture$ and
    $\CosmologicalHorizonFuture$ for $2\le i\le N$.
  \item If $X$ is any vectorfield supported in
    $\StaticRegionWithExtension$, then there exist smooth functions
    $x^i$, not necessarily unique, such that
    \begin{equation*}
      X =  \sum_i x^i\RedShiftK_i.
    \end{equation*}
  \item We have the following decomposition of the deformation
    tensor of $\RedShiftK_i$,
    \begin{equation}
      \label{nonlinear:eq:enhanced-redshift:Deform-Tens-decomp}
      \DeformationTensor{\RedShiftK_i}=\sum_{j,k}f^{jk}_i \RedShiftK^\flat_j \otimes_s \RedShiftK^\flat_k,
    \end{equation}
    for stationary functions
    $f^{jk}_i = f^{kj}_i\in C^\infty_c(\StaticRegionWithExtension)$,
    where $X^\flat$ is the canonical one-form associated to the
    vectorfield $X$, and on
    $\Horizon \in \curlyBrace*{\EventHorizonFuture,
      \CosmologicalHorizonFuture}$,
    \begin{align*}
      f^{11}_1 &= \SurfaceGravity_{\Horizon},\\
      f^{11}_i &=0,\quad i\neq 1.
    \end{align*}
  \end{enumerate}
\end{lemma}
\begin{proof}
  See Lemma
  \ref*{linear:lemma:enhanced-redshift:Ka-construction} in \cite{fang_linear_2021}. 
\end{proof}

\subsection{Sobolev spaces}
\label{nonlinear:sec:Sobolev-spaces}

In this section, we define the Sobolev spaces that we make use of
throughout the paper when discussing the Cauchy problem or
establishing energy estimates.
\begin{definition}
  Let $h:\StaticRegionWithExtension\to \Complex^D$. Then denoting by
  $\DomainOfIntegration$ a subset of $\StaticRegionWithExtension$, we
  define the regularity spaces
  \begin{equation*}
    L^2(\DomainOfIntegration) := \curlyBrace*{h: \int_{\DomainOfIntegration}\abs*{h}^2 <\infty}, \qquad
    H^k(\DomainOfIntegration) := \curlyBrace*{h: \RedShiftK^\alpha h\in L^2(\DomainOfIntegration), |\alpha|\le k}. 
  \end{equation*}
\end{definition}

Let us also define two $L^2$ inner products on spacelike
slices. 
\begin{definition}
  We define the $\InducedLTwo$ and $\LTwo$ inner products on the spacelike
  slice $\Sigma$ by
  \begin{equation*}
    \begin{split}
      \bangle{{h}_1, {h}_2}_{\InducedLTwo(\Sigma)}
      &=  \int_{\Sigma} {h}_1 \cdot \overline{h}_2\, \sqrt{\GInvdtdt},\\
      \bangle{{h}_1, {h}_2}_{\LTwo(\Sigma)}
      &= \int_{\Sigma}{h}_1\cdot\overline{h}_2
    \end{split}
  \end{equation*}
  Observe that due to the construction of the $(\tStar, r, \theta,
  \phiStar)$ Kerr-star coordinate system, the two norms are
  equivalent. 
\end{definition}

We likewise have the following higher regularity Sobolev norms.
\begin{definition}
  For any $u:\Sigma\to\Complex^D$, let
  $\upsilon:\StaticRegionWithExtension\to \Complex^D$ be the unique
  lifting satisfying
  \begin{equation}
    \label{nonlinear:eq:stationary-extension-def}
    \upsilon\vert_{\Sigma} = u,\quad \KillT \upsilon = 0.
  \end{equation}  
  Then define the regularity spaces $H^k(\Sigma), \InducedHk{k}(\Sigma)$ by:
  \begin{equation*}
    \begin{split}
      H^k(\Sigma)&:= \curlyBrace*{u:\left.\RedShiftK^\alpha
          \upsilon\right\vert_{\Sigma} \in \LTwo(\Sigma), |\alpha|\le
        k}, \\
      \InducedHk{k}(\Sigma)&:=\curlyBrace*{u:\left.\RedShiftK^\alpha
          \upsilon\right\vert_{\Sigma} \in \InducedLTwo(\Sigma), |\alpha|\le
        k}, 
    \end{split}
  \end{equation*}
  where $\alpha$ is a multi-index and $\RedShiftK_i$ are vectorfields satisfying
  the requirements of Lemma \ref{nonlinear:lemma:enhanced-redshift:Ka-construction}.

  With the same $\RedShiftK_i$, we define the regularity space $\HkWithT{k}(\Sigma)$ by:
  \begin{equation*}
    \HkWithT{k}(\Sigma_{\tStar}) :=\curlyBrace*{h:\left.\RedShiftK^\alpha
        h(\tStar, \cdot)\right\vert_{\Sigma_{\tStar}} \in \InducedLTwo(\Sigma), |\alpha|\le k}. 
  \end{equation*}
\end{definition}
\begin{remark}
  Observe that different choices of the
  family $\curlyBrace{\RedShiftK_i}_{i=1}^N$ will result in different, though
  equivalent, $H^k(\Sigma), \InducedHk{k}(\Sigma), \HkWithT{k}(\Sigma)$ norms.
\end{remark}

We also use the vectorfields defined in Theorem
\ref{nonlinear:lemma:enhanced-redshift:Ka-construction} to define the following
higher regularity Sobolev spaces.
\begin{definition}
  Given some
  $(\psi_0,\psi_1)\in H^k_{\local}(\Sigma, \Complex^D)\times
  H^{k-1}_\local(\Sigma, \Complex^D)$, letting
  $\curlyBrace{\RedShiftK_i}_{i=1}^N$ be as constructed in Lemma
  \ref{nonlinear:lemma:enhanced-redshift:Ka-construction}. We  define
  $\LSolHk{k}(\Sigma)$ to be the space consisting of $(\psi_0,\psi_1)$
  such that
  \begin{equation*}
    \norm{(\psi_0, \psi_1)}^2_{\LSolHk{k}(\Sigma)} :=
    \norm*{\psi_0}^2_{\Hk{k}(\Sigma)} + \norm{\psi_1}^2_{\Hk{k-1}(\Sigma)}.
  \end{equation*}
\end{definition}

We also define the following exponentially weighted Sobolev spaces.
\begin{definition}
  For $\alpha \in \Real$, $h:
  \StaticRegionWithExtension\to\Complex^D$, define the \emph{weighted
    Sobolev norm} with regularity $k$ and decay $\alpha$, 
  \begin{align*}
    \norm{h}_{H^{k,\alpha}(\StaticRegionWithExtension)}^2:=
    \int_0^\infty e^{2\alpha\tStar} \norm*{h}_{\HkWithT{k}(\Sigma_{\tStar})}^2\,d\tStar.
  \end{align*}
  Then define the forcing parameter space of regularity $k$, decay
  $\alpha$, and loss $m$ by
  \begin{equation}
    \label{nonlinear:eq:D-kam-norm}
    \norm{(f,h_0,h_1)}_{D^{k,\alpha, m}(\StaticRegionWithExtension)}
    := \norm{f}_{H^{k+m-1,\alpha}(\StaticRegionWithExtension)} + \norm{(h_0, h_1)}_{\LSolHk{k+m}(\Sigma_0)}.
  \end{equation}
\end{definition}

\subsection{Implicit function theorem}

We provide the explicit statement of the implicit function theorem
that we use in the present paper. 
\begin{theorem}[Theorem A.1.2 in \cite{alinhac_pseudo-differential_2007}]
  \label{nonlinear:thm:IFT}
  Let $X,Y,Z$ be Hilbert spaces. Let $U\subset X, V\subset Y$ be open
  subsets and $\mathcal{F}:U\times V\to Z$ be a $C^1$ mapping. If for
  some $(x_0,y_0)\in U\times V$, the linearization in the first
  argument
  \begin{equation*}
    D_1\mathcal{F}|_{(x_0,y_0)}:X\to Z
  \end{equation*}
  is an isomorphism, then there exist open neighborhoods $U_0, V_0$
  such that $x_0\in U_0\subset U$, $y_0\in V_0\subset V$, and a unique
  $C^1$-mapping $\mathcal{G}:V_0\to U_0$ such that for all $y\in V_0$,
  \begin{equation*}
    \mathcal{F}(\mathcal{G}(y),y)=\mathcal{F}(x_0, y_0),
  \end{equation*}
  and $\mathcal{F}(x, y) = \mathcal{F}(x_0, y_0)$ is equivalent to $x
  = \mathcal{G}(y)$ for $(x, y)\in U_0\times V_0$. 
\end{theorem}

\section{Einstein's equations}
\label{nonlinear:sec:EVE}

In this section, we review the formulation of Einstein's equations in
harmonic gauge and the nonlinear hyperbolic Cauchy problem associated
to Einstein's equations in harmonic gauge. For an in-depth reference
on harmonic coordinates, we refer the reader to Chapter 7 of
\cite{choquet-bruhat_general_2009}. 

\subsection{Harmonic gauge  and the hyperbolic Cauchy
  problem}
\label{nonlinear:sec:HC-Cauchy-problem}

Recall that Einstein's vacuum equations with a cosmological
constant $\Lambda$ for $g$, a $(-, +, +, +)$ Lorentzian metric  on a
smooth manifold $\mathcal{M}$, are
\begin{equation}
  \label{nonlinear:eq:EVE:Full}
  \Ric(g) - \Lambda g= 0.
\end{equation}
For any globally hyperbolic $(\mathcal{M}, g)$ solution to
(\ref{nonlinear:eq:EVE:Full}), and spacelike hypersurface
$\Sigma_0\subset\mathcal{M}$, the induced Riemannian metric $\InducedMetric$
on $\Sigma_0$ and the second fundamental form $k(X,Y)$ of $\Sigma_0$
satisfy the \emph{constraint equations}
\begin{equation}
  \label{nonlinear:eq:EVE:constraint-eqns}
  \begin{split}
    R(\InducedMetric) + (\Trace_{\InducedMetric}k)^2 - \abs{k}_{\InducedMetric}^2 &= -2\Lambda,\\
    \nabla_{\InducedMetric} \cdot k - d \Trace_{\InducedMetric} k &= 0,
  \end{split}
\end{equation}
where $R(\InducedMetric)$ is the scalar curvature of $\InducedMetric$.  The Cauchy
problem for Einstein's equations then asks, given an initial data set
consisting of the triple $(\Sigma_0, \InducedMetric, k)$, where
$\InducedMetric$ is a Riemannian metric on the smooth $3$-manifold
$\Sigma_0$, and $k$ is a symmetric 2-tensor on $\Sigma_0$, for a
Lorentzian 4-manifold $(\mathcal{M}, g)$ and an embedding
$\Sigma_0\hookrightarrow \mathcal{M}$ such that $\InducedMetric$ is
the induced metric on $\Sigma_0$, and $k$ is the second fundamental
form of $\Sigma_0$ in $\mathcal{M}$. We denote initial data triplets
with $(\InducedMetric , k)$ satisfying the constraint equations
\eqref{nonlinear:eq:EVE:constraint-eqns} to be \textit{admissible} initial data
triplets.

It is well-known that (\ref{nonlinear:eq:EVE:Full}) is a quasilinear
second-order partial differential system of equations for the metric
coefficients $g_{\mu\nu}$. In local coordinates, we can write
\eqref{nonlinear:eq:EVE:Full} as
\begin{equation}
  \label{nonlinear:eq:EVE-coords}
  \Ric(g)_{\mu\nu}
  = - \frac{1}{2}g^{\alpha\beta}\p_{\alpha}\p_{\beta}g_{\mu\nu}
  + \p_{(\mu}\Gamma(g)_{\nu)} + \mathcal{N}(g, \p g),\qquad
  \Gamma(g)^\mu:= g^{\alpha\beta}\ChristoffelTypeTwo[g]{\mu}{\alpha\beta},
\end{equation}
where the nonlinear term $\mathcal{N}(g, \p g)$ involves at most one
derivative of $g$. As a result of the presence of the
$\nabla_{(\mu}\Gamma(g)_{\nu)}$ term EVE lacks any useful
structure. However, as was first demonstrated by Choquet-Bruhat, this
problem can be overcome by using the general covariance of Einstein's
equations and choosing \textit{wave coordinates}, also referred to as
\textit{harmonic coordinates}. With this choice
Einstein's equations become a quasilinear hyperbolic system of
equations
\cite{choquet-bruhat_theoreme_1952,choquet-bruhat_global_1969}.

We first introduce the harmonic coordinate condition. 
\begin{definition}
  Define the \emph{(gauge) constraint operator}
  \begin{equation*}
    \Constraint(g, g^0)_\mu = g_{\mu\chi}g^{\nu\lambda}\left(\ChristoffelTypeTwo[g]{\chi}{\nu\lambda} - \ChristoffelTypeTwo[g^0]{\chi}{\nu\lambda} \right),
  \end{equation*}
  where $g^0$ is a fixed background metric which solves Einstein's
  equations. 
  Then we say that a Lorentzian metric $g$ satisfies the
  \emph{harmonic coordinate condition} (with respect to $g^0$) if
  \begin{equation}
    \label{nonlinear:eq:GHC-condition:Christoffel}
    \Constraint(g, g^0) = 0.
  \end{equation}
\end{definition}
\begin{remark}
  Instead of vanishing right-hand side, if we instead set the
  right-hand side equal to some one-form $\bH(g)$ depending on $g$ but
  not its derivatives, we would obtain the \emph{generalized harmonic
    gauge}\footnote{ For more applications of generalized harmonic
    coordinates we refer the reader to
    \cite{johnson_linear_2018,huneau_stability_2018}}.
\end{remark}
\begin{remark}
  For any given Lorentzian metric $g$, finding the appropriate
  diffeomorphism $\phi$ such that $g^{\Upsilon}[g^0] := \phi^*g$
  satisfies the harmonic coordinate condition with respect to $g^0$
  reduces to solving a semilinear wave equation, and thus the problem
  of finding such a $\phi$ is locally well-posed.
\end{remark}

Throughout the remainder of the paper, it will be convenient to pick
$g^0$ to be the \KdS{} background metric around which we linearize.
\begin{definition}
  Define the \emph{linearized constraint}
  \begin{equation}
    \label{nonlinear:eq:linearized-wave-constraint-def}
    \Constraint_{g_b}h:=D_{g_b}\Constraint(g_b+h, g_b)(h) = -\nabla_{g_b}\cdot\TraceReversal_{g_b}h,
  \end{equation}
  where
  \begin{align*}
    \TraceReversal_{g}h &:= h - \frac{1}{2}(\Trace_g h)g, 
  \end{align*}
  denotes the \emph{trace reversal operator} of $g$.
  The \emph{linearized (wave coordinate) constraint condition} is
  \begin{equation}
    \label{nonlinear:eq:linearized-wave-constraint-eqn}
    \Constraint_{g_b}h=0.
  \end{equation}
\end{definition}

\begin{lemma}
  \label{nonlinear:lemma:harmonic-coordinates:EVE-quasilinear}
  Let $g$ be a Lorentzian metric satisfying the harmonic coordinate
  condition in \eqref{nonlinear:eq:GHC-condition:Christoffel} with respect to a
  fixed background metric $g^0$. Then, there exists some nonlinear
  function $\mathcal{Q}(\cdot, \cdot)$ such that
  \begin{equation}
    \label{nonlinear:eq:harmonic-coodinates:Ric-quasilinear}
    \Ric(g) - \Lambda g = -\frac{1}{2}\ScalarWaveOp[g]g + \mathcal{Q}(g, \p g),
  \end{equation}
  where we emphasize that $\mathcal{Q}$ is independent of the second
  derivatives of $g$. 
\end{lemma}
\begin{proof}
  If $g$ satisfies $\Constraint(g, g^0)=0$, then using
  \eqref{nonlinear:eq:EVE-coords}, we can write that
  \begin{equation*}
    \Ric(g)_{\mu\nu} = -\frac{1}{2}g^{\alpha\beta}\p_\alpha\p_\beta g_{\mu\nu}
    + \p_{(\mu}\left(g_{\nu)\chi}g^{\alpha\beta}\ChristoffelTypeTwo[g^0]{\chi}{\alpha\beta}\right) + \NCal(g,\p g).
  \end{equation*}
  Since $g^0$ is simply some prescribed background metric, we observe
  that the second term on the right hand side is itself now semilinear,
  and we can write in local coordinates
  \begin{equation}
    \label{nonlinear:eq:Ric-wave-coords}
    \Ric(g)_{\mu\nu} = -\frac{1}{2}g^{\alpha\beta}\p_\alpha\p_\beta g_{\mu\nu} + \mathcal{Q}(g,\p g),
  \end{equation}
  where
  \begin{equation}
    \label{nonlinear:eq:QCal-def}
    \mathcal{Q}(g, \p g):= \p_{(\mu}\left(g_{\nu)\chi}g^{\alpha\beta}\ChristoffelTypeTwo[g^0]{\chi}{\alpha\beta}\right) + \NCal(g,\p g). 
  \end{equation}
  From \eqref{nonlinear:eq:Ric-wave-coords}, it is clear that $\Ric$ is a
  quasilinear hyperbolic operator. Observing that the scalar wave
  operator for a metric $g$ satisfying the harmonic gauge condition
  \eqref{nonlinear:eq:GHC-condition:Christoffel} can be expressed as
  \begin{equation*}
    \ScalarWaveOp[g]:= \nabla^\alpha \p_\alpha = g^{\alpha\beta}\p_\alpha\p_\beta
    - g^{\alpha\beta}\ChristoffelTypeTwo[g^0]{\sigma}{\alpha\beta}\p_\sigma,
  \end{equation*}
  we can then conclude. 
\end{proof}

Crucial to the utility of harmonic coordinates is that
they are propagated by a hyperbolic operator.
\begin{definition}
  Given a smooth one-form $\psi$, we define the
  \emph{constraint propagation operator},
  \begin{equation*}
    \ConstraintPropagationOp_g\psi :=- 2\nabla_g\cdot\TraceReversal_g\nabla_g\otimes\psi = \VectorWaveOp[g]\psi - \Ric(g)(\psi, \cdot),
  \end{equation*}
  where
  $\nabla_g\otimes\psi = -\frac{1}{2}\LieDerivative_{\psi^\sharp}g$ is
  the \emph{symmetric gradient} of $g$, $\psi^\sharp$ is the canonical
  vectorfield associated to the one-form $\psi$, and
  $\VectorWaveOp[g]= \nabla^\alpha\nabla_\alpha$ denotes the wave
  operator acting on one-tensors. Observe that
  $\ConstraintPropagationOp_g$ is a manifestly hyperbolic operator.
\end{definition}

\begin{lemma}
  \label{nonlinear:lemma:EVE:nonlinear-constraint-prop}
  Any solution $g$ to
  \begin{equation*}
    \Ric(g) - \Lambda g - \nabla_g\otimes\Constraint(g,g^0) =0
  \end{equation*}
  must also satisfy
  \begin{align*}
    \ConstraintPropagationOp_g\Constraint(g,g^0)= 0.
  \end{align*}
\end{lemma}
\begin{proof}
  The conclusion follows directly by applying the twice contracted
  second Bianchi identity to the equation.
\end{proof}

The fact that the gauge constraints are propagated by a hyperbolic
operator allows us to conclude that solving EVE in harmonic gauge
is equivalent to solving the ungauged EVE provided that the initial
data satisfies the gauge constraint. 
\begin{prop}
  \label{nonlinear:prop:EVE:GHC-EVE-equivalence}
  Let $g$ be a solution to the
  Cauchy problem
  \begin{equation}
    \label{nonlinear:eq:EVE:gauged:aux}
    \begin{split}
      \frac{1}{2} g^{\alpha\beta}\p_\alpha\p_\beta g_{\mu\nu} + \Lambda g_{\mu\nu} ={}& \NCal_{g^0}(g, \p g)_{\mu\nu}\\
      \gamma_0(g)={}& (g_0, g_1),
    \end{split}
  \end{equation}
  where
  \begin{equation}
    \label{nonlinear:eq:gamma-def}
    \gamma_{\tStar}(g):= \left(g\vert_{\Sigma_{\tStar}}, \LieDerivative_{x^0}g\vert_{\Sigma_{\tStar}}\right).
  \end{equation}
  If $g$ satisfies
  \begin{equation*}
    \Constraint(g, g^0)\vert_{\Sigma_0} = 0;
  \end{equation*}
  and $(\InducedMetric, k)$, the induced metric and second
  fundamental form respectively by $g$ on $\Sigma_0$, satisfy the
  constraint conditions \eqref{nonlinear:eq:EVE:constraint-eqns}; then in fact
  \begin{equation*}
    \Constraint(g, g^0) = 0
  \end{equation*}
  uniformly, and $g$ is a solution to (\ref{nonlinear:eq:EVE:Full}) with
  admissible initial data posed by $(\Sigma_0, \InducedMetric, k)$.
\end{prop}

\begin{proof}
  This follows directly from the semilinear nature of the constraint
  propagation operator and uniqueness of solutions for semilinear
  hyperbolic PDEs.
\end{proof}

It will be convenient to rewrite \eqref{nonlinear:eq:EVE:gauged:aux} as a
quasilinear system in terms of the metric perturbation.

\begin{lemma}
  \label{nonlinear:lemma:EVE:GHC-quasilinear}
  Let $g = g_b + h$ be a solution to the Einstein vacuum equations in
  harmonic gauge with respect to a background metric $g^0$. If we
  choose $g^0 = g$, then linearizing around $g$, we find that ${h}$
  solves
  \begin{equation}
    \label{nonlinear:eq:EVE:quasilinear-system:decay}
    \LinEinstein_{g_b} h = \NCal_{g_b}(h, \p h, \p\p h),
  \end{equation}
  where $\NCal_{g_b}(h, \p h, \p\p h)$ is a quasilinear term in $h$, and
  \begin{equation}
    \label{nonlinear:eq:EVE:LinEinstein-def}
    \LinEinstein_{g_b}{h} = \TensorWaveOp[g_b]{h} + \mathcal{R}_{g_b}({h}), 
  \end{equation}
  where $\TensorWaveOp[g_b] = \nabla^\alpha\nabla_\alpha$ denotes the
  wave operator on $g_b$ acting on symmetric two-tensors.
  We will often refer to $\LinEinstein_{g_b}$ as the
  \emph{gauged linearized Einstein operator}. 

  In particular, $h$ also solves the system
  \begin{equation}
    \label{nonlinear:eq:EVE:quasilinear-system:energy}
    \LinEinsteinS_{g}h = \mathcal{Q}_{g}(h, \p h),
  \end{equation}
  where $\mathcal{Q}_{g}(h, \p h)$ is a semilinear term in $h$ at
  least quadratic in its arguments, and
  \begin{equation}
    \label{nonlinear:eq:EVE:quasilinear-LinEinstein-def}
    \LinEinsteinS_{g}h = \ScalarWaveOp[g]{h} + \SubPOp_{g}[h] + \PotentialOp_{g}[h], 
  \end{equation}
  where $\SubPOp_{g}$ is a vectorfield-valued matrix,
  $\PotentialOp_g$ is a function-valued matrix and both $\SubPOp_g$
  and $\PotentialOp_g$ have coefficients depending on at most one
  derivative of $g$, and
  \begin{equation*}
    \mathbf{L}_{g_b} = \LinEinstein_{g_b}.
  \end{equation*}
\end{lemma}
\begin{proof}
  The linearization in \eqref{nonlinear:eq:EVE:quasilinear-system:decay} follows
  from equation (2.4) in \cite{graham_einstein_1991}. Equation
  \eqref{nonlinear:eq:EVE:quasilinear-system:energy} then follows from Lemma
  \ref{nonlinear:lemma:harmonic-coordinates:EVE-quasilinear}, observing that the
  only quasilinear terms in
  \eqref{nonlinear:eq:harmonic-coodinates:Ric-quasilinear} are the quasilinear
  terms involved in the scalar wave.
\end{proof}

\subsection{Initial data}

In this section, we will construct the mapping $i_{b, \phi}$ between
admissible initial data triplets $(\Sigma_0, \InducedMetric_0, k_0)$
for the ungauged Einstein equations and the admissible initial data
$(g_0, g_1)$ for the Cauchy problem for the gauged Einstein
equations. An important property of this mapping is that a
metric $g$ such that
$(g\vert_{\Sigma_0}, \LieDerivative_{\KillT}g\vert_{\Sigma_0})=(g_0,
g_1)$ satisfies the gauge constraint
$\Constraint(g, g^0)\vert_{\Sigma_0} = 0$. To construct $i_{b, \phi}$
we need to specify a choice of $g^0$. As previously mentioned,
it will be convenient to choose $g^0 = g_b$.

Consider the \KdS{} initial data triplet
$(\Sigma_0, \InducedMetric_b, k_b)$ that launches $g_b$. That is, let
$\InducedMetric_b$ and $k_b$ denote the induced metric and second
fundamental form on $\Sigma_0$ by $g_b$.  We will construct
$i_{b, \phi}$ mapping $(\Sigma_0, \InducedMetric_b, k_b)$ into Cauchy
data for the gauged Einstein vacuum equations launching the \KdS{}
solution $\phi^*g_b$. That is, we will have that
\begin{equation*}
  i_{b, \phi}(\phi^*\InducedMetric_b, \phi^*k_b) = (0,0).
\end{equation*}
The linearization of this mapping will also produce the correctly
gauged initial data for the gauged linearized Einstein equation.

\begin{prop}
  \label{nonlinear:prop:initial-data:ib-construction}
  Fix a diffeomorphism $\phi$. Then   
  for each $b\in \BHParamNbhd$, there exists a map
  \begin{equation*}
    \begin{split}
      i_{b, \phi}:& H^{\BSTopLvl}(\Sigma_0; S^2T^*\Sigma_0)\times H^{\BSTopLvl-1}(\Sigma_0;S^2T^*\Sigma_0)
      \to H^{\BSTopLvl}(\Sigma_0; S^2T^*_{\Sigma_0}\StaticRegionWithExtension)\times H^{\BSTopLvl-1}(\Sigma_0;S^2T^*_{\Sigma_0}\StaticRegionWithExtension),
    \end{split}
  \end{equation*}
  that is smooth for $\BSTopLvl\ge 1$ depending smoothly on $b$, such that 
  \begin{enumerate}
  \item if $h$ is some symmetric two-tensor such that
    $i_{b, \phi}(\InducedMetric_0, k_0) = \gamma_0(h)$, then 
    \begin{equation*}
      (\InducedMetric, k) = (\InducedMetric_0, k_0),
    \end{equation*}
    where $(\InducedMetric, k)$ are the induced metric and the second
    fundamental form respectively of $\phi(\Sigma_0)$ induced by
    $g = \phi^*(g_b+h)$. Moreover, $g_b+h$ satisfies the gauge constraint
    \begin{equation*}
      \Constraint(g_b+h,g_b)\vert_{\Sigma_0} = 0;
    \end{equation*}
  \item if $(\InducedMetric_b, k_b)$ is the admissible initial data launching the \KdS{}
    metric $g_b$, then
    \begin{equation*}
      i_{b, \phi}(\phi^*\InducedMetric_b, \phi^*k_b)  = (0, 0);
    \end{equation*}
  \item if $\phi = e^{\vartheta}$, for some smooth vectorfield
    $\vartheta$, then $(g_0, g_1) = i_{b, \phi}(\InducedMetric, k)$
    satisfies the condition
    \begin{align*}
      \norm{(g_0, g_1)}_{\LSolHk{k}(\Sigma_0)}
      \lesssim{}&  \sum_{0\le |I|\le k}\norm*{\p_{x}^I(\InducedMetric - \InducedMetric_b)}_{L^2(\Sigma_0)}
      + \sum_{0\le |I|\le k+1}\norm*{\p_x^I\vartheta}_{L^2(\Sigma_0)}\\
      &+ \sum_{0\le |I|\le k-1}\norm*{\p_x^I (k - k_b)}_{L^2(\Sigma_0)},
    \end{align*}
    where $I$, $J$ are multi-indexes, and $k\le \BSTopLvl$.
  \end{enumerate}
\end{prop}
\begin{proof}
  See Proposition \ref*{linear:prop:initial-data:ib-construction}
  of \cite{fang_linear_2021}. 
\end{proof}

\section{Main theorem}
\label{nonlinear:sec:main-thm}
In this section, we present the convention regarding smallness
constants, the statement of the main theorem, as well
as a basic outline of the strategy.

\subsection{Smallness constants}

We introduce the following constants involved in the main theorem and
its ensuing proof.
\begin{itemize}
\item The black hole parameters $b^0= (M^0, a^0)$ represents the mass
  and angular momentum respectively of the initial \KdS{} spacetime
  which we perturb.
\item The integers $\BSTopLvl$ represents the maximum number of
  derivatives for our high regularity exponential growth bootstrap
  assumption of the solution, and $\BSLowLvl$ represents the maximum
  number of derivatives for the low-regularity exponential decay
  bootstrap assumption.
\item The smallness of the perturbation of the initial data norm from
  $g_{b^0}$ is measured by $\InitDataSize$.
\item $\BSDDecay$ and $\BSEGrowth$ measure the bootstrap rate of
  low-regularity exponential decay and high-regularity exponential
  growth respectively.
\item The size of the bootstrap norms are controlled by
  $\BSConstant$.
\end{itemize}

In what follows, $b^0$ is fixed, $\BSEGrowth, \BSDDecay$
will be chosen such that
\begin{equation*}  
  0\le \abs*{a^0}\ll \min{\BSDDecay,\BSEGrowth, M^0, \Lambda, 1}. 
\end{equation*}
We will choose $\BSEGrowth, \BSDDecay , \BSTopLvl, \BSLowLvl$, subject
to the condition that
\begin{equation*}
  2\BSEGrowth < (\BSTopLvl-\BSLowLvl-2)\BSDDecay, \qquad  \BSTopLvl-\BSLowLvl>2,\qquad \BSLowLvl>\frac{5}{2}. 
\end{equation*}
We will also choose $\InitDataSize, \BSConstant$ such that
\begin{equation}
  \label{nonlinear:eq:smallness}
  \begin{split}
    \InitDataSize, \BSConstant&\ll \min{ M^0, \Lambda,
      1},\\
    \InitDataSize, \BSConstant &\ll \abs*{a^0}, \quad \text{if
    }a^0\neq 0,
  \end{split} 
\end{equation}
and
\begin{equation*}
  \BSConstant = \InitDataSize^{\frac{2}{3}}.
\end{equation*}
\begin{remark}
  Observe that we may always assume \eqref{nonlinear:eq:smallness}, even if
  $0<|a^0|\lesssim \varepsilon_0$. In that case, the initial data
  continues to satisfy a smallness estimate of the form
  \begin{equation*}
    \norm*{(\InducedMetric_0, k_0)
      - (\InducedMetric_{b^0}, k_{b^0})}_{H^{\BSTopLvl}(\Sigma_0)\times H^{\BSTopLvl-1}(\Sigma_0)}
    \lesssim \varepsilon_0
  \end{equation*}
  by setting $b^0 = (M^0, 0)$. See Remark 3.4.1 in
  \cite{klainerman_kerr_2021} for a similar observation in the
  $\Lambda=0$ case.  
\end{remark}
In what follows, we will use $\lesssim$ to denote a quantity bounded
by a constant depending only on universal geometric constants as well
as the constants $M^0, a^0, \Lambda$ but not on $\InitDataSize$ or
$\BSConstant$.

\subsection{Statement of the main theorem}

We are now ready to state the main theorem. 
\begin{theorem}
  \label{nonlinear:thm:main}
  Fix slowly-rotating \KdS{} parameters $b^0 = (M^0,
  a^0)$. 
  Let $(\Sigma_0, \InducedMetric_0, k_0)$ be an admissible initial data
  triplet for EVE with cosmological constant $\Lambda>0$, such that
  \begin{equation}
    \label{nonlinear:eq:main:init-data-smallness}
    \norm*{(\InducedMetric_0, k_0)
      - (\InducedMetric_{b^0}, k_{b^0})}_{H^{\BSTopLvl}(\Sigma_0)\times H^{\BSTopLvl-1}(\Sigma_0)}
    \le \varepsilon_0
  \end{equation}
  where $(\InducedMetric_{b^0}, k_{b^0})$ is the initial data induced
  by $g_{b^0}$ on $\Sigma_0$, and $\BSTopLvl= 6$. Then there exist
  slowly-rotating \KdS{} black hole parameters
  $b_{\infty} = (M_\infty, a_\infty)\in B$, a diffeomorphism
  $\phi_\infty=e^{i_\Theta\vartheta_\infty}$, and a
  metric perturbation $h\in S^2T^*\StaticRegionWithExtension$ such
  that the 2-tensor $g$ satisfying
  \begin{equation*}
    g = \phi^*_\infty(g_{b_\infty} + h),
  \end{equation*}
  is a solution of the Einstein vacuum equations
  \begin{equation*}
    \Ric(g) - \Lambda g = 0, \qquad (\InducedMetric, k) = (\InducedMetric_0, k_0),
  \end{equation*}
  where $\InducedMetric, k$ are the induced metric and second
  fundamental form of $\Sigma_0\subset
  \StaticRegionWithExtension$ by $g$. 
  Moreover, $g_{b_\infty}+h$ is in
  harmonic gauge with respect to $g_{b_\infty}$, so that
  \begin{equation*}
    \Constraint(g_{b_\infty}+h, g_{b_\infty}) = 0
  \end{equation*}
  uniformly on $\StaticRegionWithExtension$.  Furthermore, there
  exists some $\BSDDecay>0$ such that $h$ satisfies the pointwise
  estimates
  \begin{equation*}
    \sup_{(\tStar,x)\in\StaticRegionWithExtension} e^{\BSDDecay \tStar}\abs*{h(\tStar, x)} \lesssim \varepsilon_0, 
  \end{equation*}
  where $\tStar$ takes values in $[0,+\infty)$ in
  $\StaticRegionWithExtension$.
\end{theorem}

\subsection{Strategy of proof}

The proof of Theorem \ref{nonlinear:thm:main} will be provided in Section
\ref{nonlinear:sec:bs}. We will employ a bootstrap scheme with three bootstrap
assumptions (see Section~\ref{nonlinear:sec:BA-assumptions:specific}):
\begin{enumerate}
\item A bootstrap assumption of integrated (small) exponential growth
  at the level of high regularity norms of the solution, which we call
  the energy bootstrap assumption.
\item A bootstrap assumption of exponential decay at the
  level of low regularity norms of the solution, which we call the
  decay bootstrap assumption.
\item A bootstrap assumption on the existence of a small gauge choice
  such that an orthogonality condition is satisfied which avoids
  certain linear obstacles to decay. 
\end{enumerate}
\begin{remark}
  These bootstrap assumptions are imposed on solutions of a semi-global
  extension of EVE with a gauge condition enforcing a certain
  orthogonality condition that avoids finitely many linear obstacles to
  decay (see Section \ref{nonlinear:sec:lin-theory} for a more in-depth
  discussion). The introduction of a semi-global extension of EVE is
  done so that we are allowed to use the linear theory developed in
  \cite{fang_linear_2021}, where certain global assumptions are needed.  
\end{remark}

To prove Theorem \ref{nonlinear:thm:main}, it then suffices to improve each of the
bootstrap assumptions, and show that the bootstrap time can be
extended.

\begin{enumerate}
\item We improve the low-regularity exponential decay bootstrap
  assumption, we directly apply the linear theory of
  \cite{fang_linear_2021} on a fixed \KdS{} background, recalled in Section
  \ref{nonlinear:sec:lin-theory:decay}, by treating all the
  nonlinear terms, including the quasilinear terms, as forcing terms
  on the right-hand side. This results in estimates which lose
  derivatives. This is handled by an interpolation argument, using the
  fact that the high-regularity bootstrap
  assumption assumes only very slow exponential growth. This is done
  in Proposition \ref{nonlinear:prop:BA-D:close}.
\item To improve the high-regularity integrated exponential growth
  bootstrap assumptions, we use a weak Morawetz estimate which does
  not lose derivatives. This estimate is stated in Proposition
  \ref{nonlinear:prop:BA-E:lossless-exp-grow-ILED}, and is a perturbation of the
  high-frequency Morawetz estimate proven in
  \cite{fang_linear_2021}. It in turn is proven in Section
  \ref{nonlinear:sec:energy-perturb-kds}.  The actual improvement of the
  high-regularity exponential growth bootstrap assumption is then
  carried out in Proposition \ref{nonlinear:prop:BA-E:close}. Here, we no longer
  treat the quasilinear terms as forcing terms since we can not lose
  derivatives in this step.
\item To improve the bootstrap assumption on the gauge, we
  use the implicit function theorem. This is done in Proposition 
  \ref{nonlinear:prop:BA-gauge-improve}. 
\item To extend the gauge, we use the continuity of the flow and
  again rely on the implicit function theorem. This is done in Proposition
  \ref{nonlinear:prop:BA-gauge-extend}. 
\end{enumerate}
The proof of Theorem \ref{nonlinear:thm:main} is then finished in
Section \ref{nonlinear:sec:closing-proof}.

\section{Linear Theory}
\label{nonlinear:sec:lin-theory}

In this section, we state the results from the linear theory needed for
the proof of Theorem \ref{nonlinear:thm:main}.

\subsection{Exponential decay on exact \KdS{} spacetimes}
\label{nonlinear:sec:lin-theory:decay}
We introduce some of the key conclusions of the linear theory in the
companion paper in \cite{fang_linear_2021}. 

One of the main conclusions of the linear theory was that solutions to
the linearized Einstein equations decay exponentially, up to a finite
number of obstacles. These finite obstacles are
$\QNMk{k}(\LinEinstein_{g_b}, \UpperHalfSpace)$, the non-decaying
$\LSolHk{k}$-\emph{quasinormal modes} of $\LinEinstein_{g_b}$.
\begin{corollary}[See Corollary \ref{linear:coro:lambda-IVP} in
  \cite{fang_linear_2021}]
  \label{nonlinear:coro:lambda-map}
  Let $k\ge 3$, and let
  $\ZCal\subset D^{k,\SpectralGap}(\StaticRegionWithExtension)$ be a
  finite-dimensional linear subspace such that the map
  \begin{equation}
    \label{nonlinear:eq:lambda-map:z-cal-map}
    \ZCal\ni (\tilde{f}, \tilde{h}_0, \tilde{h}_1)\mapsto
    \bangle*{\frac{1}{\ImagUnit}
      \begin{pmatrix}
        \delta_0 \tilde{h}_0\\
        -\GInvdtdt_b\tilde{f} + \delta_0 \tilde{h}_1
      \end{pmatrix},
      \cdot
    }\in \mathcal{L}(\QNMk{k*}(\LinEinstein_{g_b}, \UpperHalfSpace), \overline{\Complex}),
  \end{equation}
  mapping from $\ZCal$ to the space of linear functionals on the dual
  $\LSolHk{k}$-quasinormal modes is bijective.
  Then there exists a 
  continuous linear map
  \begin{equation}
    \label{nonlinear:eq:lambda-def}
      \lambda_{\ZCal}: D^{k,\SpectralGap,1}(\StaticRegionWithExtension)
       \to \ZCal
  \end{equation}
  such that if
  $D^{k,\SpectralGap, 1}\ni\lambda_{\ZCal}(f, h_0, h_1) = z = (\tilde{f}, \tilde{{h}}_0,
  \tilde{{h}}_1) \in \ZCal$, then the initial
  value problem
  \begin{equation*}
    \begin{split}
      \LinearOp {h} &= f + \tilde{f},\\
      \gamma_0({h}) &= ({h}_0 + \tilde{{h}}_0, {h}_1 + \tilde{{h}}_1)
    \end{split}
  \end{equation*}
  has an exponentially decaying solution ${h}$ that satisfies the
  estimate
  \begin{equation}
    \label{nonlinear:eq:lambda-IVP-solution-inequality}
    \norm{h}_{\HkWithT{k}(\Sigma_{\tStar})}
    \lesssim e^{-\SpectralGap \tStar} \norm{(f, {h}_0, {h}_1)}_{D^{k+1,\SpectralGap}(\StaticRegionWithExtension)}. 
  \end{equation}
  If moreover,
  $\lambda_{\ZCal}$ is bijective, then $z$ is unique and the map $(f,
  {h}_0, {h}_1) \to z$ is continuous.
\end{corollary}

The following proposition corresponds to
Corollary \ref*{linear:coro:lambda-IVP} of \cite{fang_linear_2021}
applied to the case where\footnote{In the context of
  \cite{fang_linear_2021}, Corollary \ref*{linear:coro:lambda-IVP} is
  shown for a strongly hyperbolic operator $\LinEinstein$ with a
  discrete quasinormal spectrum satisfying certain high frequency
  resolvent estimates. That $\LinEinstein_{g_b}$ satisfies these
  conditions is shown in Lemma \ref*{linear:lemma:EVE:GHC-quasilinear},
  Theorems \ref*{linear:thm:meromorphic:main-A} and
  \ref*{linear:thm:resolvent-estimate:inf-gen} in
  \cite{fang_linear_2021}.} $\LinearOp=\LinEinstein_{g_b}$ (see also
Proposition 5.7 in \cite{hintz_global_2018}).
\begin{prop}
  \label{nonlinear:prop:orthogonality-condition}
  Fix $g_b$ a slowly-rotating \KdS{} black hole. Then there exist some
  $\SpectralGap>0$ and linear maps
  \begin{equation*}
    \lambda_{\Upsilon}[g_b]: D^{k,\SpectralGap, 1}(\StaticRegionWithExtension)\to \Real^{N_\Upsilon},\qquad
    \lambda_{\Constraint}[g_b]: D^{k,\SpectralGap, 1}(\StaticRegionWithExtension)\to \Real^{N_\Constraint},
  \end{equation*}
  such that for
  $(f, h_0, h_1)\in D^{k,\SpectralGap,
    1}(\StaticRegionWithExtension)$, where $k\ge 3$,
  if
  $\lambda_{\Upsilon}[g_b](f, h_0, h_1) =
  \lambda_{\Constraint}[g_b](f, h_0, h_1) =0$, then for any
  $0<\BSDDecay<\SpectralGap$, a solution $h$ to the Cauchy problem
  \begin{equation}
    \label{nonlinear:eq:lin-theory:Cauchy-problem}
    \begin{split}
      \LinEinstein_{g_b}h &= f,\\
      \gamma_0(h) &= (h_0, h_1),
    \end{split}    
  \end{equation}
  where $\LinEinstein_{g_b}$ is as defined in \eqref{nonlinear:eq:EVE:LinEinstein-def},
  satisfies the decay estimate
  \begin{equation}
    \label{nonlinear:eq:linear-decay-statement}
    \sup_{\tStar\ge 0}e^{\BSDDecay\tStar}\norm{h}_{\HkWithT{k}(\Sigma_{\tStar})} \lesssim \norm*{(f, h_0, h_1)}_{D^{k, \BSDDecay,1}(\StaticRegionWithExtension)}.
  \end{equation}
\end{prop}
\begin{remark}
  \label{nonlinear:rmk:threshold-regularity}
  The appearance of the requirement that $(f,h_0,h_1)\in
  D^{k,\SpectralGap,1}(\StaticRegionWithExtension)$ for $k\ge 3$ has
  to do with the threshold regularity level in
  \cite{fang_linear_2021}. For slowly-rotating \KdS{} black holes, it
  is sufficient to take $k>2$. Since we are only working with integer
  Sobolev spaces, and since we will in any case always work with $(f,h_0,h_1)\in
  D^{3,\SpectralGap,1}(\StaticRegionWithExtension)$ in
  this paper, we make no further attempts to optimize this. 
\end{remark}
The finite number of obstacles to exponential decay of the Cauchy
problem defined by $(f, h_0, h_1)$ for the linearized Einstein
operator are characterized by the condition in Proposition
\ref{nonlinear:prop:orthogonality-condition} that
$(f, h_0,h_1) \in \ker \lambda_{\Upsilon}[g_b]\bigcap
\ker\lambda_{\Constraint}[g_b]$. \textit{A priori}, these finite obstacles
could actually grow exponentially, and thus, for general solutions to
the Cauchy problem generated by $(f, h_0, h_1)$ in
\eqref{nonlinear:eq:lin-theory:Cauchy-problem}, it is not possible to enforce
exponential decay. What makes the result of Proposition
\ref{nonlinear:prop:orthogonality-condition} useful then is that in considering
Einstein's equations, we do not consider a general Cauchy problem but
one satisfying certain constraints. These constraints manifest
themselves in the nature of the mappings $\lambda_{\Upsilon}$ and
$\lambda_{\Constraint}$. 

In particular, the following proposition states that the kernel of
$\lambda_{\Constraint}[g_b]$ consists only of solutions that do not
satisfy the linearized wave coordinate constraint condition in
\eqref{nonlinear:eq:linearized-wave-constraint-eqn}. This is the content of
Proposition~\ref*{linear:prop:KdS-QNM-perturb} in \cite{fang_linear_2021}.
\begin{prop}
  \label{nonlinear:prop:lambda-kernel}
  Fix $g_b$ a slowly-rotating \KdS{} metric, and
  let $h$ be the solution to the Cauchy problem in
  \eqref{nonlinear:eq:lin-theory:Cauchy-problem}, for some $(f, h_0, h_1)\in
  D^{k,\SpectralGap,1}(\StaticRegionWithExtension), k>0$.
  
  If there exists some $\TStar$ such that
  \begin{equation*}
    f(\tStar, \cdot) = \Constraint_{g_b}(h)(\tStar, \cdot) = 0, \quad \tStar>\TStar,
  \end{equation*}
  where $\Constraint_{g_b}(h)$ is as defined in
  \eqref{nonlinear:eq:linearized-wave-constraint-def}, then
  \begin{equation*}
    (f, h_0, h_1)\in \ker(\lambda_{\Constraint}[g_b]).
  \end{equation*}
\end{prop}
Proposition \ref{nonlinear:prop:lambda-kernel} implies that some of the finite
number of linear obstacles to decay present in Proposition
\ref{nonlinear:prop:orthogonality-condition} are solutions to the linearized
Einstein equations in harmonic gauge which themselves do not satisfy
the linearized harmonic gauge condition. These solutions are
unphysical, and are avoided in when considering the linearized
Einstein equations in linearized harmonic gauge with data respecting
the linearized gauge constraint.

To handle the remaining obstacles to decay, we have the following
proposition, which amounts to a mode stability statement. This is the
content of Proposition \ref*{linear:prop:linear-stab:KdS-robust} in
\cite{fang_linear_2021}.
\begin{prop}
  \label{nonlinear:prop:lambda-isomorphism}
  
  There exists an $N_{\Theta}$-dimensional family of non-exponentially-decaying
  vectorfields $\Theta$, parametrized by
  \begin{equation}
    \label{nonlinear:eq:i-Theta:def}
    i_{\Theta}:\Real^{N_{\Theta}}\to \Theta
  \end{equation}
  such that for any $i_{\Theta}\vartheta$,
  $\vartheta\in\Real^{N_{\Theta}}$, for any slowly-rotating \KdS{}
  metric $g_b$, 
  \begin{equation*}
    \LinEinstein_{g_b}\left(\nabla_{g_b}\otimes i_{\Theta}\vartheta\right)=0,
  \end{equation*}
  so that $\nabla_{g_b}\otimes i_{\Theta}\vartheta$ is a non-decaying mode
  solution of $\LinEinstein_{g_b}$. 
  
  Moreover, with $\lambda_{\Upsilon}[g_b]$ as in Proposition
  \ref{nonlinear:prop:orthogonality-condition}, $\lambda_{\Upsilon}$ is an
  isomorphism
  \begin{equation*}
    \lambda_{\Upsilon}[g_b]:\curlyBrace*{
      \left(0, \gamma_0\left((g_b')^\Upsilon(b') + \nabla_{g_b}\otimes i_{\Theta}\vartheta'\right)\right):(
      b', \vartheta')\in \Real^4\times \Real^{N_{\Theta}}
    } \to \Real^{N_{\Upsilon}},
  \end{equation*}
  where $N_{\Upsilon} = 4+N_{\Theta}$, and
  \begin{equation}
    \label{nonlinear:eq:g-linearized-wave-gauge:def}
    (g_{b}')^\Upsilon(b') := \frac{\p g_{b}}{\p b}b' + \nabla_{g_{b}}\otimes\omega_{b}^\Upsilon(b'),
  \end{equation}
  and $\omega_{b}^\Upsilon(b')$ denotes the
  solution of the Cauchy problem
  \begin{equation}
    \label{nonlinear:eq:linearized-wave-gauge-correction:def}
    \begin{split}
      \begin{cases}
        \Constraint_{g_{b}}\circ\nabla_{g_{b}}\otimes\omega_{b}^\Upsilon(b')
        = -\Constraint_{g_{b}}(g_{b}'(b')) & \text{in }\mathcal{M},\\
        \InitData(\omega_{b}^\Upsilon(b'))=(0,0)&\text{on }\Sigma_0.
      \end{cases} 
    \end{split}
  \end{equation}
\end{prop}
Observe that
$(g_b')^\Upsilon(b') + \nabla_{g_b}\otimes i_{\Theta}\vartheta'$ are
infinitesimal diffeomorphisms and infinitesimal changes in black hole
parameters, and are solutions satisfying the linearized harmonic gauge
condition. These are precisely the solutions to the linearized
Einstein equations in harmonic gauge which correspond to general
covariance and changes in the black hole parameters $b$. Thus,
Proposition \ref{nonlinear:prop:lambda-kernel} and Proposition
\ref{nonlinear:prop:lambda-isomorphism} together state that the finite linear
obstacles to exponential decay in Proposition
\ref{nonlinear:prop:orthogonality-condition} are unphysical in that they either
\begin{enumerate}
\item do not satisfy the linearized harmonic gauge constraint, or
\item are an infinitesimal diffeomorphism or an infinitesimal change
  of the black hole parameters. 
\end{enumerate}
It will be convenient to observe that $\lambda_{\Upsilon}$ and
$\lambda_{\Constraint}$ are invariant under time translation.
\begin{prop}
  \label{nonlinear:prop:lambda-time-translation}
  Fix $g_b$ a slowly-rotating \KdS{} metric, and
  let $h$ be the solution to the Cauchy problem in
  \eqref{nonlinear:eq:lin-theory:Cauchy-problem}, for some $(f, h_0, h_1)\in
  D^{k,\SpectralGap,1}(\StaticRegionWithExtension), k>0$.
  Then in fact
  \begin{equation}
    \label{nonlinear:eq:lambda-time-translation}
    \begin{split}
      \lambda_{\Upsilon}[g_b]\left(
      f(\tStar+\TStar), h(\TStar), \p_{\tStar}h(\TStar)
    \right)
    &= \lambda_{\Upsilon}[g_b]\left(
      f(\tStar), h_0, h_1
    \right),\\
      \lambda_{\Constraint}[g_b]\left(
      f(\tStar+\TStar), h(\TStar), \p_{\tStar}h(\TStar)
    \right)
    &= \lambda_{\Constraint}[g_b]\left(
      f(\tStar), h_0, h_1
    \right).
    \end{split}    
  \end{equation}
\end{prop}

Finally, we will need the following lemma controlling the dependence
of $\lambda_{\Upsilon}[g_b]$ on the black hole parameters $b$. This is
the content of Corollary 10.10 in \cite{fang_linear_2021}.
\begin{lemma}
  \label{nonlinear:lemma:lambda-b-derivative}
  Let $g_b$ be a slowly-rotating \KdS{} metric. If
  $f\in L^2(\StaticRegionWithExtension)$ is supported on
  $\DomainOfIntegration = [0, \TStar]\times\Sigma$, then
  \begin{equation*}
    \abs*{D_{b}\lambda_{\Upsilon}[g_b](f, h_0, h_1)}
    \lesssim \norm*{f}_{L^2(\DomainOfIntegration)} + \norm*{(h_0, h_1)}_{\LSolHk{1}(\Sigma_0)}. 
  \end{equation*}
\end{lemma}

\subsection{Exponential growth on perturbation of \KdS}
\label{nonlinear:sec:lin-theory:energy}

In this section we state the main energy estimate that we need on a
perturbation of \KdS. The proof of the following proposition is
postponed to Section \ref{nonlinear:sec:proof-of-perturbed-ILED}.
\begin{prop}
  \label{nonlinear:prop:BA-E:lossless-exp-grow-ILED}
  Fix $\BSEGrowth>0$ and let $\DomainOfIntegration = [0,
  \TStar]\times\Sigma$.  Then fix $g=g_b+\tilde{g}$, where $g_b$ is a
  slowly-rotating \KdS{} metric, and
  \begin{equation*}
    \sup_{\tStar>0}e^{-\BSEGrowth\tStar}\norm{\tilde{g}}_{\HkWithT{3}(\Sigma_{\tStar})}\le
  \BSConstant
  \end{equation*}
  and let $h$ solve the Cauchy problem
  \begin{equation*}
    \begin{split}
      \LinEinsteinS_{g}h &= f,\\
      \gamma_0(h) &= (h_0, h_1),
    \end{split}
  \end{equation*}
  where $\LinEinsteinS_{g}$ is as in
  \eqref{nonlinear:eq:EVE:quasilinear-LinEinstein-def} and\footnote{The
    requirement that $k\ge 3$ again has to do with the threshold
    regularity level in \cite{fang_linear_2021}, see Remark
    \ref{nonlinear:rmk:threshold-regularity}.}
  $(f, h_0, h_1)\in D^{k,-\BSEGrowth,0}(\StaticRegionWithExtension),
  k\ge 3$.  Then for $\BSConstant$ sufficiently small, $h$ satisfies
  the estimate
  \begin{equation}
    \label{nonlinear:eq:BA-E:lossless-exp-grow-ILED}
    \norm*{e^{-\BSEGrowth \tStar}h}_{H^1(\DomainOfIntegration)}
    \lesssim \norm{(h_0, h_1)}_{\LSolHk{1}(\Sigma_0)}
    + \norm*{e^{-\BSEGrowth\tStar}h}_{\HkWithT{1}(\Sigma_{\TStar})}
    + \int_0^{\TStar}e^{-\BSEGrowth \tStar}\left(
      \norm{f}_{\InducedLTwo(\Sigma_{\tStar})} + \norm{h}_{\InducedLTwo(\Sigma_{\tStar})}\right)\,d\tStar.
  \end{equation}
\end{prop}
We make a few remarks regarding Proposition
\ref{nonlinear:prop:BA-E:lossless-exp-grow-ILED}. Proposition
\ref{nonlinear:prop:BA-E:lossless-exp-grow-ILED} is a perturbation of a weak
Morawetz estimate on exact \KdS{} spacetimes with two key properties:
\begin{enumerate}
\item first, the estimate in \eqref{nonlinear:eq:BA-E:lossless-exp-grow-ILED}
  does not lose derivatives. This is critical to improve the
  bootstrap assumptions at the highest level of regularity;
\item second, the estimate in \eqref{nonlinear:eq:BA-E:lossless-exp-grow-ILED}
  is proven for arbitrarily small $\BSEGrowth>0$. To improve the
  low-regularity exponential-decaying bootstrap assumption, we use an
  interpolation argument. Having
  \eqref{nonlinear:eq:BA-E:lossless-exp-grow-ILED} for arbitrarily small
  $\BSEGrowth>0$ then allows us to drastically reduce the number of
  derivatives needed for the aforementioned interpolation argument. 
\end{enumerate}

\section{Bootstrap assumptions}
\label{nonlinear:sec:BA-Assume}

In this section, we set up the bootstrap assumptions that we will
use to prove Theorem \ref{nonlinear:thm:main}.

\subsection{The semi-global extension}

Instead of making bootstrap assumptions on solutions of EVE up to a
finite time, we make bootstrap assumptions on a semi-global extension
of EVE.

We first introduce an auxiliary proposition that generates admissible
gauged initial data for the linearized gauged Einstein equations from
a solution to the nonlinear Einstein equations in harmonic gauge that
is quadratically close to the nonlinear solution.

\begin{prop}
  \label{nonlinear:prop:Bootstrap:iota-construction:init-data}
  Fix $\TStar > 0$. Then there exists a map
  \begin{equation*}
    \begin{split}
      \iota_{b, \TStar}: \LSolHk{k}(\Sigma_{\TStar})&\to
      \LSolHk{k}(\Sigma_{\TStar})\\
      (h_0, h_1)&\mapsto (\widetilde{h}_0, \widetilde{h}_1),
    \end{split}
  \end{equation*}
  satisfying the following properties. 
  \begin{enumerate}
  \item If $\widetilde{h}\in S^2T^*\StaticRegionWithExtension$ is a two-tensor
    inducing $\iota_{b, \TStar}(h_0, h_1)$ on
    $\Sigma_{\TStar}$ in the sense that
    \begin{equation*}
      \left.\left(\widetilde{h}(\TStar, \cdot), \LieDerivative_{\KillT}\widetilde{h}(\TStar, \cdot)\right)\right\vert_{\Sigma_{\TStar}} = \iota_{b, \TStar}(h_0, h_1),
    \end{equation*}
    then $\widetilde{h}$ satisfies the
    linearized harmonic gauge constraint on
    $\Sigma_{\TStar}$ given by
    \begin{equation*}
      \left.\Constraint_{g_b}(\widetilde{h})\right\vert_{\Sigma_{\TStar}} = 0.
    \end{equation*}
  \item If
    $\left(h|_{\Sigma_{\TStar}},=
      \LieDerivative_{\KillT}h|_{\Sigma_{\TStar}}\right)$ satisfies
    the nonlinear harmonic gauge constraint on $\Sigma_{\TStar}$ with
    respect to $g_b$, and
    \begin{equation*}
      \left(
        \widetilde{h}|_{\Sigma_{\TStar}}, \LieDerivative_{\KillT}\widetilde{h}|_{\Sigma_{\TStar}}\right)
      =\iota_{b,\TStar}\left(
        h|_{\Sigma_{\TStar}}, \LieDerivative_{\KillT}h|_{\Sigma_{\TStar}}\right),
    \end{equation*}
    then there exist $q_{\TStar}(\cdot, \cdot)$, $\tilde{q}_{\TStar}(\cdot, \cdot)$ such that
    \begin{equation}
      \label{nonlinear:eq:underline-q-def}
      \left.\left(h - \widetilde{h}\right)\right\vert_{\Sigma_{\TStar}} = 0,\qquad
      \left.\p_{\tStar}\left(h - \widetilde{h}\right)\right\vert_{\Sigma_{\TStar}} = \underline{q}_{\TStar}(h\vert_{\Sigma_{\TStar}}, \p h\vert_{\Sigma_{\TStar}}) + \tilde{\underline{q}}_{\TStar}(h\vert_{\Sigma_{\TStar}}, \p h\vert_{\Sigma_{\TStar}}),
    \end{equation}
    where $\underline{q}_{\TStar}$ is quadratic in its arguments,
    and $\tilde{\underline{q}}_{\TStar}$ is at least cubic in its
    arguments. 
  \end{enumerate}
\end{prop}

\begin{proof}
  Let $h\in S^2T^*\StaticRegionWithExtension$ be a symmetric
  two-tensor inducing $(h_0, h_1)$ on $\Sigma_{\TStar}$. Then, let
  $({\InducedMetric}_{\TStar}, \InducedMetric_{\TStar}')$ denote the
  induced metric by $g_{b} + h, \LieDerivative_{\KillT}(g_{b} + h)$
  respectively on $\Sigma_{\TStar}$. Now we will construct some
  $\tilde{g} = g_b + \widetilde{h}$ such that
  $\left.\left(\widetilde{h},
      \p_{\tStar}\widetilde{h}\right)\right\vert_{\Sigma_{\TStar}} =
  \iota_{b,\TStar}(h_0, h_1)$ explicitly satisfying the conditions in
  the lemma.

  First, we define $\widetilde{h} = \InducedMetric_{\TStar} - g_b$, and
  $\widetilde{h}'_{ij} = (\InducedMetric'_{\TStar})_{ij}$. It remains to
  define $\widetilde{h}'_{\tStar \mu}$ using the linearized gauge
  condition.  To this end, we recall the gauge condition on
  $\Sigma_{\TStar}$ linearized around $g_b$,
  \begin{equation}
    \label{nonlinear:eq:linearized-gauge-condition}
    \begin{split}
      0 ={}&- \frac{1}{2}\widetilde{h}^{\alpha\beta}\p_\mu (g_b)_{\alpha\beta}
      - \frac{1}{2}(g_b)^{\alpha\beta}\p_\mu \widetilde{h}_{\alpha\beta}
      + (g_b)^{\alpha\beta}\p_{\alpha}\widetilde{h}_{\mu\beta}
      + \widetilde{h}^{\alpha\beta}\p_{\alpha}(g_b)_{\mu\beta}\\
      &- \widetilde{h}_{\mu\gamma}(g_b)^{\alpha\beta}\ChristoffelTypeTwo[g_b]{\gamma}{{\alpha\beta}}
      - (g_b)_{\mu\gamma}\widetilde{h}^{\alpha\beta}\ChristoffelTypeTwo[g_b]{\gamma}{{\alpha\beta}}.
    \end{split}
  \end{equation}
  Contracting \eqref{nonlinear:eq:linearized-gauge-condition} with $\KillT$, we
  have that
  \begin{equation}
    \label{nonlinear:eq:linearized-gauge:pt-gtt}
    \begin{split}
      \frac{1}{2}(g_b)^{\tStar\tStar}\p_{\tStar}\widetilde{h}_{\tStar\tStar}
      ={}& - (g_b)^{i\beta}\p_{i}\widetilde{h}_{\tStar \beta}
      + \frac{1}{2}\widetilde{h}^{\alpha\beta}\p_{\tStar}(g_b)_{\alpha\beta}
      - \widetilde{h}^{\alpha\beta}\p_{\alpha}(g_b)_{\tStar\beta}\\
      & + \widetilde{h}_{\tStar \gamma}(g_b)^{\alpha\beta}\ChristoffelTypeTwo[g_b]{\gamma}{\alpha\beta}
      + (g_b)_{\tStar \gamma}\widetilde{h}^{\alpha\beta}\ChristoffelTypeTwo[g_b]{\gamma}{\alpha\beta}.
    \end{split}    
  \end{equation}
  As we already have
  $\left.\widetilde{h}\right\vert_{\Sigma_{\TStar}},
  \left.\p_i\widetilde{h}\right\vert_{\Sigma_{\TStar}}$, and
  $\left.\p_{\tStar}\widetilde{h}_{ij}\right\vert_{\Sigma_{\TStar}}$,
  solving equation
  \eqref{nonlinear:eq:linearized-gauge:pt-gtt} uniquely determines
  $\widetilde{h}'_{\tStar \tStar} = \left.\p_{\tStar}\widetilde{h}_{\tStar\tStar}\right\vert_{\Sigma_{\TStar}}$.

  Now consider contracting \eqref{nonlinear:eq:linearized-gauge-condition} with
  $\p_i$. In this case we have that
  \begin{equation}
    \label{nonlinear:eq:linearized-gauge:pi-gtt}
    \begin{split}      
      (g_b)^{\tStar\tStar}\p_{\tStar}\widetilde{h}_{\tStar i}
      ={}& -(g_b)^{\beta j}\p_j \widetilde{h}_{i\beta}
      + \frac{1}{2}(g_b)^{\alpha\beta}\p_i\widetilde{h}_{\alpha\beta}
      + \frac{1}{2}\widetilde{h}^{\alpha\beta}\p_i (g_b)_{\alpha\beta}
      - \widetilde{h}^{\alpha\beta}\p_\alpha (g_b)_{i\beta}\\
      &+ \widetilde{h}_{i \gamma}(g_b)^{\alpha\beta}\ChristoffelTypeTwo[g_b]{\gamma}{\alpha\beta}
      + (g_b)_{i \gamma}\widetilde{h}^{\alpha\beta}\ChristoffelTypeTwo[g_b]{\gamma}{\alpha\beta}.
    \end{split}
  \end{equation}  
  As before, \eqref{nonlinear:eq:linearized-gauge:pi-gtt} uniquely determines
  $\widetilde{h}'_{\tStar i} = \left.\p_{\tStar}\widetilde{h}_{\tStar i}\right\vert_{\Sigma_{\TStar}}$.

  By construction, we have defined $(\widetilde{h}, \p_{\tStar}\widetilde{h})$
  satisfying the linearized gauge constraint on
  $\Sigma_{\TStar}$. Moreover, by construction
  \begin{equation*}
    \left.(h-\widetilde{h})\right\vert_{\Sigma_{\TStar}} = 0, \qquad \left.\p_{\tStar}(h_{ij}-\widetilde{h}_{ij})\right\vert_{\Sigma_{\TStar}}=0. 
  \end{equation*}
  If $g_b+h$ solves the nonlinear gauge constraint on
  $\Sigma_{\TStar}$, and $g_b + \widetilde{h}$ solves the
  linearized gauge constraint on $\Sigma_{\TStar}$, then it is clear that
  $\left.\p_{\tStar}(h-\widetilde{h})\right\vert_{\Sigma_{\TStar}}$ can be
  decomposed into a nonlinearity
  \begin{equation*}
    \left.\p_{\tStar}(h-\widetilde{h})\right\vert_{\Sigma_{\TStar}} = \underline{q}_{\TStar}(h, \p h) + \tilde{\underline{q}}_{\TStar}(h, \p h),
  \end{equation*}
  consisting of a quadratic nonlinearity and a higher-order
  nonlinearity respectively as desired.   
\end{proof}

We can now explicitly construct the desired semi-global extension of
the EVE. First, let $\chi(\tStar)$ be a smooth, compactly supported
function such that
\begin{equation}
  \label{nonlinear:eq:chi-tStar-def}
  \chi(\tStar)=
  \begin{cases}
    1 & \tStar\le 0\\
    0 & \tStar\ge 1
  \end{cases},\qquad \chi_{\TStar}(\tStar) = \chi(\tStar-\TStar). 
\end{equation}

\begin{definition}
  \label{nonlinear:def:h-def}
  From this point forward in the paper, we denote
  \begin{equation}
    \label{nonlinear:eq:h:def}
    h_{\TStar} := \chi_{\TStar}h + (1-\chi_{\TStar})\widetilde{h}_{\TStar},
  \end{equation}
  where $\phi^*(g_b+h)$ is a solution to EVE with initial data
  $(\InducedMetric, k)$, and where $\widetilde{h}_{\TStar}$ is the solution to
  \begin{equation}
    \label{nonlinear:eq:tilde-h:Cauchy-problem}
    \begin{split}
      \LinEinstein_{g_b}\widetilde{h}_{\TStar} &=0\\
      \gamma_{\TStar}(\widetilde{h}_{\TStar}) &= \iota_{b,\TStar}(h|_{\Sigma_{\TStar}}, \p_{\tStar}h|_{\Sigma_{\TStar}}).
    \end{split}    
  \end{equation}
  To emphasize their dependence on the choice of parameters $(b,
  \vartheta, \InducedMetric, k, \TStar)$, we will write
  \begin{equation}
    \label{nonlinear:eq:h-def:with-params}
    \begin{split}
      h_{\TStar} &= h_{\TStar}(b,\vartheta,\InducedMetric, k),\\
      h &= h(b,\vartheta,\InducedMetric, k),\\
      \widetilde{h}_{\TStar} &= \widetilde{h}_{\TStar}(b,\vartheta,\InducedMetric, k).\\
    \end{split}    
  \end{equation}
\end{definition}

\begin{definition}
  \label{nonlinear:def:NTilde:def}
  Fix some $\TStar>0$. Then we define
  \begin{equation}
    \label{nonlinear:eq:NTilde:def}
    \NCal_{g_b}^{\TStar}(h_{\TStar}, \p h_{\TStar}, \p\p h_{\TStar})
    := \chi_{\TStar} \NCal_{g_b}(h,\p h, \p\p h)
    + \left[\LinEinstein_{g_b},\chi_{\TStar}\right](h-\widetilde{h}_{\TStar}), 
  \end{equation}
  where we recall $\widetilde{h}_{\TStar}$ is the solution to the Cauchy problem in
  \eqref{nonlinear:eq:tilde-h:Cauchy-problem}, $\LinEinstein_{g_b}$ is as
  defined in \eqref{nonlinear:eq:EVE:LinEinstein-def}, and $\NCal_{g_b}$ is as
  defined in \eqref{nonlinear:eq:EVE:quasilinear-system:decay}.
\end{definition}

We can now define the relevant semi-global extension of EVE. 
\begin{prop}
  \label{nonlinear:prop:bs-system}
  Fix $\TStar>0$, some diffeomorphism $\phi=e^{i_{\Theta}\vartheta}$ where $\vartheta\in \Real^{N_{\Theta}}$, a
  slowly rotating \KdS{} metric $g_b$, and
  $(\Sigma_0, \InducedMetric_0, k_0 )$ an admissible initial data
  triplet. Then, for $\phi^*(g_b + h)$ a solution to EVE launched by
  the initial data $(\Sigma_0, \InducedMetric_0, k_0)$, $h_{\TStar}$
  as defined in Definition \ref{nonlinear:def:h-def} solves the
  \emph{semi-global extension of EVE} generated by
  $(b, \vartheta, \TStar)$, given by
  \begin{equation}
    \label{nonlinear:eq:EVE:bootstrap-system}
    \begin{split}
      \LinEinstein_{g_b}h_{\TStar} &= {\NCal}^{\TStar}_{g_b}(h_{\TStar}, \p h_{\TStar}, \p\p h_{\TStar}),\\
      \gamma_{0}(h_{\TStar}) &= i_{b,\phi}(\InducedMetric_0, k_0),
    \end{split}    
  \end{equation}
  where $\NCal_{g_b}^{\TStar}$ is as given in \eqref{nonlinear:eq:NTilde:def}.
  Moreover,
  \begin{enumerate}
  \item $h_{\TStar} = h$ for $\tStar\le \TStar$,
  \item $\Constraint_{g_b}(h_{\TStar}) = 0$ for $\tStar\ge\TStar+1$.
  \end{enumerate}
\end{prop}
\begin{proof}
  The proof is straightforward from the definitions of
  $h_{\TStar}$ and $N_{g_b}^{\TStar}$ in Definitions
  \ref{nonlinear:def:h-def} and \ref{nonlinear:def:NTilde:def} respectively.
\end{proof}
To improve the bootstrap assumptions, we need certain control on
$\NCal_{g_b}^{\TStar}$.
\begin{prop}
  \label{nonlinear:prop:N-control}
  Fix some $\TStar>0$. Then, there exist nonlinear functions $q(\cdot,
  \cdot, \cdot), \tilde{q}(\cdot, \cdot, \cdot), q_{\TStar}(\cdot,
  \cdot)$, and $\tilde{q}_{\TStar}(\cdot, \cdot)$ such
  that for $\tStar\le \TStar$,
  \begin{equation}
    \label{nonlinear:eq:NCal:energy-bound}
    {\NCal}^{\TStar}_{g_b}(h_{\TStar}, \p h_{\TStar}, \p\p h_{\TStar})  = q(h_{\TStar}, \p h_{\TStar}, \p\p h_{\TStar}) + \tilde{q}(h_{\TStar}, \p h_{\TStar}, \p\p h_{\TStar})
  \end{equation}
  and
  \begin{align}    
    &\sup_{\tStar\in [\TStar, \TStar+1]}\norm*{{\NCal}^{\TStar}_{g_b}(h_{\TStar}, \p h_{\TStar}, \p\p h_{\TStar}) }_{\InducedLTwo(\Sigma_{\tStar})}\notag \\
    \lesssim{}&
    \sup_{\tStar\in [\TStar, \TStar+1]}\norm*{q(h_{\TStar}, \p h_{\TStar}, \p\p h_{\TStar} )}_{\InducedLTwo(\Sigma_{\tStar})}
    + \sup_{\tStar\in [\TStar, \TStar+1]}\norm*{\tilde{q}(h_{\TStar}, \p h_{\TStar}, \p\p h_{\TStar} )}_{\InducedLTwo(\Sigma_{\tStar})}\notag\\
    &+ \sup_{\tStar\in [\TStar, \TStar+1]}\norm*{q_{\TStar}(h_{\TStar}, \p h_{\TStar} )}_{\InducedLTwo(\Sigma_{\tStar})}
    + \sup_{\tStar\in [\TStar, \TStar+1]}\norm*{\tilde{q}_{\TStar}(h_{\TStar}, \p h_{\TStar})}_{\InducedLTwo(\Sigma_{\tStar})} ,    \label{nonlinear:eq:NCal:TStar-bound}  
  \end{align}
  where 
  \begin{equation}
    \label{nonlinear:eq:q-q-tilde:def}
    \begin{split}
      q(h_{\TStar}, \p h_{\TStar}, \p\p h_{\TStar}) &= \sum_{\abs*{\alpha}+ \abs*{\beta} \le 2} q_{\alpha\beta}\p^\alpha h_{\TStar} \p^\beta h_{\TStar},\\
      \tilde{q}(h_{\TStar}, \p h_{\TStar}) &= \sum_{\abs*{\alpha}+ \abs*{\beta} \le 2} \tilde{q}_{\alpha\beta}\p^\alpha h_{\TStar} \p^\beta h_{\TStar},  
    \end{split}    
  \end{equation}
  and $\tilde{q}, \tilde{q}_{\TStar}$ are at least cubic in their
  arguments, and
  $q(h_{\TStar}, \p h_{\TStar}, \p\p h_{\TStar})$
  and $\tilde{q}(h_{\TStar}, \p h_{\TStar}, \p\p h_{\TStar})$ are
  quasilinear in $h_{\TStar}$, and  $q_{\TStar}(h_{\TStar}, \p
  h_{\TStar})$ and $\tilde{q}_{\TStar}(h_{\TStar}, \p
  h_{\TStar})$ are semilinear in $h_{\TStar}$
\end{prop}

\begin{proof}
  The first conclusion in \eqref{nonlinear:eq:NCal:energy-bound} follows
  directly from the definition of $\NCal^{\TStar}_{g_b}$ and
  $\NCal_{g_b}$. 

  For the second conclusion, observe that $h-\widetilde{h}_{\TStar}$
  satisfies the Cauchy problem given by
  \begin{equation*}
    \begin{split}
      \LinEinstein_{g_b}(h-\widetilde{h}_{\TStar}) &= \NCal_{g_b}(h, \p h, \p\p h),\\
      \gamma_{\TStar}(h-\widetilde{h}_{\TStar}) &= \left(0, \underline{q}_{\TStar}(h_{\TStar}, \p h_{\TStar}) + \tilde{\underline{q}}_{\TStar}(h_{\TStar}, \p h_{\TStar})\right),
    \end{split}    
  \end{equation*}
  where $\underline{q}_{\TStar}(h_{\TStar}, \p h_{\TStar})$,
  $\tilde{\underline{q}}_{\TStar}(h_{\TStar}, \p h_{\TStar})$ are as
  defined in \eqref{nonlinear:eq:underline-q-def}. Then, since $[\TStar, \TStar+1]$ is a time
  interval of bounded length, an energy estimate with a timelike
  vectorfield (see for example Lemma \ref{nonlinear:lemma:Energy:Gronwall}) and
  Gronwall allows us to conclude using the definition of $h_{\TStar}$
  in \eqref{nonlinear:eq:h:def}. 
\end{proof}
The following corollary follows immediately from the construction of $\NCal_{g_b}^{\TStar}$.
\begin{corollary}
  \label{nonlinear:coro:semi-global-ext:constraint-mode-orthogonality}
  Fix a slowly-rotating \KdS{} metric $g_b$, and some
  $\TStar>0$. Then, let $h_{\TStar}$ denote the solution to the
  semi-global extension of EVE in
  \eqref{nonlinear:eq:EVE:bootstrap-system}. Then,
  \begin{equation*}
    \lambda_{\Constraint}[g_b]\left(\NCal_{g_b}^{\TStar}(h_{\TStar}, \p h_{\TStar}, \p\p h_{\TStar}), i_{b, \phi}(\InducedMetric_0, k_0) \right) = 0,
  \end{equation*}
  where $\phi=e^{i_{\Theta}\vartheta}$.
  Moreover, if
  \begin{equation*}
    \lambda_{\Upsilon}[g_b] \left(\NCal_{g_b}^{\TStar}(h_{\TStar}, \p h_{\TStar}, \p\p h_{\TStar}), i_{b, \phi}(\InducedMetric_0, k_0) \right) = 0,
  \end{equation*}
  then for $k\ge 3$, $h_{\TStar}$ satisfies the decay estimate
  \begin{equation*}
    \sup_{\tStar>0}e^{\BSDDecay\tStar}\norm*{h_{\TStar}}_{\HkWithT{k}(\Sigma_{\tStar})}\lesssim \norm*{\left(\NCal_{g_b}^{\TStar}(h_{\TStar}, \p h_{\TStar}, \p\p h_{\TStar}), i_{b,\phi}(\InducedMetric_0, k_0)\right)}_{D^{k,\BSDDecay,1}(\StaticRegionWithExtension)}.
  \end{equation*}   
\end{corollary}
\begin{proof}
  Observe that by construction,
  $\NCal_{g_b}^{\TStar}(h_{\TStar}, \p h_{\TStar}, \p\p h_{\TStar})=0$
  for $\tStar> \TStar+1$, and that a solution $h_{\TStar}$ to the
  semi-global Cauchy problem in \eqref{nonlinear:eq:EVE:bootstrap-system}
  eventually satisfies the linearized gauge constraint. That is, for
  $\tStar > \TStar+1$,
  \begin{equation*}
    \Constraint_{g_b}(h_{\TStar}) = 0.
  \end{equation*}
  The conclusion is then an immediate consequence of Propositions
  \ref{nonlinear:prop:orthogonality-condition} and \ref{nonlinear:prop:lambda-kernel}. 
\end{proof}

\subsection{The bootstrap assumptions}
\label{nonlinear:sec:BA-assumptions:specific}

We are now ready to specify our bootstrap assumptions. Fix a bootstrap
constant $\BSConstant =\varepsilon_0^{\frac{2}{3}}$. Then let
$\TStar>0$ be a bootstrap time for which there exist \KdS{} black hole
parameters $b_{\TStar}$, and a vector
$\vartheta_{\TStar} \in \Real^{N_{\Theta}}$  satisfying
\begin{equation}
  \label{nonlinear:BA-G}\tag{BA-G}
  \abs*{\vartheta_{\TStar}} + \abs*{b_{\TStar}-b^0} \le \BSConstant,
\end{equation}
such that, recalling $g = \phi_{\TStar}(b)^*(g_{b_{\TStar}}+h)$,
the metric perturbation $h_{\TStar}$ solving
\eqref{nonlinear:eq:EVE:bootstrap-system} satisfies the orthogonality condition
\begin{equation}
  \label{nonlinear:ORT-T} \tag{ORT} 
  \lambda_{\Upsilon}[g_{b_{\TStar}}]\left(\NCal^{\TStar}_{b_{\TStar}}(h_{\TStar},\p h_{\TStar}, \p\p h_{\TStar}), i_{b_{\TStar},\phi_{\TStar}}(\InducedMetric_0, k_0)\right) = 0,
\end{equation}
where $\lambda_{\Upsilon}$ is as constructed
in Proposition \ref{nonlinear:prop:orthogonality-condition}. Moreover,
assume that $h_{\TStar}$ satisfies in addition the bootstrap estimates
\begin{align}
  \sup_{\tStar\ge 0}e^{-\BSEGrowth \tStar} \norm*{h_{\TStar}}_{\HkWithT{j}(\Sigma_{\tStar})}
  &\le \BSConstant, \qquad j\le \BSTopLvl \label{nonlinear:BA-E}\tag{BA-E},\\
  \norm*{e^{\BSDDecay\tStar} h_{\TStar}}_{H^j(\StaticRegionWithExtension)}
  &\le \BSConstant,\qquad j\le \BSLowLvl, \label{nonlinear:BA-D}\tag{BA-D} 
\end{align}
where we will take $\BSTopLvl=6$, and $\BSLowLvl=3$. We pick
$\BSEGrowth>0$ and $\BSDDecay>0$ such that $\BSDDecay<\SpectralGap$,
where $\SpectralGap$ is as in Proposition
\ref{nonlinear:prop:orthogonality-condition}, and moreover, $\BSEGrowth$,
$\BSDDecay$, $\BSTopLvl$, and $\BSLowLvl$ satisfy the conditions
\begin{align}  
  2 \BSEGrowth &< ( \BSTopLvl - \BSLowLvl - 2)\BSDDecay, \label{nonlinear:eq:BA-ass:Exp-Growth-Decay}\\
  2 &< \BSTopLvl - \BSLowLvl. \label{nonlinear:eq:BA-ass:High-Low-Diff}
\end{align}
\begin{remark}
  \label{nonlinear:remark:reg-threshold}
  The conditions on $\BSEGrowth$, $\BSDDecay$, $\BSTopLvl$, and
  $\BSLowLvl$ in
  \eqref{nonlinear:eq:BA-ass:Exp-Growth-Decay} and \eqref{nonlinear:eq:BA-ass:High-Low-Diff}
  are made so that we can use an interpolation argument to improve the
  bootstrap assumptions in what follows. We also require that
  $\BSLowLvl > \frac{5}{2}$ in order to use Sobolev embeddings. As a
  result, it should be possible to sharpen the regularity requirements
  in Theorem \ref{nonlinear:thm:main} to require only $\BSTopLvl>\frac{9}{2}$
  derivatives. However, for the sake of simplicity and to avoid
  fractional Sobolev spaces, we proceed here with the choice
  $\BSTopLvl = 6$, and $\BSLowLvl=3$.
\end{remark}

\subsection{Initialization of the bootstrap argument}
\label{nonlinear:sec:justify-BS}

In this section, we will prove that there exists some sufficiently
small $\TStar>0$ such that the bootstrap assumptions detailed in
Section \ref{nonlinear:sec:BA-assumptions:specific} hold.

The main proposition of this section is the following. 
\begin{prop}
  \label{nonlinear:prop:BA-justify}
  There exists some
  $\delta_0>0$ such that for any
  $(\InducedMetric_0, k_0)\in H^{\BSLowLvl}(\Sigma_0)\times
  H^{\BSLowLvl-1}(\Sigma_0)$, and $\TStar>0$ satisfying
  \begin{equation*}
    \norm*{(\InducedMetric_0, k_0) - (\InducedMetric_{b^0}, k_{b^0})}_{H^{\BSLowLvl}(\Sigma_0)\times H^{\BSLowLvl-1}(\Sigma_0)} + \TStar < \delta_0,
  \end{equation*}
  there exists a choice of $b\in \Real^4$ and
  $\vartheta\in\Real^{N_\Theta}$ such that
  \begin{equation}
    \label{nonlinear:eq:BA-justify:param-smallness}
    \abs*{\vartheta} + |b-b^0| \lesssim \delta_0,
  \end{equation}
  and such that the solution $h$ of \eqref{nonlinear:eq:EVE:bootstrap-system}
  satisfies the orthogonality condition \eqref{nonlinear:ORT-T}, as well as the
  bootstrap estimates \eqref{nonlinear:BA-E} and \eqref{nonlinear:BA-D}.
\end{prop}

From Proposition \ref{nonlinear:prop:BA-justify}, it is straightforward to see
that we can define $\BSConstant$ and $\varepsilon_0$ to initialize the
bootstrap argument. 
\begin{corollary}
  \label{nonlinear:coro:BA-init:true-init}  
  There exists a
  time $T_0>0$ such that for initial data $(\InducedMetric_0, k_0)$
  satisfying \eqref{nonlinear:eq:main:init-data-smallness}, there exist
  $(b_{T_0}, \vartheta_{T_0})\in \Real^4\times\Real^{N_{\Theta}}$ such
  that
  \begin{align*}
    \lambda_{\Upsilon}[g_b]\left(\NCal^{T_0}_{g_{b_{T_0}}}(h_{T_0}(b_{T_{0}}, \vartheta_{T_0}, \InducedMetric_0, k_0)), i_{b_{T_0},\phi_{T_0}}(\InducedMetric_0, k_0)\right)&=0,\\
    \abs*{\vartheta_{T_0}} + \abs*{b_{T_0}-b^0} &\le \BSConstant,\\
    \sup_{\tStar\ge 0}e^{-\BSEGrowth\tStar}\norm*{h_{T_0}(b_{T_{0}}, \vartheta_{T_0}, \InducedMetric_0, k_0)}_{\HkWithT{j}(\Sigma_{\tStar})} &\le \BSConstant,\qquad j\le \BSTopLvl,\\
    \sup_{\tStar\ge 0}e^{\BSDDecay\tStar}\norm*{h_{T_0}(b_{T_{0}}, \vartheta_{T_0}, \InducedMetric_0, k_0)}_{\HkWithT{j}(\Sigma_{\tStar})} &\le \BSConstant,\qquad j\le \BSLowLvl.
  \end{align*}  
\end{corollary}
\begin{proof}
  Setting $\BSConstant=\delta^{\frac{3}{4}}$, and taking $\delta$
  sufficiently small we have that
  \begin{equation*}
    \varepsilon_0<\frac{\delta}{2}<\BSConstant<1. 
  \end{equation*}
  Moreover, taking $T_0=\frac{\delta}{2}$, 
  \begin{equation*}
    \norm*{(\InducedMetric_0, k_0) - (\InducedMetric_{b^0}, k_{b^0})}_{H^{\BSTopLvl}(\Sigma_0)\times H^{\BSTopLvl-1}(\Sigma_0)} \le \varepsilon_0 \le \frac{\delta}{2}. 
  \end{equation*}
  Then, applying Proposition \ref{nonlinear:prop:BA-justify} and taking $\delta$
  sufficiently small concludes the proof of Corollary
  \ref{nonlinear:coro:BA-init:true-init}.
\end{proof}
The remainder of this section is focused on the proof of Proposition
\ref{nonlinear:prop:BA-justify}.  We first prove a lemma utilizing the implicit
function theorem that we will make repeated use of in the remainder of
the proof.
\begin{lemma}
  \label{nonlinear:lemma:IFT-app:fixed-f}
  Define
  \begin{gather}    
    \IFTNonlinearQuantity: \Real^4\times \Real^{N_{\Theta}} \times
    H^{\BSLowLvl}(\Sigma_0)\times H^{\BSLowLvl-1}(\Sigma_0)    
    \times \Real^+\to \Real^{4+N_{\Theta}},\notag \\ 
    \IFTNonlinearQuantity(b, \vartheta, \InducedMetric, k, \TStar)= \lambda_{\Upsilon}[g_b]\left(\mathcal{N}^{\TStar}_{g_b}(h_{\TStar}, \p h_{\TStar}, \p\p h_{\TStar}), i_{b, \phi}(\InducedMetric, k)\right),   \label{nonlinear:eq:IFT-app:F-def}    
  \end{gather}
  where $\phi=e^{i_{\Theta}\vartheta}$, and $h_{\TStar}$ is the solution to the
  Cauchy problem in \eqref{nonlinear:eq:EVE:bootstrap-system}. Also let 
  $(\tilde{b}, \tilde{\vartheta}, \tilde{\InducedMetric}, \tilde{k},
  \widetilde{T})\in \Real^4\times \Real^{N_{\Theta}} \times
    H^{\BSLowLvl}(\Sigma_0)\times H^{\BSLowLvl-1}(\Sigma_0)    
    \times \Real^+$ such that 
  \begin{equation}
    \label{nonlinear:eq:IFTnonlinearquantity-param-tilde-0}
   \IFTNonlinearQuantity(b^\dagger, \vartheta^\dagger, \InducedMetric^\dagger, k^\dagger, T^\dagger)=0, 
 \end{equation}
 and moreover,
 \begin{equation}
   \label{nonlinear:eq:IFTnonlinearquantity-derivative-0}
   \begin{split}
     \left.D_{b,\vartheta}\mathcal{N}^{\TStar}_{g_b}(h_{\TStar}, \p h_{\TStar}, \p\p h_{\TStar})\right\vert_{(b^\dagger, \vartheta^\dagger, \InducedMetric^\dagger, k^\dagger, T^\dagger)} &= 0  \\
     \evalAt*{D_{b,\vartheta}\lambda[g_b]\left(\mathcal{N}^{\TStar}_{g_b}(h_{\TStar}, \p h_{\TStar}, \p\p h_{\TStar}), i_{b,\phi}(\InducedMetric, k)\right)}_{(b^\dagger, \vartheta^\dagger, \InducedMetric^\dagger, k^\dagger, T^\dagger)}&=0. 
   \end{split}   
 \end{equation}
 Then there exists some $\delta>0$ such that defining
 \begin{equation}
   \label{nonlinear:eq:IFT-app:X-delta-def}
    \mathcal{X}_{\delta,\tilde{\InducedMetric},
    \tilde{k},\widetilde{T}} \subset
    H^{\BSLowLvl}(\Sigma_0)
    \times H^{\BSLowLvl-1}(\Sigma_0)    
    \times \Real^+,
  \end{equation}
  such that
  $(\InducedMetric, k,\TStar)\in
  \mathcal{X}_{\delta,\tilde{\InducedMetric},
    \tilde{k},\widetilde{T}}$ if and only if
  \begin{equation*}
    \norm{\InducedMetric - \tilde{\InducedMetric}}_{H^{\BSLowLvl}(\Sigma_0)}
    + \norm{k-\tilde{k}}_{H^{\BSLowLvl-1}(\Sigma_0)}
    + \abs*{\TStar - \widetilde{T}}
    < \delta,
  \end{equation*}
  there exists a function $\IFTInverse(\InducedMetric, k, \TStar)$
  depending on $(\tilde{b},
  \tilde{\vartheta},\tilde{\InducedMetric},\tilde{k},\widetilde{T})$   
  \begin{equation}
    \label{nonlinear:eq:IFT-app:fixed-f:G-def}
    \IFTInverse(\InducedMetric, k, \TStar): \mathcal{X}_{\delta,\tilde{\InducedMetric},
    \tilde{k},\widetilde{T}} \to \Real^4\times \Real^{N_{\Theta}},
  \end{equation}
  which is well-defined and $C^1$ in its arguments on $\mathcal{X}_{\delta,\tilde{\InducedMetric},
    \tilde{k},\widetilde{T}}$.
  Moreover, for $(\InducedMetric, k, \TStar)\in \mathcal{X}_{\delta,\tilde{\InducedMetric},
    \tilde{k},\widetilde{T}}$,
  \begin{equation}
    \label{nonlinear:eq:IFT-app:fixed-f:G-inverse-prop}
    \IFTNonlinearQuantity(b,\vartheta, \InducedMetric, k, \TStar) = 0,\quad \iff
    \quad (b, \vartheta) = \IFTInverse(\InducedMetric, k, \TStar).
  \end{equation}
\end{lemma}
\begin{proof}
  Recalling the definition of $(g_{b}')^\Upsilon(b')$ in
  \eqref{nonlinear:eq:g-linearized-wave-gauge:def}, observe that we can calculate
  \begin{equation*}
    D_{b,\vartheta}i_{b,\phi}(\InducedMetric, k)\vert_{(\tilde{b}, \tilde{\vartheta}, \tilde{\InducedMetric}, \tilde{k}, \widetilde{T})}(b', \vartheta')
    = \gamma_0\left((g_{\tilde{b}}')^\Upsilon(b')+\nabla_{g_{\tilde{b}}}\otimes i_{\Theta}\vartheta'\right),   
  \end{equation*}
  As a result, using \eqref{nonlinear:eq:IFTnonlinearquantity-derivative-0},
  \begin{equation*}
    D_{b,\vartheta}\IFTNonlinearQuantity = \lambda_{\Upsilon}[g_{\tilde{b}}]\left(0, \gamma_0\left((g_{\tilde{b}}')^\Upsilon(b') + \nabla_{g_{\tilde{b}}}\otimes i_{\Theta}\vartheta' \right)\right).
  \end{equation*}
  Moreover, this is an isomorphism by Proposition \ref{nonlinear:prop:lambda-isomorphism}. The
  existence of $\mathcal{X}_{\delta, \tilde{\InducedMetric},
    \tilde{k}, \widetilde{T}}$ as in \eqref{nonlinear:eq:IFT-app:X-delta-def}, and
  of $\IFTInverse$ as in \eqref{nonlinear:eq:IFT-app:fixed-f:G-def} satisfying
  \eqref{nonlinear:eq:IFT-app:fixed-f:G-inverse-prop} is then a direct result of
  the implicit function theorem in Theorem \ref{nonlinear:thm:IFT}.  
\end{proof}

Next we prove two auxiliary lemmas that will be used in conjunction
with the application of the implicit function theorem in Lemma
\ref{nonlinear:lemma:IFT-app:fixed-f}. The first, Lemma
\ref{nonlinear:lemma:IFT-init:nbhd-construction}, locates the neighborhood on
which we will perform the implicit function theorem, and the second,
Lemma \ref{nonlinear:lemma:IFT-Lipchitz-Constant}, gives control over the
derivatives of $\IFTNonlinearQuantity$.
\begin{lemma}
  \label{nonlinear:lemma:IFT-init:nbhd-construction}
  There exists some $\delta_{init}>0$ sufficiently small, and a
  neighborhood $(b^0,0)\in \mathcal{B}_{init}\subset
  \Real^4\times\Real^{N_{\Theta}}$ such that for any $(b,\vartheta)\in
  \mathcal{B}_{init}$, the solution
  \begin{equation*}
    h_{\delta_{init}}(b,\vartheta,\InducedMetric_0, k_0) := \chi_{\delta_{init}}h(b, \vartheta, \InducedMetric_0, k_0) + (1-\chi_{\delta_{init}})\widetilde{h}_{\delta_{init}}(b, \vartheta, \InducedMetric_0, k_0),
  \end{equation*}
  using the definitions of
  $h_{\TStar}(b,\vartheta, \InducedMetric, k)$,
  $h(b,\vartheta, \InducedMetric, k)$, and
  $\tilde{h}_{\TStar}(b,\vartheta, \InducedMetric, k)$ in
  \eqref{nonlinear:eq:h-def:with-params}, satisfies the following estimates
  \begin{equation}
    \label{nonlinear:eq:IFT-init:nbhd-estimates}
    \begin{split}
      \sup_{\tStar<\delta_{init}+1} e^{\BSDDecay\tStar}\norm*{h_{\delta_{init}}(b,\vartheta,\InducedMetric_0, k_0)}_{\HkWithT{k}(\Sigma_{\tStar})}&\le 1,\qquad k\le \BSLowLvl,\\
      \sup_{\tStar<\delta_{init}+1} e^{-\BSEGrowth\tStar}\norm*{h_{\delta_{init}}(b,\vartheta,\InducedMetric_0, k_0)}_{\HkWithT{k}(\Sigma_{\tStar})}&\le 1,\qquad k\le \BSTopLvl.      
    \end{split}
  \end{equation}
\end{lemma}
\begin{proof}
  We consider the neighborhood
  \begin{equation*}
    \mathcal{B}_{init} := \curlyBrace*{(b,\vartheta): \abs*{b-b^0} + \abs*{\vartheta}<\delta_{init}}, 
  \end{equation*}
  and show that for a sufficiently small choice of $\delta_{init}$,
  $\mathcal{B}_{init}$ satisfies the conditions in the lemma.

  Observe that by local existence, if $\delta_{init}$ is
  sufficiently small, we have that
  \begin{equation*}
    \sup_{\tStar<\delta_{init}+1}\norm*{h(b,\vartheta,\InducedMetric_0, k_0)}_{\HkWithT{\BSTopLvl}(\Sigma_{\tStar})}
    \lesssim \varepsilon_0+\delta_{init}.
  \end{equation*}
  The estimates in \eqref{nonlinear:eq:IFT-init:nbhd-estimates} are then
  satisfied by taking $\varepsilon_0$, $\delta_{init}$ sufficiently
  small. 
\end{proof}

\begin{lemma}
  \label{nonlinear:lemma:IFT-Lipchitz-Constant}
  Let $\IFTNonlinearQuantity$ be as defined in
  \eqref{nonlinear:eq:IFT-app:F-def}. Then,
  let
  \begin{equation*}
    \left(b^\dagger, \vartheta^\dagger, \InducedMetric^\dagger,
      k^\dagger, T^\dagger\right) \in \Real^4 \times
    \Real^{N_{\Theta}}\times H^{\BSLowLvl}(\Sigma_0)\times
    H^{\BSLowLvl-1}(\Sigma_0)\times \Real^+,
  \end{equation*}
  and let us denote
  $h^\dagger_{T^\dagger}:=h_{T^\dagger}(b^\dagger, \vartheta^\dagger,
  \InducedMetric^\dagger, k^\dagger)$ using the definition of
  $h_{\TStar}(b,\vartheta,\InducedMetric,k)$ given in
  \eqref{nonlinear:eq:h-def:with-params}.

  Moreover assume that
  $\left(b^\dagger, \vartheta^\dagger, \InducedMetric^\dagger,
    k^\dagger, T^\dagger\right)$ was chosen so that $h^\dagger_{T^\dagger}$
  satisfies the control
  \begin{equation}
    \label{nonlinear:eq:IFT-Lipschitz-Constant:BA}
    \begin{split}
      \sup_{\tStar<T^\dagger+1}e^{\BSDDecay\tStar}\norm*{h^\dagger_{T^\dagger}}_{\HkWithT{\BSLowLvl}(\Sigma_{\tStar})}
      &\le 1,\\
      \sup_{\tStar<T^\dagger+1}e^{-\BSEGrowth\tStar}\norm*{h^\dagger_{T^\dagger}}_{\HkWithT{\BSTopLvl}(\Sigma_{\tStar})} &\le 1,  
    \end{split}    
  \end{equation}  
  and the orthogonality condition \eqref{nonlinear:ORT-T}.

  Then for $T^\dagger$ sufficiently small, 
  \begin{align}    
    \abs*{\left.
    D_{b, \vartheta, \InducedMetric, k, \TStar}\IFTNonlinearQuantity(b, \vartheta, \InducedMetric, k, \TStar)
    \right\vert_{\left(b^\dagger, \vartheta^\dagger, \InducedMetric^\dagger,k^\dagger, T^\dagger\right)}}
    \lesssim 1. \label{nonlinear:eq:IFT:extension:Lipschitz-constant}
  \end{align}
\end{lemma}
\begin{proof}  
  We first consider
  \begin{equation*}
    \left.D_{b,\vartheta}\IFTNonlinearQuantity\right\vert_{\left(b^\dagger, \vartheta^\dagger, \InducedMetric^\dagger,k^\dagger, T^\dagger\right)}
    = \left.D_{b,\vartheta} \lambda_{\Upsilon}[g_b]\left(
        \NCal^{\TStar}_{g_b}(h_{\TStar}, \p h_{\TStar}, \p\p h_{\TStar}), i_{b,\phi}(\InducedMetric, k)\right)
      \right\vert_{\left(b^\dagger, \vartheta^\dagger, \InducedMetric^\dagger,k^\dagger, T^\dagger\right)}. 
  \end{equation*}
  Since $\NCal^{\TStar}_{g_b}(h_{\TStar}, \p h_{\TStar}, \p\p
  h_{\TStar})$ is supported on $[0, \TStar+1]$, we have from Lemma
  \ref{nonlinear:lemma:lambda-b-derivative} and
  \eqref{nonlinear:eq:IFT-Lipschitz-Constant:BA} that 
  \begin{equation*}
    \abs*{\evalAt*{
        (D_{b}\lambda_{\Upsilon}[g_b])\left(
          \NCal^{\TStar}_{g_b}(h_{\TStar}, \p h_{\TStar}, \p\p h_{\TStar}), i_{b,\phi}(\InducedMetric, k)\right)
      }_{\left(b^\dagger, \vartheta^\dagger, \InducedMetric^\dagger, k^\dagger, T^\dagger\right)}}
    \lesssim 1.
  \end{equation*}
  Observe that by construction,
  \begin{equation*}
    \norm*{\left.D_{b,\vartheta,\InducedMetric, k, \TStar}i_{b, \phi}(\InducedMetric, k)\right\vert_{\left(b^\dagger, \vartheta^\dagger, \InducedMetric^\dagger,
        k^\dagger, T^\dagger\right)}}_{\LSolHk{\BSLowLvl+1}(\Sigma_0)}\lesssim 1. 
  \end{equation*}
  We now consider
  $D_{b,\vartheta,\InducedMetric, k, \TStar}\NCal^{\TStar}_{g_b}(h_{\TStar},
  \p h_{\TStar}, \p\p h_{\TStar})$. Recall that up to
  $\TStar$,  the nonlinearity $\NCal^{\TStar}_{g_b}$ can be decomposed as
  \begin{equation*}
    \NCal^{\TStar}_{g_b}(h_{\TStar}, \p h_{\TStar}, \p\p h_{\TStar}) = q(h_{\TStar}, \p h_{\TStar}, \p\p h_{\TStar}) + \tilde{q}(h_{\TStar}, \p h_{\TStar}, \p\p h_{\TStar}),
  \end{equation*}
  where $q(\cdot, \cdot, \cdot)$ and $\tilde{q}(\cdot, \cdot, \cdot)$
  are as defined in \eqref{nonlinear:eq:q-q-tilde:def}, and that on $[\TStar,
  \TStar+1]$, we have the control given in
  \eqref{nonlinear:eq:NCal:TStar-bound}. 

  We show how to deal with $q(\cdot,\cdot,\cdot)$ since this is
  the primary difficulty. First observe that % for $T^\dagger\ll 1$,
  \begin{align*}
    &\norm*{\left.
        D_{b,\vartheta,\InducedMetric, k, \TStar}q\left(h_{\TStar}, \p h_{\TStar}, \p\p h_{\TStar}\right)
      \right\vert_{\left(b^\dagger, \vartheta^\dagger, \InducedMetric^\dagger,k^\dagger, T^\dagger\right)}
      }_{H^{\BSLowLvl-1}(\StaticRegionWithExtension)}\\
    \lesssim{}& \sup_{\tStar\le T^\dagger+1}\norm*{\left.D_{b,\vartheta,\InducedMetric, k, \TStar}
                h_{\TStar}\right\vert_{\left(b^\dagger, \vartheta^\dagger, \InducedMetric^\dagger,k^\dagger, T^\dagger\right)}}_{\HkWithT{\BSLowLvl}(\Sigma_{\tStar})}
                \int_{0}^{T^\dagger+1}\norm*{h_{\TStar}}_{\HkWithT{\BSLowLvl+1}(\Sigma_{\tStar})}\,d\tStar\\
    &+ \sup_{\tStar\le T^\dagger+1}\norm{h_{\TStar}}_{\HkWithT{\BSLowLvl}(\Sigma_{\tStar})}
      \int_{0}^{T^\dagger+1}\norm*{\left.D_{b,\vartheta,\InducedMetric, k, \TStar}h_{\TStar}\right\vert_{\left(b^\dagger, \vartheta^\dagger, \InducedMetric^\dagger,k^\dagger, T^\dagger\right)}}_{\HkWithT{\BSLowLvl+1}(\Sigma_{\tStar})}\,d\tStar.
  \end{align*}
  Thus we see that in order to prove the desired bound in
  \eqref{nonlinear:eq:IFT:extension:Lipschitz-constant}, it suffices to prove
  that
  \begin{equation}
    \label{nonlinear:eq:IFT:extension:Lipschitz-constant:h-bound}
    \sup_{\tStar\le T^\dagger+1}\norm*{
      \left.D_{b,\vartheta,\InducedMetric, k,\TStar}h_{\TStar}
      \right\vert_{\left(b^\dagger,\vartheta^\dagger, \InducedMetric^\dagger,k^\dagger,T^\dagger\right)}
    }_{\HkWithT{\BSLowLvl+1}(\Sigma_{\tStar})}
    \lesssim 1.
  \end{equation}
  To prove \eqref{nonlinear:eq:IFT:extension:Lipschitz-constant:h-bound}, we
  first observe that
  \begin{equation*}
    D_{b,\vartheta,\InducedMetric, k,\TStar}h_{\TStar}
    = \chi_{\TStar}D_{b,\vartheta,\InducedMetric, k}h
    + (1-\chi_{\TStar})D_{b,\vartheta,\InducedMetric, k,\TStar}\widetilde{h}_{\TStar}
    + D_{\TStar}\chi_{\TStar}(h - \widetilde{h}_{\TStar})
    ,
  \end{equation*}
  where 
  $h, \widetilde{h}_{\TStar}$ are as defined in
  \eqref{nonlinear:eq:h-def:with-params}.
  For shorthand, we denote
  \begin{equation*}
    h^\dagger:=h\left(b^\dagger, \vartheta^\dagger, \InducedMetric^\dagger, k^\dagger\right),\qquad
    \widetilde{h}^\dagger:=\widetilde{h}_{T^\dagger}\left(b^\dagger, \vartheta^\dagger, \InducedMetric^\dagger, k^\dagger\right). 
  \end{equation*}
  Using the controls in
  \eqref{nonlinear:eq:IFT-Lipschitz-Constant:BA}, it suffices to
  show that 
  \begin{align}
    \sup_{\tStar\le T^\dagger+1}\norm*{\evalAt*{
    D_{b,\vartheta,\InducedMetric, k,\TStar}h}_{\left(b^\dagger,\vartheta^\dagger, \InducedMetric^\dagger,k^\dagger,T^\dagger\right)}}_{\HkWithT{\BSLowLvl+1}(\Sigma_{\tStar})}
    &\lesssim 1,\label{nonlinear:eq:IFT:extension:Lipschitz-constant:Dh-bound}\\
    \sup_{\tStar\le T^\dagger+1}\norm*{\evalAt*{
    D_{b,\vartheta,\InducedMetric, k,\TStar}\widetilde{h}_{\TStar}}_{\left(b^\dagger,\vartheta^\dagger, \InducedMetric^\dagger,k^\dagger,T^\dagger\right)}}_{\HkWithT{\BSLowLvl+1}(\Sigma_{\tStar})}
    &\lesssim 1 \label{nonlinear:eq:IFT:extension:Lipschitz-constant:Dtildeh-bound}
  \end{align}
  to prove \eqref{nonlinear:eq:IFT:extension:Lipschitz-constant:h-bound}.  We
  first consider
  \eqref{nonlinear:eq:IFT:extension:Lipschitz-constant:Dtildeh-bound} and show
  that in fact it suffices to show
  \eqref{nonlinear:eq:IFT:extension:Lipschitz-constant:Dh-bound}.  Observe that
  we can write the following Cauchy problem for
  $\evalAt*{D_{b,\vartheta,\InducedMetric,
      k,\TStar}\widetilde{h}_{\TStar}}_{\left(b^\dagger,\vartheta^\dagger,
      \InducedMetric^\dagger, k^\dagger, T^\dagger\right)}$.
  \begin{equation*}
    \begin{split}
      \LinEinstein_{g_{b^\dagger}} \evalAt*{D_{b,\vartheta,\InducedMetric, k,\TStar}\widetilde{h}_{\TStar}
      }_{\left(b^\dagger,\vartheta^\dagger, \InducedMetric^\dagger, k^\dagger, T^\dagger\right)}
      &= 0\\
      \gamma_{T^\dagger}\left(\evalAt*{D_{b,\vartheta,\InducedMetric, k,\TStar}\widetilde{h}_{\TStar}
        }_{\left(b^\dagger,\vartheta^\dagger, \InducedMetric^\dagger, k^\dagger, T^\dagger\right)}\right)
    &= \evalAt*{
      D_{b,\vartheta,\InducedMetric, k, \TStar}
        \iota_{b, \TStar}(\evalAt*{h}_{\Sigma_{\TStar}}, \evalAt*{\p_{\tStar}h}_{\Sigma_{\TStar}})
      }_{\left(b^\dagger,\vartheta^\dagger, \InducedMetric^\dagger,k^\dagger,T^\dagger\right)}
      .
    \end{split}
  \end{equation*}
  From the construction of $\iota_{b,\TStar}$ in Proposition
  \ref{nonlinear:prop:Bootstrap:iota-construction:init-data},
  \begin{equation*}
    \begin{split}
      &\norm*{\evalAt*{
      D_{b,\vartheta,\InducedMetric, k, \TStar}
        \iota_{b, \TStar}(\evalAt*{h}_{\Sigma_{\TStar}}, \evalAt*{\p_{\tStar}h}_{\Sigma_{\TStar}})
      }_{\left(b^\dagger,\vartheta^\dagger, \InducedMetric^\dagger,k^\dagger,T^\dagger\right)}
      }_{\LSolHk{\BSLowLvl+1}(\Sigma)}\\
      \lesssim{}& \norm*{h^\dagger}_{\HkWithT{\BSLowLvl+2}(\Sigma_{T^\dagger})}
      + \norm*{\evalAt*{
          D_{b,\vartheta,\InducedMetric, k}h
        }_{\left(b^\dagger, \vartheta^\dagger, \InducedMetric^\dagger, k^\dagger\right)}
      }_{\HkWithT{\BSLowLvl+1}(\Sigma_{T^\dagger})}.
    \end{split}    
  \end{equation*}
  Thus it is clear that in order to prove
  \eqref{nonlinear:eq:IFT:extension:Lipschitz-constant:h-bound}, it suffices to
  prove \eqref{nonlinear:eq:IFT:extension:Lipschitz-constant:Dh-bound}.
  
  To prove  \eqref{nonlinear:eq:IFT:extension:Lipschitz-constant:Dh-bound}, we
  consider separately the derivatives involved.  We will obtain the
  bound in \eqref{nonlinear:eq:IFT:extension:Lipschitz-constant} separately for
  each of the derivatives desired.  First observe that by construction
  of the semi-global extension Cauchy problem, we have that for
  $\tStar\le \TStar$,
  \begin{equation*}
    h(b,\vartheta,\InducedMetric, k) = (\phi^{-1})^*g - g_b, \qquad g = \phi^*(g_b+h_{\TStar}(b,\vartheta,\InducedMetric, k))
  \end{equation*}
  where $\phi = e^{i_{\Theta}\vartheta}$. Then, for
  $\tStar\le \TStar+1$
  \begin{equation}
    \label{nonlinear:eq:IFT:extension:Lipschitz-constant:D-h}
    \left.D_{b,\vartheta}h(b,\vartheta,\InducedMetric, k)\right\vert_{\left(b^\dagger,\vartheta^\dagger, \InducedMetric^\dagger,k^\dagger\right)}(b', \vartheta')
    = (g'_{b^\dagger})^{\Upsilon}(b') + \nabla{g_{b^\dagger}} \otimes\vartheta',
  \end{equation}
  where $(g'_{b^\dagger})^{\Upsilon}(b')$ is as defined in
  \eqref{nonlinear:eq:g-linearized-wave-gauge:def}. We briefly note that the
  appearance of the $\omega^{\Upsilon}_b(b')$ term, reflected in the
  appearance of $(g'_{b^\dagger})^{\Upsilon}(b')$ rather than
  $(g'_{b^\dagger})(b')$ in
  \eqref{nonlinear:eq:IFT:extension:Lipschitz-constant:D-h} is necessary in
  light of the fact that the specific form of the \KdS{} metrics
  defined in \eqref{nonlinear:eq:KdS:regular} did not take any gauge
  considerations into account. 
  
  Since we know that there exists some fixed $\GronwallExp>0$ such that
  \begin{equation*}
    \norm*{(g'_{b^\dagger})^{\Upsilon}(b') + \nabla{g_{b^\dagger}} \otimes\vartheta'}_{H^{\BSTopLvl}(\Sigma_{\tStar})} \lesssim e^{\GronwallExp \tStar},
  \end{equation*}
  we  have that for $T^{\dagger}\ll 1$,
  \begin{equation}
    \label{nonlinear:eq:IFT:extension:Lipschitz-constant:b-theta}
    \sup_{\tStar\le T^{\dagger}}\norm*{\left.D_{b,\vartheta}h_{\TStar}\right\vert_{\left(b^\dagger,\vartheta^\dagger, \InducedMetric^\dagger,k^\dagger,T^\dagger\right)}}_{\HkWithT{\BSLowLvl+1}(\Sigma_{\tStar})}\lesssim 1. 
  \end{equation}

  Now let us consider
  $\evalAt*{D_{\InducedMetric, k}h}_{\left(b^\dagger,
      \vartheta^\dagger, \InducedMetric^\dagger, k^\dagger\right)}$.
  Observe that
  $\evalAt*{D_{\InducedMetric, k}h}_{\left(b^\dagger,
      \vartheta^\dagger, \InducedMetric^\dagger, k^\dagger\right)}$
  solves the equation
  \begin{equation*}
    \begin{split}
      \LinEinsteinS_{g^\dagger}\left. D_{\InducedMetric, k}h \right\vert_{\left(b^\dagger,\vartheta^\dagger, \InducedMetric^\dagger,k^\dagger\right)}
      &= \left.\mathcal{P}\left(h,  D_{\InducedMetric, k}h \right)\right\vert_{\left(b^\dagger,\vartheta^\dagger, \InducedMetric^\dagger,k^\dagger\right)},\\
      \gamma_0(h) &= \left.D_{\InducedMetric, k}i_{b^\dagger,\phi^\dagger}(\InducedMetric, k)\right\vert_{\left(b^\dagger,\vartheta^\dagger, \InducedMetric^\dagger,k^\dagger\right)},
    \end{split}
  \end{equation*}
  where  
  \begin{equation}
     \label{nonlinear:eq:IFT:extension:Lipschitz-constant:P-cal-def}
    \mathcal{P}\left(h, D_{\InducedMetric,k}h\right)
    :={} D_{\InducedMetric,k}\QCal(h, \p h)          
    - D_{\InducedMetric,k}\SubPOp_{g}[h]
    - D_{\InducedMetric,k}\PotentialOp_{g}[h]
    + D_{\InducedMetric,k}\ScalarWaveOp[g]h
    .  
  \end{equation}
  Then directly using local existence, we have that
  \begin{align*}
    \sup_{\tStar\le T^\dagger} \norm*{\evalAt*{D_{\InducedMetric,k}h}_{\left(b^\dagger, \vartheta^\dagger, \InducedMetric^\dagger, k^\dagger\right)}}_{\HkWithT{\BSLowLvl+1}(\Sigma_{\tStar})}
    \lesssim{}& \norm*{\evalAt*{D_{\InducedMetric,k}i_{b^\dagger, \phi^\dagger}(\InducedMetric, k)}_{\left(b^\dagger, \vartheta^\dagger, \InducedMetric^\dagger, k^\dagger\right)} }_{\LSolHk{\BSLowLvl+1}(\Sigma_0)} \lesssim 1.\\      
  \end{align*}
  This proves \eqref{nonlinear:eq:IFT:extension:Lipschitz-constant:Dh-bound} and
  consequently concludes the proof of Lemma \ref{nonlinear:lemma:IFT-Lipchitz-Constant}.    
\end{proof}

We are now ready to prove Proposition \ref{nonlinear:prop:BA-justify}.
\begin{proof}[Proof of Proposition \ref{nonlinear:prop:BA-justify}]
  We first show that if $\widecheck{T}$ is sufficiently small and there
  exists a choice of $(\check{b}, \check{\vartheta})$ such that
  \eqref{nonlinear:ORT-T} is satisfied for $(\check{b}, \check{\vartheta},
  \InducedMetric_0, k_0, \widecheck{T})$, and in addition we have that
  \begin{equation}
    \label{nonlinear:eq:BA-justify:check-b-theta-smallness}
    \norm*{\InducedMetric_0-\widecheck{\phi}^*\InducedMetric_{{\check{b}}}}_{H^{\BSTopLvl}(\Sigma_0)}
    + \norm*{k_0 - \widecheck{\phi}^*k_{{\check{b}}}}_{H^{\BSTopLvl-1}(\Sigma_0)}
    \lesssim \BSConstant^{\frac{4}{3}},\qquad
    \widecheck{\phi} = e^{i_{\Theta}\check{\vartheta}},
  \end{equation}
  then for sufficiently small $\BSConstant$, we also have that
  \eqref{nonlinear:BA-D} and \eqref{nonlinear:BA-E} are satisfied.  To begin, observe that
  using local existence
  theory, we have if \eqref{nonlinear:eq:BA-justify:check-b-theta-smallness}
  holds, then for $\delta_0$ sufficiently small,
  \begin{equation}
    \label{nonlinear:eq:BA-justify:local-existence-1+delta}
    \sup_{0\le\tStar\le 1+\delta_0} e^{-\BSEGrowth\tStar}\norm{\check{h}_{\widecheck{T}}}_{\HkWithT{\BSTopLvl}(\Sigma_{\tStar})}
    +  \sup_{0\le\tStar\le 1+\delta_0} e^{\BSDDecay\tStar}\norm{\check{h}_{\widecheck{T}}}_{\HkWithT{\BSLowLvl+1}(\Sigma_{\tStar})}
    \lesssim \BSConstant^{\frac{4}{3}}.
  \end{equation}
  Since \eqref{nonlinear:ORT-T} is satisfied for
  $(\check{b}, \check{\vartheta}, \InducedMetric_0, k_0,
  \widecheck{T})$, we can apply Propositions
  \ref{nonlinear:prop:orthogonality-condition} and
  \ref{nonlinear:prop:BA-E:lossless-exp-grow-ILED} to deduce that on
  $\tStar>1+\delta_0$,
  \begin{align*}
    \sup_{\tStar>1+\delta_0}e^{-\BSEGrowth\tStar}\norm*{\check{h}_{\widecheck{T}}}_{\HkWithT{\BSTopLvl}(\Sigma_{\tStar})}
    &\lesssim e^{-\BSEGrowth(1+\delta_0)}\norm*{\check{h}_{\widecheck{T}}}_{\HkWithT{\BSTopLvl}(\Sigma_{\widecheck{T}})}\\
    \sup_{\tStar>1+\delta_0}e^{\BSDDecay\tStar}\norm*{\check{h}_{\widecheck{T}}}_{\HkWithT{\BSLowLvl}(\Sigma_{\tStar})}
    &\lesssim e^{\BSDDecay(1+\delta_0)}\norm*{\check{h}_{\widecheck{T}}}_{\HkWithT{\BSLowLvl+1}(\Sigma_{\widecheck{T}})}.
  \end{align*}
  Then using \eqref{nonlinear:eq:BA-justify:local-existence-1+delta}, we have
  that for $\varepsilon_0$, and consequentially $\BSConstant$
  sufficiently small, \eqref{nonlinear:BA-E} and \eqref{nonlinear:BA-D} are satisfied as
  desired.
  
  We now prove that we can indeed find some choice of
  $(\check{b}, \check{\theta}, \widecheck{T})$
  such that the orthogonality condition \eqref{nonlinear:ORT-T}, the smallness
  condition \eqref{nonlinear:eq:BA-justify:check-b-theta-smallness}, and the
  bootstrap assumption \eqref{nonlinear:BA-G} are satisfied. The main idea will
  be to apply Lemma \ref{nonlinear:lemma:IFT-app:fixed-f}.  To this end, let us
  first define
  \begin{gather*}
    \IFTNonlinearQuantity_{init}:
    \Real^4 \times \Real^{N_{\Theta}}
    \times H^{\BSLowLvl}(\Sigma_0)\times H^{\BSLowLvl-1}(\Sigma_0)
    \times \Real^+
    \to \Real^{4+N_{\Theta}},\\
    \IFTNonlinearQuantity_{init}(b, \vartheta, \InducedMetric, k, \TStar)=\lambda_{\Upsilon}[g_b]\left(\NCal^{\TStar}_{g_b}(h_{\TStar}, \p h_{\TStar}, \p\p h_{\TStar}), i_{b,\phi}(\InducedMetric, k)\right),
  \end{gather*}
  where $\phi=e^{i_\Theta\vartheta}$, and
  $h_{\TStar}=h_{\TStar}(b,\vartheta, \InducedMetric, k)$ is as defined
  in \eqref{nonlinear:eq:h-def:with-params}. We
  will show that there exists some $\delta_{0}>0$ sufficiently small
  such that for
  \begin{equation}
    \label{nonlinear:eq:BA-justify:smallness-aux}
    \norm*{\InducedMetric - \InducedMetric_{b^0}}_{H^{\BSLowLvl}(\Sigma_0)}
    +\norm*{k - k_{b^0}}_{ H^{\BSLowLvl-1}(\Sigma_0)}
    + \TStar < \delta_{0}, 
  \end{equation}
  there exists a $C^1$ mapping
  $\IFTInverse_{init}: H^{\BSLowLvl}(\Sigma_0) \times H^{\BSLowLvl-1}(\Sigma_0)
  \times\Real^+ \to \Real^{4+N_{\Theta}}$
  such that 
  \begin{equation*}
    \IFTNonlinearQuantity_{init}\left(\IFTInverse_{init}( \InducedMetric, k, \TStar),  \InducedMetric, k, \TStar\right) = 0
  \end{equation*}
  for $(\InducedMetric, k, \TStar)$ satisfying
  \eqref{nonlinear:eq:BA-justify:smallness-aux}. Observe that using the
  construction and estimates in Lemma
  \ref{nonlinear:lemma:IFT-init:nbhd-construction}, it is clear that
  $\IFTNonlinearQuantity_{init}$ is $C^1$ in the arguments for
  $(b,\vartheta,\TStar)\in \mathcal{B}_{init}\times (0,
  \delta_{init})$, and $(\InducedMetric, k)$ satisfying
  \eqref{nonlinear:eq:main:init-data-smallness}. Moreover, observe that if
  $(b, \vartheta, \InducedMetric, k) = (b^0,0, \InducedMetric_{b^0},
  k_{b^0})$, then for any $\TStar>0$,
  $h_{\TStar}(b^0,0, \InducedMetric_{b^0}, k_{b^0}) = 0$ is the
  solution of \eqref{nonlinear:eq:EVE:bootstrap-system}. Keeping this in mind,
  we can calculate
  \begin{align*}
    \left.D_{b,\vartheta}{\NCal}_{g_b}^{\TStar}(h_{\TStar}, \p h_{\TStar}, \p\p h_{\TStar}) \right\vert_{(b^0, 0, \InducedMetric_{b^0}, k_{b^0}, 0)}(b', \vartheta') & = 0.    
  \end{align*}
  Applying Lemma \ref{nonlinear:lemma:IFT-app:fixed-f}, with the choice
  \begin{equation*}
    \left(b^\dagger, \vartheta^\dagger, \InducedMetric^\dagger, k^\dagger, T^\dagger\right)
    = \left(b^0, 0, \InducedMetric_{b^0}, k_{b^0}, 0\right)
  \end{equation*}
  we see that there exists some $\delta_0$ sufficiently small such
  that defining
  \begin{equation*}
   \mathcal{X}_{init}\subset H^{\BSLowLvl}(\Sigma_0)\times H^{\BSLowLvl-1}(\Sigma_0)\times \Real^+ ,
  \end{equation*}
  where 
  \begin{equation*}    
    \begin{split}
      \mathcal{X}_{init}
    :={}& \curlyBrace*{(\InducedMetric, k, \TStar):      
      \norm*{\InducedMetric - \InducedMetric_{b^0}}_{H^{\BSLowLvl}(\Sigma_0)}
      + \norm*{k - k_{b^0}}_{H^{\BSLowLvl-1}(\Sigma_0)}
      + \abs{\TStar} < \delta_0},    
    \end{split}       
  \end{equation*}
  there exists a $C^1$ mapping 
  \begin{equation*}
    \mathcal{G}_{init}: \mathcal{X}_{init}\to \mathcal{B}_{init}
  \end{equation*}
  such that for all $(\InducedMetric, k, \TStar)\in
  \mathcal{X}_{init}$,
  \begin{equation*}
    \mathcal{F}_{init}(\mathcal{G}_{init}(\InducedMetric, k, \TStar),\InducedMetric, k, \TStar) = 0.
  \end{equation*}
  As a result, for all
  $(\InducedMetric, k, \TStar)\in \mathcal{X}_{init}$, we have that
  $h_{\TStar}(\mathcal{G}_{init}(\InducedMetric, k,
  \TStar),\InducedMetric, k)$ satisfies \eqref{nonlinear:ORT-T}.

  It remains to control $\abs*{b-b^0} + \abs*{\vartheta}$ so that we
  can guarantee \eqref{nonlinear:eq:BA-justify:check-b-theta-smallness}, and
  \eqref{nonlinear:BA-G}. To this end, we apply Lemma
  \ref{nonlinear:lemma:IFT-Lipchitz-Constant}. From Lemma
  \ref{nonlinear:lemma:IFT-init:nbhd-construction}, we have that for
  $(b,\vartheta,\InducedMetric, k, \TStar)\in \mathcal{B}_{init}\times
  \mathcal{X}_{init}$, the controls in
  \eqref{nonlinear:eq:IFT-Lipschitz-Constant:BA} are satisfied. As a result, for
  \begin{equation*}
    \left(\IFTInverse_{init}\left(\widecheck{\InducedMetric}, \widecheck{k},\widecheck{T}\right), \widecheck{\InducedMetric}, \widecheck{k},\widecheck{T}\right) \in  \mathcal{B}_{init}\times
    \mathcal{X}_{init},
  \end{equation*}
  we have from Lemma \ref{nonlinear:lemma:IFT-Lipchitz-Constant} that for
  $\widecheck{T}\ll 1$ sufficiently small,
  \begin{equation}
    \label{nonlinear:eq:IFT-Lipschitz-Constant:F-deriv-aux-1}
    \begin{split}
      \abs*{D_{\InducedMetric, k, \TStar}\IFTNonlinearQuantity_{init}\left(
    \IFTInverse_{init}\left(\widecheck{\InducedMetric}, \widecheck{k},\widecheck{T}\right), \widecheck{\InducedMetric}, \widecheck{k},\widecheck{T}
    \right)}&\lesssim 1,\\
    \abs*{D_{b,\vartheta}\IFTNonlinearQuantity_{init}\left(
    \IFTInverse_{init}\left(\widecheck{\InducedMetric}, \widecheck{k},\widecheck{T}\right), \widecheck{\InducedMetric}, \widecheck{k},\widecheck{T}
    \right)^{-1}}&\lesssim 1. 
    \end{split}    
  \end{equation}  
  From the chain rule, we have that
  \begin{align*}
    D_{\InducedMetric, k, \TStar} \IFTInverse_{init}\left(\widecheck{\InducedMetric}, \widecheck{k}, \widecheck{T}\right)
    ={}& D_{b, \vartheta}\IFTNonlinearQuantity_{init}\left(\IFTInverse_{init}\left(\widecheck{\InducedMetric}, \widecheck{k},\widecheck{T}\right), \widecheck{\InducedMetric}, \widecheck{k},\widecheck{T}\right)^{-1}
         D_{\InducedMetric, k, \TStar}\IFTNonlinearQuantity_{init}\left(\IFTInverse_{init}\left(\widecheck{\InducedMetric}, \widecheck{k},\widecheck{T}\right), \widecheck{\InducedMetric}, \widecheck{k},\widecheck{T}
      \right).
  \end{align*}
  Then for
  $\left(\widecheck{\InducedMetric},\widecheck{k},\widecheck{T}\right)\in
  \mathcal{X}_{init}$, denote
  $(\check{b},
  \check{\vartheta})=\IFTInverse\left(\widecheck{\InducedMetric},\widecheck{k},\widecheck{T}\right)$. From
  the $C^1$ nature of $\IFTInverse_{init}$ and
  \eqref{nonlinear:eq:IFT-Lipschitz-Constant:F-deriv-aux-1},
  \begin{align}
    \abs*{\check{b}-b^0} + \abs*{\check{\vartheta}}
    &\le \sup_{\mathcal{X}_{init}}\abs*{D_{\InducedMetric, k, \TStar}\IFTInverse_{init}}\left(
      \norm*{\widecheck{\InducedMetric}-\InducedMetric_{b^0}}_{H^{\BSLowLvl}(\Sigma_0)}
      + \norm*{\widecheck{k}-k_{b^0}}_{H^{\BSLowLvl-1}(\Sigma_0)}
      + \abs*{\widecheck{T}}
      \right)\notag \\
    &\le C\left(\norm*{\widecheck{\InducedMetric}-\InducedMetric_{b^0}}_{H^{\BSLowLvl}(\Sigma_0)}
      + \norm*{\widecheck{k}-k_{b^0}}_{H^{\BSLowLvl-1}(\Sigma_0)}
      + \abs*{\widecheck{T}}\right), \label{nonlinear:eq:IFT-init:b-theta:bound}    
  \end{align}
  where $C$ is an implicit constant independent of $\delta_0$,
  $\BSConstant$, and $\varepsilon_0$. 
  Now choosing $\BSConstant^{\frac{4}{3}} = \delta<\delta_0$, and making
  $\delta$ smaller as necessary, we can in fact guarantee that
  \begin{equation*}
    2C\varepsilon_0 \le \delta_0. 
  \end{equation*}
  Thus, choosing $T_0 = \frac{1}{2C}\varepsilon_0$, we have
  in fact that
  \begin{equation*}
    (\InducedMetric_0, k_0, T_0) \in \mathcal{X}_{init}.
  \end{equation*}
  As a result, denoting
  \begin{equation*}
    (b_{T_0}, \vartheta_{T_0}):=\IFTInverse_{init}\left(\InducedMetric_0, k_0, T_0\right),
  \end{equation*}
  it is clear that \eqref{nonlinear:eq:IFT-init:b-theta:bound} implies that 
  \begin{equation*}
    \abs*{b_{T_0}-b_0} + \abs*{\vartheta_{T_0}} \le \BSConstant^{\frac{4}{3}}. 
  \end{equation*}
  This shows that \eqref{nonlinear:BA-G} holds.  Finally, it remains to show
  that we can fulfill
  \eqref{nonlinear:eq:BA-justify:check-b-theta-smallness}. Denoting
  $\phi_{T_0} = e^{i_{\Theta}\vartheta_{T_0}}$, observe
  that
  \begin{align*}
     &\norm*{\InducedMetric_0-\phi_{T_0}^*\InducedMetric_{b_{T_0}}}_{H^{\BSTopLvl}(\Sigma_0)}
     + \norm*{k_0 - \phi_{T_0}^*k_{b_{T_0}}}_{H^{\BSTopLvl-1}(\Sigma_0)}\\
     \lesssim{}&
     \norm*{\InducedMetric_0-\InducedMetric_{b^0}}_{H^{\BSTopLvl}(\Sigma_0)}
     + \norm*{k_0 - k_{b^0}}_{H^{\BSTopLvl-1}(\Sigma_0)}\\
     &+ \norm*{\InducedMetric_{b^0}-\phi_{T_0}^*\InducedMetric_{b_{T_0}}}_{H^{\BSTopLvl}(\Sigma_0)}
     + \norm*{k_{b^0} - \phi_{T_0}^*k_{b_{T_0}}}_{H^{\BSTopLvl-1}(\Sigma_0)}\\
     \lesssim{}& \BSConstant^{\frac{4}{3}}.
  \end{align*}
  Thus \eqref{nonlinear:eq:BA-justify:check-b-theta-smallness} holds, and so do
  \eqref{nonlinear:BA-D}, \eqref{nonlinear:BA-E} by the argument made at the beginning of
  the proof. This concludes the proof of Proposition
  \ref{nonlinear:prop:BA-justify}.
\end{proof}

\section{Proof of Theorem \ref{nonlinear:thm:main}}
\label{nonlinear:sec:bs}

This section is devoted to the proof of Theorem \ref{nonlinear:thm:main}. Recall
from Proposition \ref{nonlinear:prop:BA-justify} that the bootstrap assumptions
\eqref{nonlinear:BA-D}, \eqref{nonlinear:BA-E}, \eqref{nonlinear:BA-G}, and the orthogonality
condition \eqref{nonlinear:ORT-T} hold for some $\TStar>0$ sufficiently
small. We now consider any bootstrap time $\TStar>0$ and show that
\eqref{nonlinear:BA-D}, \eqref{nonlinear:BA-E}, \eqref{nonlinear:BA-G} can be improved, and thus,
that the bootstrap time $\TStar$ can be extended slightly. We will
then conclude the continuity argument in Section
\ref{nonlinear:sec:closing-proof}.

\subsection{Improving the bootstrap assumption on decay}
\label{nonlinear:sec:bs-decay}

We first improve the bootstrap assumption on decay \eqref{nonlinear:BA-D}.  The
main goal of this section is the following proposition.
\begin{prop}
  \label{nonlinear:prop:BA-D:close}
  Let $\TStar>0$ be the bootstrap time, for which there exist
  slowly-rotating \KdS{} black hole parameters $b_{\TStar}$ and a
  diffeomorphism $\phi_{\TStar}=e^{i_{\Theta}\vartheta_{\TStar}}$ such
  that the solution
  $h_{\TStar} = h_{\TStar}(b,\vartheta,\InducedMetric_0, k_0)$ to
  \eqref{nonlinear:eq:EVE:bootstrap-system} satisfies the bootstrap assumptions
  given in \eqref{nonlinear:BA-E}, \eqref{nonlinear:BA-D}, \eqref{nonlinear:BA-G}, and
  \eqref{nonlinear:ORT-T}. Then in fact, $h_{\TStar}$ satisfies the improved
  estimate
  \begin{equation}
    \label{nonlinear:eq:BA-D:improved}
    \sup_{\tStar>0}e^{\BSDDecay \tStar} \norm*{h_{\TStar}}_{\HkWithT{j}(\Sigma_{\tStar})}
    \lesssim \varepsilon_0 + \BSConstant^2, \qquad j\le N_d.
  \end{equation}
\end{prop}

\begin{proof}
  Recall that by assumption, we have $b_{\TStar}, \vartheta_{\TStar},
  h_{\TStar}, \TStar$ such that \eqref{nonlinear:ORT-T} is satisfied. As a
  result, using Proposition \ref{nonlinear:prop:orthogonality-condition} we have that for any $k\le N_d$,
  \begin{equation*}
    e^{\BSDDecay\tStar}\norm*{h_{\TStar} (\tStar, \cdot)}_{\HkWithT{k}(\Sigma_{\tStar})}
    \lesssim \norm*{\left(\NCal_{g_{b_{\TStar}}}^{\TStar}( h_{\TStar} , \p h_{\TStar} , \p\p h_{\TStar} ), h_0, h_1\right)}_{D^{k,\BSDDecay,1}(\StaticRegionWithExtension)},
  \end{equation*}
  where $(h_0, h_1) = i_{b_{\TStar},\phi_{\TStar}}(g_0, k_0)$. Then,
  using the control that we have for
  $\NCal_{g_{b_{\TStar}}}^{\TStar}$ in Proposition \ref{nonlinear:prop:N-control}, we have that
  \begin{align*}
    \int_{\Real^+} e^{\BSDDecay \tStar}\norm*{\NCal_{g_{b_{\TStar}}}^{\TStar}(h_{\TStar}, \p h_{\TStar}, \p\p h_{\TStar})}_{\HkWithT{k}(\Sigma_{\tStar})}\,d\tStar
    \lesssim &{}
               \int_0^{\TStar} e^{\BSDDecay \tStar}\norm{q(h_{\TStar},\p h_{\TStar}, \p\p h_{\TStar})}_{\HkWithT{k}(\Sigma_{\tStar})}\,d\tStar\\
             &+\int_{\TStar}^{\TStar+1} e^{\BSDDecay \tStar}\norm{q_{\TStar}(h_{\TStar},\p h_{\TStar}, \p\p h_{\TStar})}_{\HkWithT{k}(\Sigma_{\tStar})}\,d\tStar\\
             &+\int_0^{\TStar} e^{\BSDDecay \tStar}\norm{\tilde{q}(h_{\TStar},\p h_{\TStar}, \p\p h_{\TStar})}_{\HkWithT{k}(\Sigma_{\tStar})}\,d\tStar\\ 
             &+\int_{\TStar}^{\TStar+1} e^{\BSDDecay \tStar}\norm{\tilde{q}_{\TStar}(h_{\TStar},\p h_{\TStar}, \p\p h_{\TStar})}_{\HkWithT{k}(\Sigma_{\tStar})}\,d\tStar.  
  \end{align*}
  We show how to control the quadratic terms $q, q_{\TStar}$ first
  since they are the main difficulty. For $k\le \BSLowLvl$, using that
  $\HkWithT{\BSLowLvl}(\Sigma)$ is an algebra since $\BSLowLvl\ge 2$,
  we have that
  \begin{align}    
    &\int_0^{\TStar} e^{\BSDDecay \tStar}\norm{q(h_{\TStar},\p h_{\TStar}, \p\p h_{\TStar})}_{\HkWithT{k}(\Sigma_{\tStar})}\,d\tStar
      +\int_{\TStar}^{\TStar+1} e^{\BSDDecay \tStar}\norm{q_{\TStar}(h_{\TStar},\p h_{\TStar}, \p\p h_{\TStar})}_{\HkWithT{k}(\Sigma_{\tStar})}\,d\tStar\notag \\
    \lesssim{}& \int_{\Real^+} e^{\BSDDecay\tStar} \norm{h_{\TStar} }_{\HkWithT{\BSLowLvl+2}(\Sigma_{\tStar})}\norm{h_{\TStar} }_{\HkWithT{\BSLowLvl}(\Sigma_{\tStar})}\,d\tStar\label{nonlinear:eq:BA-D:aux1}.
  \end{align}
  Then, we control
  $e^{\BSDDecay
    \tStar}\norm{h_{\TStar}}_{\HkWithT{\BSLowLvl}(\Sigma_{\tStar})}$
  using the bootstrap assumption
  \eqref{nonlinear:BA-D}. 
  To handle the
  $\norm{h_{\TStar} }_{\HkWithT{\BSLowLvl+2}(\Sigma_{\tStar})}$ term, we
  interpolate between the $\HkWithT{\BSTopLvl}(\Sigma)$ and
  $\HkWithT{\BSLowLvl}(\Sigma)$ norms, observing that
  \begin{equation}
    \label{nonlinear:eq:BA-D:interpolation-step-1}
    \norm*{ h_{\TStar}}_{\HkWithT{\BSLowLvl+2}(\Sigma_{\tStar})}
    \lesssim{} \norm*{ h_{\TStar} }_{\HkWithT{\BSLowLvl}(\Sigma_{\tStar})}^{1-\theta}\norm*{ h_{\TStar} }_{\HkWithT{\BSTopLvl}(\Sigma_{\tStar})}^{\theta},\qquad \theta = \frac{2}{\BSTopLvl-\BSLowLvl}.
  \end{equation}
  As a result, we have that
  \begin{align}
    \int_{\Real^+}\norm*{ h_{\TStar} }_{\HkWithT{\BSLowLvl+2}(\Sigma_{\tStar})}\,d\tStar
    &\lesssim \int_{\Real^+} e^{(\BSEGrowth\theta - (1-\theta)\BSDDecay)\tStar}
      e^{-\BSEGrowth\theta\tStar}\norm{h_{\TStar} }_{\HkWithT{\BSTopLvl}(\Sigma_{\tStar})}
      e^{\BSDDecay(1-\theta)\tStar}\norm{h_{\TStar} }_{\HkWithT{\BSLowLvl}(\Sigma_{\tStar})} \, d\tStar\notag \\
    &\lesssim \BSConstant \norm*{e^{(\BSEGrowth\theta - (1-\theta)\BSDDecay)\tStar}}_{L^2(\StaticRegionWithExtension)}, \label{nonlinear:eq:BA-D:interpolation-conclusion}
  \end{align}
  where the last line follows by applying the bootstrap
  assumptions in \eqref{nonlinear:BA-E}, and \eqref{nonlinear:BA-D}.  Moreover, recall
  that we chose $\BSTopLvl - \BSLowLvl>2$, and
  $2\BSEGrowth < (\BSTopLvl - \BSLowLvl - 2)\BSDDecay$. Thus,
  \begin{equation}
    \label{nonlinear:eq:BA-D:interpolation-integrability}
    \BSEGrowth\theta - (1-\theta)\BSDDecay < 0,
  \end{equation}
  and plugging \eqref{nonlinear:eq:BA-D:interpolation-conclusion} and \eqref{nonlinear:BA-D}
  in to \eqref{nonlinear:eq:BA-D:aux1} and taking into account
  \eqref{nonlinear:eq:BA-D:interpolation-integrability}, we have that for $k\le \BSLowLvl$,
  \begin{equation}
    \label{nonlinear:eq:Decay:NL-BS}
    \int_{0}^{\TStar} e^{\BSDDecay\tStar}\norm{q(h_{\TStar} , \p h_{\TStar} , \p\p h_{\TStar} )}_{\HkWithT{k}(\Sigma_{\tStar})}\,d\tStar
    + \int_{\TStar}^{\TStar+1} e^{\BSDDecay\tStar}\norm{q_{\TStar}(h_{\TStar} , \p h_{\TStar} , \p\p h_{\TStar} )}_{\HkWithT{k}(\Sigma_{\tStar})}\,d\tStar
    \lesssim \BSConstant^2. 
  \end{equation}
  The higher-order nonlinear terms can be handled similarly.
  Recall from the definition of
  $D^{k,\BSDDecay,1}(\StaticRegionWithExtension)$ in
  \eqref{nonlinear:eq:D-kam-norm} that
  \begin{align*}
    &\norm*{\left(\NCal_{g_{b_{\TStar}}}^{\TStar}( h_{\TStar} , \p h_{\TStar} , \p\p h_{\TStar} ), h_0, h_1\right)}_{D^{k,\BSDDecay,1}(\StaticRegionWithExtension)}\\
    \lesssim{}&
                \norm*{\NCal_{g_{b_{\TStar}}}^{\TStar}( h_{\TStar} , \p h_{\TStar} , \p\p h_{\TStar} )}_{H^{k, \BSDDecay}(\StaticRegionWithExtension)}
                + \norm*{i_{b_{\TStar}, \phi_{\TStar}}(\InducedMetric_0, k_0)}_{\LSolHk{k+1}(\Sigma_0)}.
  \end{align*}
  Using \eqref{nonlinear:eq:Decay:NL-BS}, we then have that for $k\le \BSLowLvl$,
  \begin{equation*}
    e^{\BSDDecay\tStar}\norm*{h_{\TStar} }_{\HkWithT{k}(\Sigma_{\tStar})}\lesssim\varepsilon_0 + \BSConstant^2,
  \end{equation*}
  improving the bootstrap assumption \eqref{nonlinear:BA-D} as desired and concluding the
  proof of Proposition \ref{nonlinear:prop:BA-D:close}.
\end{proof}

\begin{remark}
  The fact that we have to estimate
  $\norm{h_{\TStar} }_{\HkWithT{\BSLowLvl+2}(\Sigma)}$ terms in
  \eqref{nonlinear:eq:BA-D:interpolation-conclusion} reflects two different
  losses of derivatives. The first is due to
  derivative loss in the Morawetz estimate in the underlying estimate, which is due
  to the presence of the trapped set in \KdS. The second derivative
  loss comes from treating the quasilinear terms as a forcing term on
  the right-hand side.
\end{remark}

\subsection{Improving the bootstrap assumption on energy}
\label{nonlinear:sec:bs-energy}

The goal of this section is to improve the high-regularity
exponential-growing bootstrap assumption in \eqref{nonlinear:BA-E}.
The main proposition is as follows. 
\begin{prop}
  \label{nonlinear:prop:BA-E:close}
  Fix some $\TStar>0$, for which there exist slowly-rotating \KdS{}
  black hole parameters $b_{\TStar}$ and a diffeomorphism
  $\phi_{\TStar} = e^{i_{\Theta}\vartheta_{\TStar}}$ such that letting
  $h_{\TStar}=h_{\TStar}(b,\vartheta,\InducedMetric_0, k_0)$ be the
  solution to the Cauchy problem in \eqref{nonlinear:eq:EVE:bootstrap-system},
  the bootstrap assumptions given in \eqref{nonlinear:BA-E}, \eqref{nonlinear:BA-D},
  \eqref{nonlinear:BA-G}, and the orthogonality condition \eqref{nonlinear:ORT-T} are
  satisfied. Then in fact, $h_{\TStar}$ satisfies the improved
  estimate
  \begin{equation}
    \label{nonlinear:eq:BA-E:improved}
    \norm*{e^{-\BSEGrowth \tStar} h_{\TStar}}_{H^j(\StaticRegionWithExtension)}
    \lesssim \varepsilon_0 + \BSConstant^2, \qquad j\le \BSTopLvl. 
  \end{equation}
\end{prop}

In order to prove \eqref{nonlinear:eq:BA-E:improved}, we need an estimate that
does not lose derivatives. To this end, we cannot rely on an estimate
on exact \KdS, and must use an estimate on a perturbation of
\KdS. This is where Proposition \ref{nonlinear:prop:BA-E:lossless-exp-grow-ILED}
comes in.

First, we write down the commuted system of equations.
\begin{lemma}
  \label{nonlinear:lemma:BA-E:commuted-eqn}
  Let $g = g_b+\tilde{g}$, where $g_b$ is a slowly-rotating \KdS{}
  metric. Moreover, let $h$ be a solution to the Cauchy problem given
  by
  \begin{align*}
    \mathbf{L}_gh &= f,\\
    \gamma_0(h) &= (h_0, h_1).
  \end{align*}
  Then for any multi-index $|\alpha|=k$,
  \begin{equation}
    \label{nonlinear:eq:BA-E:commuted-eqn}
    \mathbf{L}_g\RedShiftK^\alpha h = \RedShiftK^\alpha f + \tilde{f},
  \end{equation}
  where
  \begin{equation}
    \label{nonlinear:eq:BA-E:tilde-f-control}
    \norm*{\tilde{f}}_{L^2(\Sigma)}
    \lesssim
    \norm{\tilde{g}}_{H^3(\Sigma)}\norm{h}_{\HkWithT{k+1}(\Sigma)}
    + \norm{h}_{H^3(\Sigma)}\norm{\tilde{g}}_{\HkWithT{k+1}(\Sigma)}.    
  \end{equation}
\end{lemma}
\begin{proof}
  Recall from the definition of $\mathbf{L}_g$ in
  \eqref{nonlinear:eq:EVE:quasilinear-LinEinstein-def} that
  \begin{equation*}
    \mathbf{L}_g = \ScalarWaveOp[g] + \SubPOp_g + \PotentialOp_g,
  \end{equation*}
  where the coefficients of both $\SubPOp_g$ and $\PotentialOp_g$
  depend on at most one derivative of $g$. Then, commuting
  $\mathbf{L}_g$ and $\RedShiftK^\alpha$, we have that
  \begin{equation*}
    \left[\mathbf{L}_g, \RedShiftK^\alpha\right] = \sum_{|\beta|+|\gamma|\le k+2, 1\le |\beta|\le k+1} \tilde{k}_{\beta\gamma}(t,x)\p^\beta\tilde{g}\p^\gamma,
  \end{equation*}
  where $\tilde{k}_{\beta\gamma}$ are smooth bounded functions. 
  The conclusion then follows from the fact that $\Sigma$ is
  $3$-dimensional and an application of standard Sobolev embedding. 
\end{proof}

% It will also be convenient to have the following application of the
% basic energy estimate. 
% \begin{lemma}
%   \label{nonlinear:lemma:energy-estimate-bulk-trick}
%   Let $g = g_b+\tilde{g}$, where $g_b$ is a slowly-rotating \KdS{}
%   metric. Moreover, let $h$ be a solution to the Cauchy problem given
%   by
%   \begin{align*}
%     \LinEinsteinS_gh &= f,\\
%     \gamma_0(h) &= (h_0, h_1).
%   \end{align*}
%   Then,
%   \begin{equation*}
%     \sup_{\tStar\in [\TStar, \TStar+1]}\norm*{h}_{\HkWithT{1}(\Sigma_{\tStar})}
%     \lesssim \int_{\TStar-1}^{\TStar}\norm*{h}_{\HkWithT{1}(\Sigma_{\tStar})}\,d\tStar
%     + \int_{\TStar-1}^{\TStar+1}\norm*{f}_{L^2(\Sigma_[\tStar])}\,d\tStar.
%   \end{equation*}
% \end{lemma}
% \begin{proof}
  
% \end{proof}

Since at the level of improving the high-regularity bootstrap, we are
forced to work with estimates that do not lose derivatives, it will be
convenient to consider a slightly different semi-global extension.

\begin{lemma}
  \label{nonlinear:lemma:new-extension}
  Let $h_{\TStar}$ be the solution to the Cauchy problem in
  \eqref{nonlinear:eq:EVE:bootstrap-system}, and assume that the bootstrap
  assumptions given in \eqref{nonlinear:BA-E}, \eqref{nonlinear:BA-D}, \eqref{nonlinear:BA-G}, and
  the orthogonality condition \eqref{nonlinear:ORT-T} are satisfied.

  Then there exists some $\tilde{z}$ such that the solution
  $\mathfrak{h}_{\TStar}$ to the Cauchy problem 
  \begin{equation}
    \label{nonlinear:eq:new-extension:cauchy-problem}
    \begin{split}
      \LinEinsteinS_{\chi_{\TStar-1} g_{\TStar} + (1-\chi_{\TStar-1})g_{b_{\TStar}}}\mathfrak{h}_{\TStar}
      &= \chi_{\TStar-1}\mathcal{Q}_{g_{\TStar}}(\mathfrak{h}_{\TStar}, \p\mathfrak{h}_{\TStar})
        + \tilde{z}\\
      \evalAt*{(\mathfrak{h}_{\TStar}(\tStar, \cdot), \p_{\tStar}\mathfrak{h}_{\TStar}(\tStar, \cdot))}_{\tStar = 0} &= (h_0, h_1),
    \end{split}
  \end{equation}
  where $\chi_{\TStar-1}(\tStar)$ is a smooth cutoff function as
  specified in \eqref{nonlinear:eq:chi-tStar-def} and $g_{\TStar} := g_{b_{\TStar}}
  + h_{\TStar}$, satisfies
  \begin{equation}
    \label{nonlinear:eq:new-extension:main-control}
    \norm*{e^{-\BSEGrowth\tStar}\mathfrak{h}_{\TStar}}_{H^{\BSTopLvl}(\StaticRegionWithExtension)}
    \lesssim \varepsilon_0 + \BSConstant^2,
  \end{equation}
  and
  \begin{equation}
    \label{nonlinear:eq:new-extension:decay-control}
    \sup_{\tStar>0}e^{\BSDDecay\tStar}\norm*{\mathfrak{h}_{\TStar}}_{\HkWithT{\BSLowLvl}(\Sigma_{\tStar})}
    \lesssim \varepsilon_0 + \BSConstant^2.  
  \end{equation}
\end{lemma}

\begin{proof}
  The main idea will be to construct some $\tilde{z}$ supported on
  $[\TStar, \TStar+1]$ which depends nonlinearly on
  $\mathfrak{h}_{\TStar}$ in order to enforce exponential decay of
  $\mathfrak{h}_{\TStar}$ after time $\TStar$. We will then use the
  bootstrap assumptions to prove the desired control.

  We start with the construction of $\tilde{z}$.
  Let $\mathfrak{h}$ be the solution to
  \begin{equation*}
    \begin{split}
      \LinEinsteinS_{\chi_{\TStar-1} g_{\TStar} + (1-\chi_{\TStar-1})g_{b_{\TStar}}}\mathfrak{h}
      &= \chi_{\TStar-1}\mathcal{Q}_{g_{\TStar}}(\mathfrak{h}, \p\mathfrak{h})\\
      \evalAt*{(\mathfrak{h}(\tStar, \cdot), \p_{\tStar}\mathfrak{h}(\tStar, \cdot))}_{\tStar = 0} &= (h_0, h_1).
    \end{split}
  \end{equation*}
  We define $\mathfrak{N}(\mathfrak{h}, \p \mathfrak{h},
  \p\p \mathfrak{h})$ such that 
  \begin{equation}
    \label{nonlinear:eq:new-extension:semi-to-quasi}
    \mathfrak{N}(\mathfrak{h}, \p \mathfrak{h}, \p\p \mathfrak{h})
    = \LinEinsteinS_{\chi_{\TStar-1} g_{\TStar} + (1-\chi_{\TStar-1})g_{b_{\TStar}}} \mathfrak{h}
    - \LinEinstein_{g_{b_{\TStar}}}\mathfrak{h}
    - \chi_{\TStar-1}\mathcal{Q}_{g_{\TStar}}(\mathfrak{h}, \p \mathfrak{h}).
  \end{equation}
  Observe that by construction, 
  \begin{equation}
    \label{nonlinear:eq:new-extension:LinEinstein-diff-is-quasi}
    \norm*{\LinEinsteinS_{\chi_{\TStar-1} g_{\TStar} + (1-\chi_{\TStar-1})g_{b_{\TStar}}} \mathfrak{h}
    - \LinEinstein_{g_{b_{\TStar}}}\mathfrak{h}}_{\HkWithT{3}(\Sigma_{\tStar})}
    \lesssim
    \norm*{h_{\TStar}}_{\HkWithT{3}(\Sigma_{\tStar})}\norm*{\mathfrak{h}}_{\HkWithT{5}(\Sigma_{\tStar})}
    + \norm*{h_{\TStar}}_{\HkWithT{5}(\Sigma_{\tStar})}\norm*{\mathfrak{h}_{\TStar}}_{\HkWithT{3}(\Sigma_{\tStar})}
    .
  \end{equation}    
  Since $\mathfrak{h} = h_{\TStar}$ up until time $\TStar-1$, we
  observe that $\mathfrak{h}_{\TStar}-h_{\TStar}$ solves the system
  \begin{equation*}
    \begin{split}
      \LinEinstein_{g_{b_{\TStar}}}\left(\mathfrak{h}-h_{\TStar}\right)
      &= \mathfrak{N}(\mathfrak{h}, \p \mathfrak{h}, \p\p \mathfrak{h})
        -\Nonlinearity_{g_{b_{\TStar}}}^{\TStar}(h_{\TStar}, \p h_{\TStar}, \p\p h_{\TStar}), \\
      \evalAt*{(\mathfrak{h}_{\TStar}(\tStar, \cdot), \p_{\tStar}\mathfrak{h}_{\TStar}(\tStar, \cdot))}_{\tStar = \TStar-1} &= (0, 0).
    \end{split}    
  \end{equation*}  
  Now, let us consider the finite-dimensional family of
  compactly-supported\footnote{In particular, compactly supported in
    the $\tStar$ direction.} functions
  \begin{equation*}
    \mathcal{Z} :=
    \curlyBrace*{\LinEinstein_{g_{b_{\TStar}}}(\chi_{\TStar}(\tStar) \phi): \phi \in \QNMk{k}(\LinEinstein_{g_b}, \UpperHalfSpace)}.
  \end{equation*}
  The map in \eqref{nonlinear:eq:lambda-map:z-cal-map} is bijective just by
  dimension counting. All non-decaying asymptotic behavior can be
  eliminated by adding some forcing term in $\ZCal$ to the right hand
  side, while at the same time, $\ZCal$ is at most as large as the
  space of non-decaying $\LSolHk{k}$-quasinormal mode solutions. 
  % spanned by the basis given by
  % $\curlyBrace*{\LinEinstein_{g_b}(\chi_{\TStar}\psi_j):\psi_j \in
  %   \QNMk{k}(\LinEinstein_{g_b}, \Xi),
  %   \norm*{\psi_j}_{L^2(\DomainOfIntegration_{\TStar}^{\TStar+1})} =
  %   1}$.
  Then from Lemma \ref{nonlinear:coro:lambda-map}, we know that for\footnote{The
    $H^3$ regularity required has to do with the threshold regularity
    level required in the $\LSolHk{k}$-quasinormal spectrum in order
    to pick up exponential decay.} $k\ge 3$ there exists a continuous
  linear map
  \begin{gather*}
    \lambda_{\ZCal}: D^{k, \SpectralGap, 1}(\StaticRegionWithExtension)\to \ZCal,\\
    \lambda_{\ZCal}\left(\mathfrak{N}(\mathfrak{h}, \p \mathfrak{h}, \p\p \mathfrak{h}) - \Nonlinearity(h_{\TStar}, \p h_{\TStar}, \p\p h_{\TStar}), 0, 0 \right) \mapsto \tilde{z}
  \end{gather*}
  such that $\mathfrak{h}_{\TStar}$ solving
  \eqref{nonlinear:eq:new-extension:cauchy-problem} is exponentially decaying.
  
  Moreover, since $\ZCal$ is a finite-dimensional family of compactly
  supported functions, we have that for $4 \le k\le 6$,
  \begin{equation}
    \label{nonlinear:eq:new-extension:z-bound:preliminary}
    \norm*{e^{-\BSEGrowth\tStar}\tilde{z}}_{H^{k}(\StaticRegionWithExtension)}
    \lesssim \norm*{\mathfrak{N}(\mathfrak{h}, \p\mathfrak{h}, \p\p \mathfrak{h})}_{H^{3, \SpectralGap}([\TStar-1, +\infty)\times \Sigma)}
    +  \norm*{\Nonlinearity_{g_{b_{\TStar}}}^{\TStar}(h_{\TStar}, \p h_{\TStar}, \p\p h_{\TStar})}_{H^{3, \SpectralGap}([\TStar-1, +\infty)\times \Sigma)}.    
  \end{equation}
  We already have from Section \ref{nonlinear:sec:bs-decay} that
  \begin{equation*}
    \norm*{\Nonlinearity_{g_{b_{\TStar}}}^{\TStar}(h_{\TStar}, \p h_{\TStar}, \p\p h_{\TStar})}_{H^{3, \SpectralGap}(\StaticRegionWithExtension)}\lesssim \varepsilon_0 + \BSConstant^2. 
  \end{equation*}
  To bound the other term,
  we use the relationship in
  \eqref{nonlinear:eq:new-extension:LinEinstein-diff-is-quasi}, to write that
  \begin{align*}
    &\int_{\TStar-1}^{\TStar}e^{\BSDDecay\tStar}\norm*{
      \mathfrak{N}(\mathfrak{h}, \p \mathfrak{h}, \p\p\mathfrak{h})
    }_{\HkWithT{3}(\Sigma_{\tStar})}\,d\tStar\\
    \lesssim{}& \int_{\TStar-1}^{\TStar} e^{\BSDDecay\tStar}
                \left(\norm*{h_{\TStar}}_{\HkWithT{3}(\Sigma_{\tStar})}\norm*{\mathfrak{h}}_{\HkWithT{5}(\Sigma_{\tStar})}
    + \norm*{h_{\TStar}}_{\HkWithT{5}(\Sigma_{\tStar})}\norm*{\mathfrak{h}}_{\HkWithT{3}(\Sigma_{\tStar})}\right)\,d\tStar.
  \end{align*}
  Now using the mean value theorem, there exists some $t_{avg}\in
  [\TStar-2,\TStar-1]$ such
  that
  \begin{equation*}
    \norm*{\mathfrak{h}}_{\HkWithT{k}(\Sigma_{t_{avg}})} = \int_{\TStar-2}^{\TStar-1}\norm*{\mathfrak{h}}_{\HkWithT{k}(\Sigma_{\tStar})}\,d\tStar. 
  \end{equation*}
  Recall from local existence that
  \begin{equation}
    \label{nonlinear:eq:new-extension:local-existence}
    \sup_{\tStar\in [\TStar-1, \TStar]}e^{-\BSEGrowth\tStar}\norm*{\mathfrak{h}}_{\HkWithT{k}(\Sigma_{\tStar})} \lesssim e^{-\BSEGrowth t_{avg}}\norm*{\mathfrak{h}}_{\HkWithT{k}(\Sigma_{t_{avg}})}
    \lesssim \norm*{e^{-\BSEGrowth\tStar}\mathfrak{h}}_{H^k(\StaticRegionWithExtension)}.
  \end{equation}
  Using the interpolation argument in
  \eqref{nonlinear:eq:BA-D:interpolation-conclusion}, we have that
  \begin{equation*}
    \int_{\TStar-1}^{\TStar+1}\norm*{\mathfrak{h}}_{\HkWithT{5}(\Sigma_{\tStar})}\,d\tStar
    \lesssim \sup_{\tStar\in [\TStar-1,\TStar+1]}e^{\BSDDecay\tStar}\norm*{\mathfrak{h}}_{\HkWithT{3}(\Sigma_{\tStar})}\norm*{e^{-\BSEGrowth\tStar}\mathfrak{h}}_{H^6([\TStar-1, \TStar+1)\times \Sigma)}.
  \end{equation*}
  Then the fact that
  $h_{\TStar}(t_{avg}, \cdot) = \mathfrak{h}(t_{avg},
  \cdot)$, and the bootstrap assumptions for $h_{\TStar}$, we have
  that
  \begin{equation}
    \label{nonlinear:eq:new-extension:T-1-to-T:decay:nonlinear-control}
    \abs*{\int_{\TStar-1}^{\TStar}e^{\BSDDecay\tStar}\norm*{
        \mathfrak{N}(\mathfrak{h}, \p \mathfrak{h}, \p\p\mathfrak{h})
      }_{\HkWithT{3}(\Sigma_{\tStar})}\,d\tStar}
    \lesssim \varepsilon_0 + \BSConstant^2. 
  \end{equation}
  Plugging this back into
  \eqref{nonlinear:eq:new-extension:z-bound:preliminary}, we have that,
  \begin{equation}
    \label{nonlinear:eq:new-extension:z-bound:in-proof}
    \norm*{e^{-\BSEGrowth\tStar}\tilde{z}}_{H^6(\StaticRegionWithExtension)} \lesssim \varepsilon_0 + \BSConstant^2,
  \end{equation}
  as desired.

  We now prove \eqref{nonlinear:eq:new-extension:decay-control}. To this end, we
  rewrite \eqref{nonlinear:eq:new-extension:cauchy-problem} using
  \eqref{nonlinear:eq:new-extension:semi-to-quasi} as
  \begin{equation*}
    \begin{split}
      \LinEinstein_{g_{b_{\TStar}}} \mathfrak{h}_{\TStar} &= \mathfrak{N} (\mathfrak{h}_{\TStar}, \p \mathfrak{h}_{\TStar}, \p\p \mathfrak{h}_{\TStar}) + \tilde{z},\\
      \evalAt*{(\mathfrak{h}_{\TStar}(\tStar, \cdot), \p_{\tStar}\mathfrak{h}_{\TStar}(\tStar, \cdot))}_{\tStar = 0} &= (h_0, h_1),
    \end{split}
  \end{equation*}
  where $\mathfrak{N}$ vanishes after time $\TStar$.
  Then, we can use \eqref{nonlinear:eq:linear-decay-statement} on
  $\mathfrak{h}_{\TStar}$ to deduce that
  \begin{align*}
    &\sup_{\tStar>0} e^{\BSDDecay\tStar}\norm*{\mathfrak{h}_{\TStar}}_{\HkWithT{3}(\Sigma_{\tStar})}\\
    \lesssim{}& \int_{\Real^+} e^{\BSDDecay\tStar}
                \norm*{\mathfrak{N}\left(\mathfrak{h}_{\TStar}, \p \mathfrak{h}_{\TStar}, \p\p \mathfrak{h}_{\TStar}\right)}_{\HkWithT{3}(\Sigma_{\tStar})}
                \,d\tStar
                + \int_{\Real^+} e^{\BSDDecay\tStar}
                \norm*{\tilde{z}}_{\HkWithT{3}(\Sigma_{\tStar})}
                \,d\tStar
                + \varepsilon_0
                .
  \end{align*}
  Until time $\tStar = \TStar-1$,
  \begin{equation*}
    \mathfrak{N}(\mathfrak{h}_{\TStar}, \p \mathfrak{h}_{\TStar}, \p\p\mathfrak{h}_{\TStar})
    = \Nonlinearity_{g_{b_{\TStar}}}(h, \p h, \p\p h), \qquad
    h_{\TStar} = \mathfrak{h}_{\TStar},
  \end{equation*}
  so from Proposition \ref{nonlinear:prop:BA-D:close}, we have that
  \begin{align*}
    \int_{\Real^+} e^{\BSDDecay\tStar}
    \norm*{\mathfrak{N}\left(\mathfrak{h}_{\TStar}, \p \mathfrak{h}_{\TStar}, \p\p \mathfrak{h}_{\TStar}\right)}_{\HkWithT{3}(\Sigma_{\tStar})}
    ={}& \int_{0}^{\TStar-1}e^{\BSDDecay\tStar}\norm*{
      \Nonlinearity_{g_{b_{\TStar}}}^{\TStar}(h_{\TStar}, \p h_{\TStar}, \p\p h_{\TStar})
    }_{\HkWithT{3}(\Sigma_{\tStar})}\,d\tStar\\
    &+ \int_{\TStar-1}^{\TStar}e^{\BSDDecay\tStar}\norm*{
      \mathfrak{N}(\mathfrak{h}_{\TStar}, \p \mathfrak{h}_{\TStar}, \p\p\mathfrak{h}_{\TStar})
    }_{\HkWithT{3}(\Sigma_{\tStar})}\,d\tStar.
  \end{align*}
  The first term we already know from Proposition
  \ref{nonlinear:prop:BA-D:close} is controlled by
  $\varepsilon_0+\BSConstant^2$.
  To control the second term, we use
  \eqref{nonlinear:eq:new-extension:T-1-to-T:decay:nonlinear-control}, observing
  that $\mathfrak{h}_{\TStar} = \mathfrak{h}$ for $\tStar\le \TStar$. 
  
  We now move onto proving the control in
  \eqref{nonlinear:eq:new-extension:main-control}, we apply Proposition
  \ref{nonlinear:prop:BA-E:lossless-exp-grow-ILED} commuted equation for
  $\RedShiftK^\alpha \mathfrak{h}_{\TStar}$ with
  $|\alpha|=k, k\le \BSTopLvl-1$ in \eqref{nonlinear:eq:BA-E:commuted-eqn} to
  $\mathfrak{h}_{\TStar}$ using
  $\DomainOfIntegration = \StaticRegionWithExtension$. Since we
  constructed $\tilde{z}$ such that $\mathfrak{h}_{\TStar}$ is
  exponentially decaying, this yields
  \begin{align*}
    \norm*{e^{-\BSEGrowth\tStar}\RedShiftK^\alpha\mathfrak{h}_{\TStar}}_{H^1(\StaticRegionWithExtension)}
    \lesssim{}& \varepsilon_0
                + \int_{0}^{\TStar}e^{-\BSEGrowth\tStar}\norm*{\RedShiftK^\alpha\mathcal{Q}_{g_{\TStar}}(\mathfrak{h}_{\TStar}, \p\mathfrak{h}_{\TStar})}_{L^2(\Sigma_{\tStar})}\,d\tStar\\
              & + \norm*{e^{-\BSEGrowth\tStar}\tilde{z}}_{H^{\BSTopLvl}(\StaticRegionWithExtension)}
                + \norm*{e^{-\BSEGrowth\tStar}\tilde{f}}_{L^2(\StaticRegionWithExtension)}
                + \norm*{e^{-\BSEGrowth\tStar}\RedShiftK^\alpha\mathfrak{h}_{\TStar}}_{L^2(\StaticRegionWithExtension)}. 
  \end{align*}
  We consider each of the terms on the right-hand side in turn.

  Recall from the definition of $\mathcal{Q}_{g_{\TStar}}$ in
  \eqref{nonlinear:eq:EVE:quasilinear-system:energy} that
  $\mathcal{Q}_{g_{\TStar}}(\mathfrak{h}_{\TStar} , \p
  \mathfrak{h}_{\TStar} )$ is at least quadratic. Thus, there exist
  nonlinear functions $q(\cdot, \cdot)$ and $\tilde{q}(\cdot, \cdot)$
  such that
  \begin{equation*}
    \mathcal{Q}_{g_{\TStar}}(\mathfrak{h}_{\TStar} , \p \mathfrak{h}_{\TStar} ) = q(\mathfrak{h}_{\TStar} , \p \mathfrak{h}_{\TStar} ) + \tilde{q}(\mathfrak{h}_{\TStar} , \p \mathfrak{h}_{\TStar} ),
  \end{equation*}
  where $q(\cdot, \cdot)$ is quadratic in its arguments, and
  $\tilde{q}(\cdot, \cdot)$ is at least cubic. We first address the
  quadratic terms as they are the main difficulty. Recalling that
  $H^k(\Sigma_{\tStar})$ is a calculus for $k > \frac{3}{2}$, we have
  that
  \begin{equation*}
    \norm*{q(\mathfrak{h}_{\TStar} , \p \mathfrak{h}_{\TStar} )}_{\HkWithT{k}(\Sigma_{\tStar})}
    \lesssim \norm{\mathfrak{h}_{\TStar} }_{\HkWithT{\BSTopLvl}(\Sigma_{\tStar})}\norm{\mathfrak{h}_{\TStar} }_{\HkWithT{3}(\Sigma_{\tStar})},
  \end{equation*}
  where the appearance of $\HkWithT{3}(\Sigma_{\tStar})$ comes from
  the Sobolev inequalities.

  Recalling the bootstrap assumptions \eqref{nonlinear:BA-E} and
  \eqref{nonlinear:BA-D}, and the fact that $h_{\TStar} = \mathfrak{h}_{\TStar}$
  up until time $\TStar-1$, we then have that
  \begin{equation*}
    \int_0^{\TStar-1}e^{-\BSEGrowth\tStar}\norm*{q(\mathfrak{h}_{\TStar} , \p \mathfrak{h}_{\TStar})}_{\HkWithT{k}(\Sigma_{\tStar})}\,d\tStar
    \lesssim \BSConstant^2.
  \end{equation*}
  On the other hand, we can control
  \begin{align}
    &\int_{\TStar-1}^{\TStar}e^{-\BSEGrowth\tStar}\norm*{q(\mathfrak{h}_{\TStar} , \p \mathfrak{h}_{\TStar})}_{\HkWithT{k}(\Sigma_{\tStar})}\,d\tStar  \notag  \\
    \lesssim{}& \sup_{\tStar\in [\TStar-1, \TStar]}e^{\BSDDecay\tStar}\norm*{\mathfrak{h}_{\TStar}}_{\HkWithT{3}(\Sigma_{\tStar})}
                \int_{\TStar-1}^{\TStar}e^{-(\BSDDecay-\BSEGrowth)\tStar}e^{-\BSEGrowth\tStar}\norm*{\mathfrak{h}_{\TStar}}_{\HkWithT{k}(\Sigma_{\tStar})}\,d\tStar\notag \\
    \lesssim{}& \BSConstant \norm*{e^{-\BSEGrowth\tStar}\mathfrak{h}_{\TStar}}_{H^k(\DomainOfIntegration)}.
                % \int_{\TStar-1}^{\TStar}e^{-\BSEGrowth\tStar}\norm*{\mathfrak{h}_{\TStar}}_{\HkWithT{k}(\Sigma_{\tStar})}\,d\tStar. 
                \label{nonlinear:eq:new-extension:q-control:aux}
                % \left(
                % \int_{\TStar-1}^{\TStar}\norm*{\mathfrak{h}_{\TStar}}_{\HkWithT{3}(\Sigma_{\tStar})}
                % + \int_{\TStar-1}^{\TStar}e^{-\BSEGrowth\tStar}\norm*{\QCal_{g_{\TStar}}(\mathfrak{h}_{\TStar}, \p \mathfrak{h}_{\TStar})}_{\HkWithT{2}(\Sigma_{\tStar})}\,d\tStar
                % \right),   
  \end{align}
  At this point, we again use the mean value theorem and the local
  existence theory to bound the remaining term by an integral of
  $h_{\TStar}$.  Plugging \eqref{nonlinear:eq:new-extension:local-existence}
  into \eqref{nonlinear:eq:new-extension:q-control:aux} yields that
  \begin{equation*}
    \int_{\TStar-1}^{\TStar}e^{-\BSEGrowth\tStar}\norm*{q(\mathfrak{h}_{\TStar} , \p \mathfrak{h}_{\TStar})}_{\HkWithT{k}(\Sigma_{\tStar})}\,d\tStar\lesssim \BSConstant^2. 
  \end{equation*}

  The higher-order nonlinear terms in
  $\tilde{q}(h ,\p h )$ can be handled similarly
  using the calculus inequality and the bootstrap assumptions
  \eqref{nonlinear:BA-E} and \eqref{nonlinear:BA-D} to
  show that
  \begin{equation*}
    \int_0^{\TStar}e^{-\BSEGrowth\tStar}\norm{\tilde{q}(\mathfrak{h}_{\TStar} , \p \mathfrak{h}_{\TStar} )}_{\HkWithT{k}(\Sigma_{\tStar})}\,d\tStar
    \lesssim \BSConstant^2,\qquad k\le \BSTopLvl-1.
  \end{equation*}

  Again using the the bootstrap assumptions \eqref{nonlinear:BA-E} and
  \eqref{nonlinear:BA-D}, the fact that $h_{\TStar} = \mathfrak{h}_{\TStar}$
  until time $\TStar-1$, and \eqref{nonlinear:eq:new-extension:local-existence},
  we also have from \eqref{nonlinear:eq:BA-E:tilde-f-control} that
  \begin{equation*}
    \norm*{e^{-\BSEGrowth\tStar}\tilde{f}}_{L^2(\StaticRegionWithExtension)}
    \lesssim \BSConstant^2. 
  \end{equation*}
  Finally, to handle the lower-order control, we observe that we can
  argue iteratively from $k=4$, since we have already proven the
  stronger lower-order control in
  \eqref{nonlinear:eq:new-extension:decay-control}.
\end{proof}

We now move onto the proof of the main proposition. 
\begin{proof}[Proof of Proposition \ref{nonlinear:prop:BA-E:close}]
  Recall from \eqref{nonlinear:eq:h:def} that we can decompose
  $h_{\TStar} = \chi_{\TStar}(\tStar)(h-\widetilde{h}_{\TStar}) +
  \widetilde{h}_{\TStar}$ where $h$ solves the full nonlinear Einstein
  vacuum equations in harmonic gauge, and $\widetilde{h}_{\TStar}$
  solves the Cauchy problem in \eqref{nonlinear:eq:tilde-h:Cauchy-problem}.

  To improve \eqref{nonlinear:BA-E}, we observe that
  \begin{equation*}
    \norm*{h_{\TStar}}_{H^k(\StaticRegionWithExtension)}
    \lesssim \norm*{\mathfrak{h}_{\TStar}}_{H^k(\StaticRegionWithExtension)}
    + \norm*{h_{\TStar}}_{H^k(\DomainOfIntegration_{\tStar\ge \TStar-1})}, 
  \end{equation*}
  where
  \begin{equation*}
    \DomainOfIntegration_{\tStar\ge \TStar-1}:=[\TStar-1, +\infty)\times\Sigma,
  \end{equation*}
  and $\mathfrak{h}_{\TStar}$ is as constructed in Lemma
  \ref{nonlinear:lemma:new-extension}. But from Lemma \ref{nonlinear:lemma:new-extension}
  we already have that
  \begin{equation}
    \label{nonlinear:eq:Energy:full-nonlinear}
    \norm*{\mathfrak{h}_{\TStar}}_{H^k(\StaticRegionWithExtension)} \lesssim \varepsilon_0 + \BSConstant^2. 
  \end{equation}
  Thus, it suffices to control
  $\norm*{h_{\TStar}}_{H^k(\DomainOfIntegration_{\tStar\ge
      \TStar-1})}$.  To this end, we improve \eqref{nonlinear:BA-E} for
  $\tStar>\TStar-1$ by controlling $h$ on $[\TStar-1, \TStar+1]$, and
  controlling $\widetilde{h}_{\TStar}$ on $[\TStar, +\infty)$. We
  begin with estimating $h $ on $[\TStar-1, \TStar+1]$. Using the mean
  value theorem, we define $t_{avg}'$ such that
  \begin{equation*}
    \norm*{h_{\TStar}}_{\HkWithT{k}(\Sigma_{t_{avg}'})} = \int_{\TStar-2}^{\TStar-1}\norm*{h_{\TStar}}_{\HkWithT{k}(\Sigma_{\tStar})}\,d\tStar. 
  \end{equation*}
  Then we have from local existence theory that for
  $3<k\le \BSTopLvl$,
  \begin{equation}
    \label{nonlinear:eq:Energy:local-existence-estimate}
    \sup_{\tStar\in [\TStar-1, \TStar+1]} e^{-\BSEGrowth\tStar}\norm*{h}_{\HkWithT{k}(\Sigma_{\tStar})}
    \lesssim e^{-\BSEGrowth t_{avg}'}\norm*{h_{\TStar}}_{\HkWithT{k}(\Sigma_{t_{avg}'})}
    \lesssim \varepsilon_0 + \BSConstant^2,
  \end{equation}
  where the last inequality follows from
  \eqref{nonlinear:eq:new-extension:main-control} using the fact that on
  $\tStar\le \TStar-1$, $\mathfrak{h}_{\TStar} = h_{\TStar}$.

  Furthermore, using Proposition
  \ref{nonlinear:prop:BA-E:lossless-exp-grow-ILED}, we have that
  \begin{equation*}
    \sup_{\tStar > \TStar}e^{-\BSEGrowth (\tStar-\TStar)}\norm*{\widetilde{h}_{\TStar}}_{\HkWithT{k}(\Sigma_{\tStar})}
    \lesssim 
    \norm*{h_{\TStar} }_{\HkWithT{k}(\Sigma_{\TStar})}
    + \norm*{\widetilde{h}_{\TStar} - h_{\TStar} }_{\HkWithT{k}(\Sigma_{\TStar})}.
  \end{equation*}
  By the construction of $\iota_{b,\tStar}$ in Proposition
  \ref{nonlinear:prop:Bootstrap:iota-construction:init-data}, we also have
  that
  $\left.\widetilde{h}_{\TStar} - h \right\vert_{\Sigma_{\TStar}}$ vanishes
  quadratically at $h =0$. Thus, following similar reasoning
  as above and using the bootstrap assumptions \eqref{nonlinear:BA-E} and
  \eqref{nonlinear:BA-D} we have that for $k\le \BSTopLvl$,
  $\tStar \ge \TStar$,
  \begin{align}
    \sup_{\tStar > \TStar}e^{-\BSEGrowth \tStar}\norm*{\widetilde{h}_{\TStar}}_{\HkWithT{k}(\Sigma_{\tStar})}
    &\lesssim e^{-\BSEGrowth\TStar}\left(\norm*{h_{\TStar} }_{\HkWithT{k}(\Sigma_{\TStar})} + \norm*{\widetilde{h}_{\TStar} - h_{\TStar} }_{\HkWithT{k}(\Sigma_{\TStar})}\right)\notag \\
    &\lesssim e^{-\BSEGrowth\TStar}\norm*{h_{\TStar} }_{\HkWithT{k}(\Sigma_{\TStar})} + \BSConstant^2e^{-\BSDDecay\TStar}\notag \\
    &\lesssim \varepsilon_0 + \BSConstant^2, \label{nonlinear:eq:Energy:linear-estimate}
  \end{align}
  where the last inequality follows from
  \eqref{nonlinear:eq:Energy:local-existence-estimate}. 

  Combining \eqref{nonlinear:eq:Energy:full-nonlinear},
  \eqref{nonlinear:eq:Energy:local-existence-estimate}, and
  \eqref{nonlinear:eq:Energy:linear-estimate} concludes the proof of Proposition
  \ref{nonlinear:prop:BA-E:close}.
\end{proof}

\subsection{Improving bootstrap assumption on gauge}
\label{nonlinear:sec:bs-gauge}

In this section, we improve the bootstrap assumption on the gauge,
\eqref{nonlinear:BA-G}. The main proposition is as follows.
\begin{prop}
  \label{nonlinear:prop:BA-gauge-improve}
  Let $\TStar>0$ be some bootstrap time for which there exist
  $b_{\TStar}, \vartheta_{\TStar}, h_{\TStar}, \TStar$ such that
  the bootstrap assumptions \eqref{nonlinear:BA-G}, \eqref{nonlinear:ORT-T}, \eqref{nonlinear:BA-E},
  \eqref{nonlinear:BA-D} are satisfied. Then in fact 
  $b_{\TStar},\vartheta_{\TStar}$ satisfy the improved estimate
  \begin{equation}
    \label{nonlinear:eq:BA-G:improved}
    \abs*{\vartheta_{\TStar}} + \abs*{b_{\TStar}-b^0} \lesssim \varepsilon_0 + \BSConstant^2.
  \end{equation}
\end{prop}

To prove Proposition \ref{nonlinear:prop:BA-gauge-improve}, we use the following
application of the implicit function theorem. 
\begin{lemma}
  \label{nonlinear:lemma:IFT-app:fixed-T}
  Define
  \begin{gather}
    \label{nonlinear:eq:IFT-app:fixed-T:F-def}    
    \IFTNonlinearQuantity: \Real^4\times \Real^{N_{\Theta}} \times
    H^{\BSTopLvl}(\Sigma_0)\times H^{\BSTopLvl-1}(\Sigma_0)    
    \times H^{\BSLowLvl,\BSDDecay}(\StaticRegionWithExtension)
    \to \Real^{4+N_{\Theta}}\\
    \IFTNonlinearQuantity(b, \vartheta, \InducedMetric, k, f)= \lambda_{\Upsilon}[g_b](f, i_{b, \phi}(\InducedMetric, k)),   
  \end{gather}
  where $\phi=e^{i_{\Theta}\vartheta}$, and $h_{\TStar}$ is the solution to the
  Cauchy problem in \eqref{nonlinear:eq:EVE:bootstrap-system}. Then,
  if $(b_\star, \vartheta_\star, \InducedMetric_\star, k_\star)\in \Real^4\times \Real^{N_{\Theta}} \times
  H^{\BSTopLvl}(\Sigma_0)\times H^{\BSTopLvl-1}(\Sigma_0) $ is
  chosen such that
  \begin{equation*}
    \IFTNonlinearQuantity(b_\star, \vartheta_\star,\InducedMetric_\star, k_\star, 0)=0
  \end{equation*}
  and moreover,
  \begin{equation}
    \label{nonlinear:eq:IFT-app:fixed-T:f-condition}
    \left.D_bf\right|_{\left(b_\star, \vartheta_\star, \InducedMetric_\star, k_\star, 0\right)} = 0,\qquad
    \left.D_\vartheta f\right|_{\left(b_\star, \vartheta_\star, \InducedMetric_\star, k_\star, 0\right)} = 0,
  \end{equation}
  then there exists some $\delta>0$ such that defining
  \begin{equation*}
    \mathcal{X}_\delta \subset H^{\BSTopLvl}(\Sigma_0)\times H^{\BSTopLvl-1}(\Sigma_0)
    \times  H^{\BSLowLvl,\BSDDecay}(\StaticRegionWithExtension)    
  \end{equation*}
  such that $(\InducedMetric, k,\TStar)\in \mathcal{X}_\delta$ if and
  only if
  \begin{equation*}
    \norm{\InducedMetric - \InducedMetric_{b_\star}}_{H^{\BSTopLvl}(\Sigma_0)}
    + \norm{k-k_{b_\star}}_{H^{\BSTopLvl-1}(\Sigma_0)}
     + \norm{f}_{ H^{\BSLowLvl,\BSDDecay}(\StaticRegionWithExtension)}    
    < \delta,
  \end{equation*}
  there exists a function
  \begin{equation}
    \label{nonlinear:eq:IFT-app:fixed-T:G-def} 
    \IFTInverse(\InducedMetric, k, f): \mathcal{X}_\delta \to  \Real^4\times\Real^{N_{\Theta}},
  \end{equation}
  which is well-defined and $C^1$ in its arguments on $\mathcal{X}_\delta$,
  and moreover, for $(\InducedMetric, k, f)\in \mathcal{X}_\delta$,
  \begin{equation*}
    \IFTNonlinearQuantity(b,\vartheta, \InducedMetric, k, f) = 0,\quad \iff\quad (b, \vartheta) = \IFTInverse(\InducedMetric, k, f).
  \end{equation*}
\end{lemma}
\begin{proof}
  Recalling the definition of $(g_{b}')^\Upsilon(b')$ in
  \eqref{nonlinear:eq:g-linearized-wave-gauge:def}, observe that we can calculate
  \begin{align*}
    D_bi_{b,\phi}(\InducedMetric, k)\vert_{(b_\star, \vartheta_\star, \InducedMetric_{\star}, k_{\star}, 0)}(b', \vartheta')
    &= \gamma_0\left((g_{b_\star}')^\Upsilon(b')\right),\\
    D_{\vartheta}i_{b,\phi}(\InducedMetric, k)\vert_{(b_\star, \vartheta_\star, \InducedMetric_{\star}, k_{\star}, 0)}(b', \vartheta')
    &= \gamma_0\left(\nabla_{g_{b_\star}}\otimes i_{\Theta}\vartheta'\right).
  \end{align*}
  As a result,
  \begin{equation*}
    D_{b,\vartheta}\IFTNonlinearQuantity = \lambda_{\Upsilon}[g_{b_\star}]\left(0, \gamma_0\left((g_{b_\star}')^\Upsilon(b') + \nabla_{g_{b_\star}}\otimes i_{\Theta}\vartheta' \right)\right), 
  \end{equation*}
  is an isomorphism by Proposition \ref{nonlinear:prop:lambda-isomorphism}. The
  conclusion then follows from a direct application of the implicit
  function theorem in Theorem \ref{nonlinear:thm:IFT}.
\end{proof}

We are now ready to move onto the proof of Proposition
\ref{nonlinear:prop:BA-gauge-improve}. 
\begin{proof}[Proof of Proposition \ref{nonlinear:prop:BA-gauge-improve}]

  We will apply Lemma \ref{nonlinear:lemma:IFT-app:fixed-T} with the map
  \begin{gather*}    
    \mathcal{F}_{\TStar}: \Real^4 \times \Real^{N_{\Theta}} \times H^{\BSTopLvl}(\Sigma_0) \times H^{\BSTopLvl-1}(\Sigma_0)\times H^{\BSLowLvl,\BSDDecay}(\StaticRegionWithExtension)\to \Real^{4+N_{\Theta}},\\
    \mathcal{F}_{\TStar}(b, \vartheta, \InducedMetric, k, f)=  \lambda_{\Upsilon}[g_b](f, i_{b, \phi}(\InducedMetric, k)),
  \end{gather*}
  where $\phi=e^{i_{\Theta}\vartheta}$, around the choice
  \begin{equation*}
    (b, \vartheta, \InducedMetric, k, f) = (b^0, 0, \InducedMetric_{b^0}, k_{b^0}, 0).
  \end{equation*}
  Observe that 
  \begin{equation*}
    \left.D_b f\right\vert_{(b^0, 0, \InducedMetric_{{b^0}}, k_{b^0}, 0)} = 0,\qquad
    \left.D_\vartheta f\right\vert_{(b^0, 0, \InducedMetric_{{b^0}}, k_{b^0}, 0)} = 0.
  \end{equation*}
  Then, we can apply Lemma \ref{nonlinear:lemma:IFT-app:fixed-T} to see that there
  exists some $C^1$ function
  \begin{equation*}
    \IFTInverse_{\TStar}: H^{\BSTopLvl}(\Sigma_0) \times H^{\BSTopLvl-1}(\Sigma_0)\times H^{\BSLowLvl,\BSDDecay}(\StaticRegionWithExtension) \to \Real^4 \times \Real^{N_{\Theta}}, 
  \end{equation*} and some $\delta_{\TStar}>0$ such that for
  \begin{equation}
    \label{nonlinear:eq:BA-G:IFT:smallness-condition}
    \norm{\InducedMetric - \InducedMetric_{b^0}}_{H^{\BSTopLvl}(\Sigma_0)}
    + \norm{k - k_{b^0}}_{H^{\BSTopLvl-1}(\Sigma_0)}
    + \norm{f}_{H^{\BSLowLvl,\BSDDecay}(\StaticRegionWithExtension)}
    < \delta_{\TStar},
  \end{equation}
  we have that
  \begin{equation*}
    \IFTNonlinearQuantity_{\TStar}(b, \vartheta, \InducedMetric, k, f) = 0 \quad \iff \quad
    (b, \vartheta) = \IFTInverse_{\TStar}(\InducedMetric, k, f).
  \end{equation*}
  Then, we can pick $(\InducedMetric, k, f) = \left(\InducedMetric_0,
    k_0, \NCal_{g_{b_{\TStar}}}^{\TStar}(h_{\TStar},
  \p h_{\TStar}, \p\p h_{\TStar})\right)$. Recall that we have already shown
  in \eqref{nonlinear:eq:Decay:NL-BS} that
  \begin{equation*}
    \norm*{\NCal_{g_{b_{\TStar}}}^{\TStar}(h_{\TStar},
      \p h_{\TStar}, \p\p h_{\TStar})}_{H^{\BSLowLvl,\BSDDecay}(\StaticRegionWithExtension)}\lesssim
    \BSConstant^2.
  \end{equation*}
  Thus, taking $\varepsilon_0$, $\BSConstant$ sufficiently small, we
  can ensure that \eqref{nonlinear:eq:BA-G:IFT:smallness-condition} is satisfied
  for
  $(\InducedMetric, k, f) = \left(\InducedMetric_0, k_0,
    \NCal_{g_{b_{\TStar}}}^{\TStar}(h_{\TStar}, \p h_{\TStar}, \p\p
    h_{\TStar})\right)$.

  Since we already know that
  \begin{equation*}
    \IFTNonlinearQuantity_{\TStar}\left(b_{\TStar}, \vartheta_{\TStar}, \InducedMetric_0, k_0,\NCal_{g_{b_{\TStar}}}^{\TStar}(h_{\TStar},
      \p h_{\TStar}, \p\p h_{\TStar})\right)=0,
  \end{equation*}
  we have in fact that 
  \begin{equation*}
    (b_{\TStar}, \vartheta_{\TStar}) =
    \IFTInverse_{\TStar}\left(\InducedMetric_0, k_0, \NCal_{g_{b_{\TStar}}}^{\TStar}(h_{\TStar},
      \p h_{\TStar}, \p\p h_{\TStar})\right). 
  \end{equation*}
  We then have the following estimate
  \begin{align*}
    &\abs*{b_{\TStar} - b^0} + \abs*{\vartheta_{\TStar}}\\
    \lesssim{}& \norm*{\InducedMetric_0 - \InducedMetric_{b^0}}_{H^{\BSTopLvl}(\Sigma_0)}
                + \norm*{k_0 - k_{b^0}}_{H^{\BSTopLvl-1}(\Sigma_0)}
                + \norm*{\NCal_{g_{b_{\TStar}}}^{\TStar}(h_{\TStar}, \p h_{\TStar}, \p\p h_{\TStar})}_{H^{\BSLowLvl,\BSDDecay}(\StaticRegionWithExtension)}\\
    \lesssim{}& \varepsilon_0 +\BSConstant^2,
  \end{align*}
  where the first inequality follows from the fact that
  $\IFTInverse_{\TStar}$ is $C^1$ in its arguments and the second
  follows from the smallness of the initial data in \eqref{nonlinear:eq:main:init-data-smallness} and from the
  improvements to the bootstrap assumptions in
  \eqref{nonlinear:eq:BA-D:improved} and \eqref{nonlinear:eq:BA-E:improved}.
\end{proof}

\subsection{Extending the bootstrap time}
\label{nonlinear:sec:bs-ext}

We will also need to extend the bootstrap gauge. We formalize this in
the following proposition.
\begin{prop}
  \label{nonlinear:prop:BA-gauge-extend}
  Assume that the orthogonality condition in \eqref{nonlinear:ORT-T}, bootstrap
  assumptions \eqref{nonlinear:BA-E}, \eqref{nonlinear:BA-D} and \eqref{nonlinear:BA-G} all hold
  with bootstrap time $\TStar >0 $. Then there exists some $\deltaExt >0$
  sufficiently small, slowly rotating \KdS{} black-hole
  parameters $b_{\TStar + \deltaExt}$ and a diffeomorphism
  $\phi_{\TStar+\deltaExt}=e^{i_{\Theta}\vartheta_{\TStar+\deltaExt}}$  
  such that 
  \begin{equation*}
    h_{\TStar+\deltaExt}:=h_{\TStar+\deltaExt}(b_{\TStar+\deltaExt}, \vartheta_{\TStar+\deltaExt}, \InducedMetric_0, k_0)
  \end{equation*}
  satisfies the
  orthogonality condition
  \begin{equation}
    \label{nonlinear:eq:BA-extend:orthogonality-condition}
    \lambda_{\Upsilon}\left[g_{b_{\TStar+\deltaExt}}\right]\left({\NCal}^{\TStar+\deltaExt}_{g_{b_{\TStar + \deltaExt}}}(h_{\TStar+\deltaExt}), i_{b_{\TStar + \deltaExt},\phi_{\TStar + \deltaExt}}(\InducedMetric_0, k_0)\right) = 0.
  \end{equation}
  Moreover, $h_{\TStar + \deltaExt}$ satisfies the bootstrap
  assumptions \eqref{nonlinear:BA-E} and \eqref{nonlinear:BA-D}, and
  $\vartheta_{\TStar+\deltaExt}$ and $b_{\TStar+\deltaExt}$ satisfy
  \eqref{nonlinear:BA-G}.
\end{prop}

We first prove the following auxiliary lemma. 
\begin{lemma}
  \label{nonlinear:lemma:BA-extend:nbhd-construction}
  Under the same assumptions as in Proposition
  \ref{nonlinear:prop:BA-gauge-extend}, there exists $\delta_* >0$ sufficiently
  small, and a neighborhood
  $(b_{\TStar}, \vartheta_{\TStar})\in \mathcal{B}_{ext}\subset
  \Real^4\times \Real^{N_{\Theta}}$ such that for any
  $(b,\vartheta)\in \mathcal{B}_{ext}$, and $\delta<\delta_*$ the
  solution
  \begin{equation*}
    h_{\TStar+\delta}(b, \vartheta, \InducedMetric_0, k_0)
    := \chi_{\TStar+\delta}h(b,\vartheta,\InducedMetric_0, k_0)
    + (1-\chi_{\TStar+\delta})\widetilde{h}_{\TStar+\delta}(b,\vartheta, \InducedMetric_0, k_0)
  \end{equation*}
  satisfies the following estimates
  \begin{equation}
    \label{nonlinear:eq:BA-extend:nbhd-construction:local-estimates}
    \begin{split}
      \sup_{\tStar<\TStar+\delta+1}e^{\BSDDecay\tStar}\norm*{h_{\TStar+\delta}(b, \vartheta, \InducedMetric_0, k_0)}_{\HkWithT{k}(\Sigma_{\tStar})}&\le \BSConstant, \qquad k\le \BSLowLvl, \\
    \sup_{\tStar<\TStar+\delta+1}e^{-\BSEGrowth\tStar}\norm*{h_{\TStar+\delta}(b, \vartheta, \InducedMetric_0, k_0)}_{\HkWithT{k}(\Sigma_{\tStar})}&\le \BSConstant,\qquad k\le \BSTopLvl,\\
    \abs*{b-b^0} + \abs*{\vartheta} &\le \BSConstant.
    \end{split}    
  \end{equation}
  If in addition, for some $(b,\vartheta)\in \mathcal{B}_{ext}$,
  $\delta<\delta_*$, we have that
  \begin{equation}
    \label{nonlinear:eq:BA-extend:nbhd-construction:ort-condition}
    \lambda_{\Upsilon}[g_b]\left(\NCal^{\TStar+\delta}\left(h_{\TStar+\delta}(b, \vartheta, \InducedMetric_0, k_0)\right), i_{b,\vartheta}(\InducedMetric_0, k_0)\right),
  \end{equation}
  then in fact for $\BSConstant$ sufficiently small, we have the
  global estimates
  \begin{equation}
    \label{nonlinear:eq:BA-extend:nbhd-construction:global-estimates}
    \begin{split}
      \sup_{\tStar>0}e^{\BSDDecay\tStar}\norm*{h_{\TStar+\delta}(b, \vartheta, \InducedMetric_0, k_0)}_{\HkWithT{k}(\Sigma_{\tStar})}&\le \BSConstant, \qquad k\le \BSLowLvl, \\
    \sup_{\tStar>0}e^{-\BSEGrowth\tStar}\norm*{h_{\TStar+\delta}(b, \vartheta, \InducedMetric_0, k_0)}_{\HkWithT{k}(\Sigma_{\tStar})}&\le \BSConstant,\qquad k\le \BSTopLvl. 
    \end{split}    
  \end{equation}
\end{lemma}

\begin{proof}
  We consider the neighborhood
  \begin{equation*}
    \mathcal{B}_{ext} := \curlyBrace*{(b,\vartheta):\abs*{b-b_{\TStar}}+\abs*{\vartheta-\vartheta_{\TStar}} <\delta_*},
  \end{equation*}
  and show that for a sufficiently small choice of $\delta_*$,
  $\mathcal{B}_{ext}$ satisfies the conditions in the lemma. 
  
  Observe that on $[0, \TStar]$,
  \begin{equation*}    
    \phi_{\TStar}^*(g_{b_{\TStar}} + h_{\TStar})
    = \phi_{\TStar+\delta_*}^*(g_{b_{\TStar+\delta_*}}+ h_{\TStar+\delta_*}),
  \end{equation*}
  so for some $\GronwallExp>0$ sufficiently large,
  \begin{equation}
    \label{nonlinear:eq:BA-extend:less-than-BSTime}
    \begin{split}
      \sup_{\tStar\le \TStar}e^{\BSDDecay\tStar}\norm*{h_{\TStar+\delta_*}-h_{\TStar}}_{\HkWithT{k}(\Sigma_{\tStar})}
      &\lesssim  e^{\GronwallExp\TStar}\delta_*, \qquad k\le \BSLowLvl,\\
      \sup_{\tStar\le \TStar}e^{-\BSEGrowth\tStar}\norm*{h_{\TStar+\delta_*}-h_{\TStar}}_{\HkWithT{k}(\Sigma_{\tStar})}
      &\lesssim e^{\GronwallExp\TStar}\delta_*, \qquad k\le \BSTopLvl.   
    \end{split}        
  \end{equation}
  On the interval $[\TStar, \TStar+1+\delta_*]$, for $\delta_*$
  sufficiently small, a local existence
  result and \eqref{nonlinear:eq:BA-extend:less-than-BSTime} yield that
  \begin{equation}
    \label{nonlinear:eq:BA-extend:BSTime-to-BSTime-plus-deltaExt}
    \begin{split}
      &\sup_{\tStar\in [\TStar, \TStar+1+\delta_*]}e^{\BSDDecay\tStar}\norm*{h_{\TStar+\delta_*}-h_{\TStar}}_{\HkWithT{k}(\Sigma_{\tStar})}\\
      \lesssim{}&  e^{\BSDDecay\TStar}\norm*{h_{\TStar+\delta_*}-h_{\TStar}}_{\HkWithT{k}(\Sigma_{\TStar})}
      + e^{\BSDDecay\TStar}\norm*{h_{\TStar}}_{\HkWithT{k}(\Sigma_{\TStar})},
      \qquad k\le \BSLowLvl,\\
      &\sup_{\tStar\in [\TStar, \TStar+1+\delta_*]}e^{-\BSEGrowth\tStar}\norm*{h_{\TStar+\delta_*}-h_{\TStar}}_{\HkWithT{k}(\Sigma_{\tStar})}\\
      \lesssim&{} e^{-\BSEGrowth\TStar}\norm*{h_{\TStar+\delta_*}-h_{\TStar}}_{\HkWithT{k}(\Sigma_{\TStar})}
      + e^{-\BSEGrowth\TStar}\norm*{h_{\TStar}}_{\HkWithT{k}(\Sigma_{\TStar})},
      \qquad k\le \BSTopLvl.   
    \end{split}    
  \end{equation}
  Combining
  \eqref{nonlinear:eq:BA-extend:less-than-BSTime} and 
  \eqref{nonlinear:eq:BA-extend:BSTime-to-BSTime-plus-deltaExt}, and the
  improvements to the bootstrap assumptions we have already proven
  in Propositions \ref{nonlinear:prop:BA-D:close}, \ref{nonlinear:prop:BA-E:close}, and
  \ref{nonlinear:prop:BA-gauge-improve} respectively, we have that
  \begin{align*}
    \sup_{\tStar<\TStar+1+\delta_*}e^{\BSDDecay\tStar}\norm*{h_{\TStar+\delta_*}}_{\HkWithT{k}(\Sigma_{\tStar})} &\lesssim \varepsilon_0 + \BSConstant^2 + e^{\GronwallExp \TStar}\delta_* , \qquad k\le \BSLowLvl,\\
    \sup_{\tStar<\TStar+1+\delta_*}e^{-\BSEGrowth\tStar}\norm*{h_{\TStar+\delta_*}}_{\HkWithT{k}(\Sigma_{\tStar})}&\lesssim \varepsilon_0 + \BSConstant^2 + e^{\GronwallExp \TStar}\delta_*, \qquad k\le \BSTopLvl,\\
    \abs*{b-b^0} + \abs*{\vartheta} &\lesssim \varepsilon_0 + \BSConstant^2 + \delta_*.
  \end{align*}
  The estimates in
  \eqref{nonlinear:eq:BA-extend:nbhd-construction:local-estimates} the follow by
  taking $\varepsilon_0$, $\BSConstant$, and $\delta_*$ sufficiently
  small.

  If in addition, the orthogonality condition in
  \eqref{nonlinear:eq:BA-extend:nbhd-construction:ort-condition} is satisfied,
  then the global estimates in
  \eqref{nonlinear:eq:BA-extend:nbhd-construction:global-estimates} follow
  directly from an application of Propositions
  \ref{nonlinear:prop:orthogonality-condition} and
  \ref{nonlinear:prop:BA-E:lossless-exp-grow-ILED}.
\end{proof}

We are now ready to prove Proposition \ref{nonlinear:prop:BA-gauge-extend}. 

\begin{proof}[Proof of Proposition \ref{nonlinear:prop:BA-gauge-extend}]
  The main goal will be to apply Lemma \ref{nonlinear:lemma:IFT-app:fixed-f}
  to bifurcate around the point $(b_{\TStar}, \vartheta_{\TStar},
  \InducedMetric_0, k_0, \TStar)$ in order to find an appropriate
  $(b_{\TStar+\deltaExt}, \vartheta_{\TStar+\deltaExt},
  \InducedMetric_0, k_0, \TStar+\deltaExt)$ satisfying the conditions
  in the proposition.

  First, observe that using the interpolation inequality in
  \eqref{nonlinear:eq:BA-D:interpolation-step-1}, we know that for
  $(b_{\TStar}, \vartheta_{\TStar})\in \mathcal{B}_{ext}\subset
  \Real^4\times\Real^{N_{\Theta}}$ and $\delta<\delta_*$, where
  $\mathcal{B}_{ext}$ and $\delta_*$ are as constructed in Lemma
  \ref{nonlinear:lemma:BA-extend:nbhd-construction}, we have that 
  \begin{equation}
    \label{nonlinear:eq:BA-extend:smallness-at-T}
    \norm*{\gamma_{\TStar}(h_{\TStar+\delta}(b, \vartheta, \InducedMetric_0, k_0))}_{\LSolHk{\BSLowLvl+2}(\Sigma_{\TStar})} \le \frac{\BSConstant}{2}.
  \end{equation}
  Now, to apply Lemma \ref{nonlinear:lemma:IFT-app:fixed-f}, we define
  \begin{gather*}
    \IFTNonlinearQuantity_{ext}: \Real^4\times \Real^{N_{\Theta}}\times H^{\BSLowLvl}(\Sigma_\TStar)\times H^{\BSLowLvl-1}(\Sigma_{\TStar})\times \Real^+\to \Real^4\times\Real^{N_{\Theta}},\\    
    (b, \vartheta, \InducedMetric, k, \delta)\mapsto \lambda_{\Upsilon}[g_{b}]\left(\NCal^{\TStar+\delta}_{g_b} \left(h(b,\vartheta, \InducedMetric, k, \delta)(\tStar+\TStar+\delta)\right), i_{b,\vartheta}(\InducedMetric_{\TStar}, k_{\TStar}) \right).
  \end{gather*}
  Observe from Proposition \ref{nonlinear:prop:lambda-time-translation} that
  \begin{equation*}
    \begin{split}
      &\lambda_{\Upsilon}[g_{b}]\left(\NCal^{\TStar+\delta}_{g_b} \left(h(b,\vartheta, \InducedMetric, k, \delta)(\tStar+\TStar)\right), i_{b,\vartheta}(\InducedMetric_{\TStar}, k_{\TStar}) \right)\\
      ={}& \lambda_{\Upsilon}[g_{b}]\left(\NCal^{\TStar+\delta}_{g_b} \left(h(b,\vartheta, \InducedMetric, k, \delta)(\tStar)\right), i_{b,\vartheta}(\InducedMetric_{0}, k_0) \right).  
    \end{split}    
  \end{equation*}

  The main step will be to show that
  $\evalAt*{D_{b,\vartheta}\IFTNonlinearQuantity_{ext}}_{(b_{\TStar},\vartheta_{\TStar},
    \InducedMetric_{\TStar}, k_{\TStar}, 0)}$ is invertible, where
  $(\InducedMetric_{\TStar}, k_{\TStar})$ denotes the induced metric
  and second fundamental form respectively by
  $g = \phi_{\TStar}^*(g_{b_{\TStar}}+h_{\TStar})$ on
  $\phi(\Sigma_{\TStar})$.

  To this end, observe that
  \begin{align}
    &\evalAt*{D_{b,\vartheta}i_{b,\vartheta}(\InducedMetric, k)}_{\left(b_{\TStar}, \vartheta_{\TStar}, \InducedMetric_{\TStar}, k_{\TStar}\right)}(b',\vartheta') \notag \\
    ={}& \evalAt*{D_{b,\vartheta}\gamma_{\TStar}\left(h_{\TStar}(b,\vartheta,\InducedMetric, k)\right)}_{\left(b_{\TStar}, \vartheta_{\TStar}, \InducedMetric_{\TStar}, k_{\TStar}\right)}(b',\vartheta')\notag \\
    ={}&
         \gamma_{\TStar}\left((g_{b_{\TStar}}')^{\Upsilon}(b') + \nabla_{g_{b_\TStar}}\otimes i_{\Theta}\vartheta'\right)
         +\gamma_{\TStar}\left((h_{{\TStar}}')^{\Upsilon}(b') + \nabla_{h_{\TStar}}\otimes i_{\Theta}\vartheta'\right). \label{nonlinear:eq:BA-extend:D-b-theta-i}
  \end{align}
  From Propositions \ref{nonlinear:prop:lambda-isomorphism} and
  \ref{nonlinear:prop:lambda-time-translation},
  $\lambda_{\Upsilon}[g_{b_{\TStar}}]\left(0,
    \gamma_{\TStar}\left((g_{b_{\TStar}}')^{\Upsilon}(b') +
      \nabla_{g_{b_\TStar}}\otimes i_{\Theta}\vartheta'\right)\right)$ is an
  isomorphism of $\Real^4\times\Real^{N_{\Theta}}$ to
  itself. Furthermore, since $h_{\TStar}$ satisfies the improved
  bootstrap estimates as proven in Propositions \ref{nonlinear:prop:BA-D:close}
  and \ref{nonlinear:prop:BA-E:close},
  \begin{equation}
    \label{nonlinear:eq:BA-extend:D-b-theta-i:err}
    \norm*{\gamma_{\TStar}\left((h_{{\TStar}}')^{\Upsilon}(b') + \nabla_{h_{\TStar}}\otimes i_{\Theta}\vartheta'\right)}_{\LSolHk{\BSLowLvl+1}(\Sigma)} \lesssim \varepsilon_0+\BSConstant^2.
  \end{equation}  
  Likewise, we have using local existence that
  \begin{align}
    &\sup_{\TStar\le \tStar \le \TStar+1}\norm*{\evalAt*{D_{b,\vartheta}\NCal^{\TStar+\deltaExt}_{g_{b}}\left(h_{\TStar}(b,\vartheta,\InducedMetric, k)\right)}_{\left(b_{\TStar}, \vartheta_{\TStar}, \InducedMetric_{\TStar}, k_{\TStar}, 0 \right)}}_{\HkWithT{\BSLowLvl-1}(\Sigma_{\tStar})}\notag \\
    \lesssim{}&  \sup_{\TStar\le \tStar \le \TStar+1}\norm*{h_{\TStar}(b_{\TStar}, \vartheta_{\TStar}, \InducedMetric_\TStar, k_{\TStar}, 0)}_{\HkWithT{\BSLowLvl+1}(\Sigma_{\tStar})}\norm*{\evalAt*{D_{b,\vartheta}h_{\TStar}}_{\left(b_{\TStar}, \vartheta_{\TStar}, \InducedMetric_{\TStar}, k_{\TStar}\right)} }_{\HkWithT{\BSLowLvl+1}(\Sigma_{\TStar})} \notag \\
    \lesssim{}& \BSConstant\norm*{\evalAt*{D_{b,\vartheta}h_{\TStar}}_{\left(b_{\TStar}, \vartheta_{\TStar}, \InducedMetric_{\TStar}, k_{\TStar}\right)} }_{\HkWithT{\BSLowLvl+1}(\Sigma_{\TStar})},  \label{nonlinear:eq:BA-extend:D-b-theta-N}
  \end{align}
  where the last inequality follows from the observation made in
  \eqref{nonlinear:eq:BA-extend:smallness-at-T}. 
  
  Moreover, from Lemma \ref{nonlinear:lemma:lambda-b-derivative}, we have that
  \begin{equation}
     \label{nonlinear:eq:BA-extend:D-b-theta-lambda}
    \abs*{\evalAt*{D_{b,\vartheta}\lambda_{\Upsilon}[g_{b}]}_{b_{\TStar}, \vartheta_{\TStar}}\left(\NCal^{\TStar+\deltaExt}_{g_b}(h_{\TStar}), i_{b_{\TStar},\vartheta_{\TStar}}(\InducedMetric_{\TStar}, k_{\TStar})\right)}\lesssim \BSConstant^2. 
  \end{equation}
  Combining \eqref{nonlinear:eq:BA-extend:D-b-theta-i},
  \eqref{nonlinear:eq:BA-extend:D-b-theta-N}, and
  \eqref{nonlinear:eq:BA-extend:D-b-theta-lambda} and taking $\BSConstant$
  sufficiently small then yields the invertibility of
  $\evalAt*{D_{b,\vartheta}\IFTNonlinearQuantity_{ext}}_{(b_{\TStar},
    \vartheta_{\TStar}, \InducedMetric_{\TStar}, k_{\TStar}, 0)}$,
  as desired.

  We can now apply Lemma \ref{nonlinear:lemma:IFT-app:fixed-f}, bifurcating
  around $(b_{\TStar}, \vartheta_{\TStar}, \InducedMetric_{\TStar},
  k_{\TStar}, 0)$ to see that there exists a neighborhood
  $\mathcal{X}_{ext}\subset H^{\BSLowLvl}(\Sigma_{\TStar}) \times
  H^{\BSLowLvl-1}(\Sigma_{\TStar})\times [0,1]$ on which there exists
  a mapping
  \begin{equation*}
    \IFTInverse_{ext}:\mathcal{X}_{ext}\to \mathcal{B}_{ext} 
  \end{equation*}
  such that
  \begin{equation*}
    \IFTNonlinearQuantity_{ext}(b,\vartheta,\InducedMetric, k, \delta) = 0\quad\iff\quad (b,\vartheta) = \IFTInverse(\InducedMetric, k, \delta). 
  \end{equation*}
  Thus, for $\deltaExt>0$ sufficiently small, it is clear that
  $\left(\InducedMetric_{\TStar}, k_{\TStar}, \deltaExt\right)\in
  \mathcal{X}_{ext}$, and thus, there exists some
  \begin{equation*}
    (b_{\TStar+\deltaExt}, \vartheta_{\TStar+\deltaExt}):=\IFTInverse(\InducedMetric_{\TStar}, k_{\TStar}, \deltaExt)
  \end{equation*}
  such that
  \begin{equation*}
    \IFTNonlinearQuantity_{ext}(b_{\TStar+\deltaExt}, \vartheta_{\TStar+\deltaExt}, \InducedMetric_{\TStar}, k_{\TStar}, \deltaExt) = 0,
  \end{equation*}
  and moreover, $(b_{\TStar+\deltaExt},
  \vartheta_{\TStar+\deltaExt})\in \mathcal{B}_{ext}$. Then, using
  Lemma \ref{nonlinear:lemma:BA-extend:nbhd-construction} concludes the proof of
  Proposition \ref{nonlinear:prop:BA-gauge-extend}.   
\end{proof}

\subsection{Closing the proof of Theorem \ref{nonlinear:thm:main}}
\label{nonlinear:sec:closing-proof}

We are now ready to prove Theorem \ref{nonlinear:thm:main} via a standard
continuity argument. Recall from Proposition \ref{nonlinear:prop:BA-justify}
that there exists some $\TStar>0$ sufficiently small so that
\eqref{nonlinear:BA-G}, \eqref{nonlinear:BA-D}, and \eqref{nonlinear:BA-E} are satisfied. We now let
$\BSTime$ be the supremum of all such $\TStar$, and assume for the
sake of contradiction that $\BSTime<+\infty$. By the continuity of the
flow, there exists $(b_{\BSTime}, \vartheta_{\BSTime})$ such that
\eqref{nonlinear:BA-G}, \eqref{nonlinear:BA-D}, and \eqref{nonlinear:BA-E} hold at
$\TStar=\BSTime$. 

But then, Proposition \ref{nonlinear:prop:BA-gauge-extend} states that there
exist choice of
$(b_{\BSTime+\deltaExt}, \vartheta_{\BSTime+\deltaExt})$ for which in
fact \eqref{nonlinear:BA-G}, \eqref{nonlinear:BA-D}, and \eqref{nonlinear:BA-E} continue to
hold. Thus, we have a contradiction and in fact, $\BSTime=+\infty$. We
thus have a family of black hole parameters and diffeomorphism
parameterizations $(b_{\tStar}, \vartheta_{\tStar}, h_{\tStar})$,
parametrized by $\tStar\to +\infty$ such that
$h_{\tStar} = h_{\tStar}(b_{\tStar},
\vartheta_{\tStar},\InducedMetric_0, k_0)$ solves
\eqref{nonlinear:eq:EVE:bootstrap-system} with
$(b_{\tStar}, \vartheta_{\tStar})$ satisfying \eqref{nonlinear:BA-G},
\eqref{nonlinear:BA-D}, and \eqref{nonlinear:BA-E}. In particular, from \eqref{nonlinear:BA-G},
$(b_{\tStar}, \vartheta_{\tStar})$ is a bounded, finite-dimensional
family. Thus, it must possess a convergent subsequence
\begin{equation*}
  \lim_{k\to +\infty}(b_{t_k}, \vartheta_{t_k}) \to (b_{\infty}, \vartheta_{\infty}).
\end{equation*}
Furthermore, $(b_{\infty}, \vartheta_{\infty})$ is the unique limit since
otherwise, distinct \KdS{} metrics would be diffeomorphic to each
other.

Thus, the solution $g$ to EVE is global and writing
$g = \phi^*_\infty(g_{b_\infty}+h)$, where
$\phi_\infty=e^{i_\Theta\vartheta_\infty}$, we have that
\begin{equation*}
  \sup_{\tStar>0}e^{-\BSDDecay\tStar}\norm{h}_{\HkWithT{3}(\Sigma_{\tStar})} \lesssim \varepsilon_0.
\end{equation*}
This concludes the proof of Theorem \ref{nonlinear:thm:main}.

\section{Energy estimates on perturbations of \KdS}
\label{nonlinear:sec:energy-perturb-kds}

This section contains all the necessary analysis on perturbations of
\KdS{} for the rest of the paper. 

\subsection{Proof of Proposition
  \ref{nonlinear:prop:BA-E:lossless-exp-grow-ILED}}
\label{nonlinear:sec:proof-of-perturbed-ILED}

We prove Proposition \ref{nonlinear:prop:BA-E:lossless-exp-grow-ILED}
by proving the following, stronger perturbation of the high-frequency
Morawetz estimate in Proposition
\ref*{linear:thm:resolvent-estimate:main} of \cite{fang_linear_2021}.
\begin{prop}
  \label{nonlinear:prop:BA-E:main-energy-estimate}
  Let $g = g_b + \tilde{g}$, where $g_b$ is a fixed, slowly-rotating
  \KdS{} metric, and
  \begin{equation}
    \label{nonlinear:eq:BA-E:g-perturb-control}
    \sup_{\tStar>0}e^{\BSDDecay\tStar}\norm{\tilde{g}}_{\HkWithT{3}(\Sigma_{\tStar})} \le \BSConstant.
  \end{equation}
  Moreover, consider some solution $h$ to the Cauchy problem given
  by
  \begin{align*}
    \LinEinsteinS_{g}  h &= f\\
    \gamma_0( h) &= ( h_0,  h_1),
  \end{align*}
  where $\LinEinsteinS_g$ is defined as in
  \eqref{nonlinear:eq:EVE:quasilinear-LinEinstein-def}, and the initial data
  $(h_0, h_1)\in \LSolHk{1}(\Sigma_0)$. 
  Also, let $\DomainOfIntegration = [0, \TStar]\times \Sigma$,
  $\TStar>0$.  Then for any fixed $\delta_0>0$, there exists a
  choice of $\BSConstant$ sufficiently small such that $h$
  satisfies the estimate
  \begin{equation}
    \label{nonlinear:eq:BA-E:main-energy-estimate}
    \norm*{e^{-\delta_0\tStar}h}_{H^1(\DomainOfIntegration)}
    \lesssim \norm{(h_0, h_1)}_{\LSolHk{1}(\Sigma_0)}
    + \norm*{e^{-\delta_0\tStar}h}_{\HkWithT{1}(\Sigma_{\TStar})}
    + \norm{e^{-\delta_0\tStar}f}_{L^2(\DomainOfIntegration)}
    + \norm{e^{-\delta_0\tStar}h}_{L^2(\DomainOfIntegration)}.
  \end{equation}
\end{prop}
The proof of Proposition \ref{nonlinear:prop:BA-E:main-energy-estimate} is
postponed to Section
\ref{nonlinear:sec:proof:prop:BA-E:main-energy-estimate}. 
\begin{remark}
  When compared to the Morawetz estimate in Theorem
  \ref*{linear:thm:resolvent-estimate:main} in \cite{fang_linear_2021}
  in the linear theory, the main estimate
  \eqref{nonlinear:eq:BA-E:main-energy-estimate} is weaker in the sense that it
  is consistent with exponential growth (albeit slow exponential
  growth) instead of exponential decay, but is stronger in the sense
  that it only involves non-degenerate norms, and moreover, does not
  lose any derivatives.
\end{remark}
\begin{remark}
  The appearance of $\HkWithT{3}(\Sigma)$ in
  \eqref{nonlinear:eq:BA-E:g-perturb-control} is chosen so that we have the
  control
  \begin{equation*}
    \sup_{\tStar>0}\left(\norm{\tilde{g}}_{W^{1,\infty}(\Sigma_{\tStar})} + \norm{\KillT\tilde{g}}_{L^\infty(\Sigma_{\tStar})}\right)
    \lesssim \BSConstant
  \end{equation*}
  by Sobolev inequalities. Observe that again, we only need $\HkWithT{s}(\Sigma)$,
  $s>\frac{5}{2}$, but since we choose $\BSLowLvl=3$ anyways, we do
  not make attempts to optimize this further. 
\end{remark}

Given Proposition \ref{nonlinear:prop:BA-E:main-energy-estimate}, Proposition
\ref{nonlinear:prop:BA-E:lossless-exp-grow-ILED} follows directly.
\begin{proof}[Proof of Proposition
  \ref{nonlinear:prop:BA-E:lossless-exp-grow-ILED}.]
  The estimate in \eqref{nonlinear:eq:BA-E:lossless-exp-grow-ILED} follows
  immediately taking $\delta_0 = \BSEGrowth$. 
\end{proof}

\subsection{Basic pseudo-differential theory}

Before we move onto the remainder of the section, we first introduce
the basics of the pseudo-differential analysis we will be using. For
an in-depth reference, we refer the reader to
Chapter 1 of \cite{alinhac_pseudo-differential_2007},  Chapter 18 of
\cite{hormander_analysis_2007}, or Chapters 1-4 of
\cite{taylor_pseudodifferential_1991}.

\begin{definition}
  For $m\in \Real$, we define $\Psi^{m}$ to be the class of order-$m$
  symbols on $\Real^d$, consisting of $C^\infty$ functions
  $a(\STPoint,\zeta)$ such that
  \begin{equation*}
    \abs*{D_{\STPoint}^\beta D_\zeta^\alpha a(\STPoint, \zeta) }\le C_{\alpha\beta}\bangle*{\zeta}^{m-|\alpha|}
  \end{equation*}
  for all multi-indexes $\alpha$, where $\bangle*{\zeta} =
  (1+|\zeta|^2)^{\frac{1}{2}}$. We also define the symbol class
  \begin{equation*}
    \Psi^{-\infty} := \bigcap_{m\in\mathbb{Z}}\Psi^{m}.
  \end{equation*}
  To each symbol is its associated \emph{pseudo-differential operator}
  acting on Schwartz functions $\phi$,
  \begin{equation*}
    a(x, D)(\STPoint) = Op(a)\phi(\STPoint) := (2\pi)^{-d}\int_{\Real^d}e^{\ImagUnit\STPoint\cdot \zeta}a(\STPoint, \zeta)\widehat{\phi}(\zeta)\,d\zeta,
  \end{equation*}
  where $\widehat{\phi}$ is the Fourier transform of $\phi$. 
\end{definition}

These pseudo-differential operators have well-behaved mapping
properties based on their symbol.
\begin{prop}
  If $a\in \Psi^m(\Real^d)$, then the operator $a(x, D)$ is a
  well-defined mapping from $H^s(\Real^d)$ to $H^{s-m}(\Real^d)$ for
  any $s\in \Real$. 
\end{prop}

\begin{definition}
  We call a symbol $a(\STPoint, \zeta)\in \Psi^m(\Real^d)$ and its
  corresponding operator $a(\STPoint, D)$ an \emph{elliptic symbol} and
    an \emph{elliptic operator} respectively if there exists some $c, C$ such
    that for $\bangle*{\zeta}>C$,
  \begin{equation*}
    \abs*{a(\STPoint, \zeta)}\ge c \bangle*{\zeta}^m.
  \end{equation*}
\end{definition}

\begin{prop}
  \label{nonlinear:prop:parametrix}
  If $a(\STPoint, \zeta)\in \Psi^m$ is elliptic, then it has a
  \emph{parametrix} $b(\STPoint, \zeta)\in \Psi^{-m}$ such that
  \begin{equation*}
    a(x, D)b(x, D) - \Identity \in Op(\Psi^{-\infty}), \qquad b(x, D)a(x, D) - \Identity \in Op(\Psi^{-\infty}).
  \end{equation*}  
\end{prop}

A crucial property of pseudo-differential operators is the following
commutation property.
\begin{prop}[Coifman-Meyer commutator estimate, see Proposition 4.1.A
  in \cite{taylor_pseudodifferential_1991}]
  \label{nonlinear:prop:Coifman-Meyer}
  For $f\in C^\infty$, $P\in OP\Psi^1$,
  \begin{equation*}
    \norm*{[P,f]u}_{L^2} \le C \norm*{f}_{C^1} \norm*{u}_{L^2}.
  \end{equation*}
\end{prop}

Finally, we can also define pseudo-differential operators on a
manifold.

\begin{prop}
  [See Proposition 7.1 in Chapter 1 of \cite{alinhac_pseudo-differential_2007}]
  Let $\phi:\Omega\to \Omega'$ be a smooth diffeomorphism between two
  open subsets of $\Real^d$. Moreover, let $a\in \SymClass^m$ be an
  order $m$ symbol such that the operator $a(\STPoint, D)$ has kernel
  with compact support in $\Omega\times\Omega$.

  Then the following hold. 
  \begin{enumerate}
  \item The function $a'(y, \zeta)$ defined by
    \begin{equation*}
      a'(\phi(\STPoint), \zeta) = e^{-\ImagUnit\phi(\STPoint)\cdot
        \zeta} a(\STPoint,D) e^{\ImagUnit \phi(\STPoint)\cdot \zeta},
      \qquad a' = 0 \text{ for } y\not\in \Omega',
    \end{equation*}
    is also a member of $\SymClass^m$.
  \item The kernel of $a'(\STPoint,D)$ has compact support in
    $\Omega'\times \Omega'$,
  \item For $u\in \TemperedDist(\Real^d)$,
    \begin{equation*}
      a(\STPoint, D)(u\circ \phi) = (a'(\STPoint, D)u)\circ \phi.
    \end{equation*}
  \item If $a$ has the form
    \begin{equation*}
      a= a_m \mod \SymClass^{m-1}(\Real^d),
    \end{equation*}
    where $a_m$ is a homogeneous symbol of order $m$, then the same is
    true for $a'$. That is, there is a homogeneous symbol $a_m'$ of
    order $m$ such that
    \begin{equation*}
      a'= a_m' \mod \SymClass^{m-1}(\Real^d),
    \end{equation*}
    and in fact
    \begin{equation*}
      a'_m(\phi(\STPoint), \zeta) = a_m(\STPoint, \tensor[^t]{{\chi'}}{}(\STPoint)\zeta).
    \end{equation*}
  \end{enumerate}
\end{prop}

\begin{definition}
  [See Definition 7.1 in Chapter 1 of \cite{alinhac_pseudo-differential_2007}]
  An operator $A:C_0^\infty(\mathcal{M})\to C^\infty(\mathcal{M})$ is
  a \emph{pseudo-differential operator of order} $m$ is for any
  coordinate system $\kappa:V\to \widetilde{V} \subset \Real^n$, the
  transported operator
  \begin{align*}
    \widetilde{A}:C_0^\infty(\widetilde{V}) &\to  C_0^\infty(\widetilde{V}),\\
    u&\mapsto A(u\circ \kappa)\circ \kappa^{-1}
  \end{align*}
  is a pseudo-differential of operator of order $m$ in
  $\widetilde{V}$. In other words, $A$ is a pseudo-differential
  operator of order $m$ if for all $\phi,\psi\in
  C^\infty_0(\widetilde{V})$, $\phi\widetilde{A}\psi\in Op(\SymClass^m)$.
\end{definition}

\subsection{Perturbations of linear estimates}

In this section, we prove a series of estimates for $\LinEinsteinS_g$
where $g = g_b + \tilde{g}$, $g_b$ is a slowly-rotating \KdS{} metric,
and $\tilde{g}$ is an exponentially decaying perturbation. 

We first introduce a series of basic results that stem from the fact
that $\tilde{g}$ is an exponentially decaying perturbation.
\begin{lemma}
  \label{nonlinear:lemma:BA-E:Deform-Tensor:perturb}
  Let $g = g_b + \tilde{g}$, where $g_b$ is a fixed, slowly-rotating
  \KdS{} metric, and $\tilde{g}$ satisfies the bootstrap
  assumptions \eqref{nonlinear:BA-E}, \eqref{nonlinear:BA-D}.  Then for $X$ a smooth
  vectorfield with deformation tensor $\DeformationTensor{X}$,
  \begin{equation}
    \label{nonlinear:eq:BA-E:Deform-Tensor:perturb}
    \DeformationTensor{\MorawetzVF}_g
    = \DeformationTensor{\MorawetzVF}_{g_b} + \DeformationTensorErr{\MorawetzVF},
  \end{equation}
  where
  \begin{equation*}
    \abs*{\DeformationTensorErr{\MorawetzVF}} \lesssim \BSConstant e^{-\BSDDecay \tStar}.
  \end{equation*}
\end{lemma}
\begin{proof}
  The inequality in \eqref{nonlinear:eq:BA-E:Deform-Tensor:perturb} follows
  immediately from observing that in local coordinates,
  \begin{equation*}
    (\DeformationTensor{X}_g)_{\mu\nu} = \frac{1}{2}\left(X(g_{\mu\nu})
      + g_{\gamma\nu} \frac{\p X^\gamma}{\p x^\mu}
      + g_{\gamma\mu} \frac{\p X^\gamma}{\p x^\nu}
    \right)
  \end{equation*}
  and that $g = g_b + \tilde{g}$.
\end{proof}

As a result, we immediately have the following corollary.
\begin{corollary}
  \label{nonlinear:coro:BA-E:K-Current:perturb}
  For $g$ as in Lemma \ref{nonlinear:lemma:BA-E:Deform-Tensor:perturb}, if $\MorawetzVF$
  is a smooth vectorfield, and $\LagrangeCorr$ is a smooth function, then 
  \begin{equation}
    \label{nonlinear:eq:BA-E:K-Current:perturb}
    \KCurrent{\MorawetzVF, \LagrangeCorr, m}_{g}[h]
    = \KCurrent{\MorawetzVF, \LagrangeCorr, m}_{g_b}[h]
    + \KCurrentPert{\MorawetzVF, \LagrangeCorr, m}[h], 
  \end{equation}
  where
  \begin{equation*}
    \KCurrentPert{\MorawetzVF, \LagrangeCorr, m}[h]
    \lesssim \BSConstant e^{-\BSDDecay \tStar}\left(|\nabla h|^2 + |h|^2\right). 
  \end{equation*}
  Moreover,
  \begin{equation*}
    \label{nonlinear:eq:BA-E:SX:perturb}
    \left(\SubPOp_{g}h\cdot \MorawetzVF\overline{h}\right) =
    \left(\SubPOp_{g_b}h\cdot \MorawetzVF\overline{h}\right)
    + \widetilde{\SubPOp}[\MorawetzVF][h], 
  \end{equation*}
  where
  \begin{equation*}
    \widetilde{\SubPOp}[\MorawetzVF][h] \lesssim \BSConstant e^{-\BSDDecay \tStar}|\nabla h|^2. 
  \end{equation*}
\end{corollary}
\begin{proof}
  The conclusion in \eqref{nonlinear:eq:BA-E:K-Current:perturb} follows directly
  from Lemma \ref{nonlinear:lemma:BA-E:Deform-Tensor:perturb} and the definition
  of $\KCurrent{X,q,m}[h]$ in \eqref{nonlinear:eq:J-K-currents:def}.

  The second conclusion follows from the observation that
  $\SubPOp_{g}$ is a linear combination of terms of the form
  \begin{equation*}
    S(g, \p g)^\mu\p_\mu,
  \end{equation*}
  where $S(g, \p g)$ is linear in its arguments. The conclusion then
  follows directly from the Sobolev inequality. 
\end{proof}

The redshift vector $\RedShiftN$ of Proposition
\ref{nonlinear:prop:redshift:N-construction} is uniformly timelike on
$\StaticRegionWithExtension$ with respect to slowly-rotating \KdS{}
metrics $g_b$. 
Since we are only considering metrics
$g=g_b+\tilde{g}$ which are small (and exponentially decaying)
perturbation of $g_b$, we must have that $\RedShiftN$ remains globally
timelike on $g$ as well.
\begin{lemma}
  \label{nonlinear:lemma:Energy:N-uniformly-timelike}
  Fix $g_b$ a slowly-rotating \KdS{} metric, and let $\RedShiftN$ be
  as constructed in Proposition
  \ref{nonlinear:prop:redshift:N-construction}. Then for $\BSConstant$ sufficiently
  small, if $\tilde{g}\in S^2T^*\StaticRegionWithExtension$ is a metric
  perturbation such that
  \begin{equation*}
    \sup_{\tStar>0}e^{\BSDDecay\tStar}\norm{\tilde{g}}_{\HkWithT{3}(\Sigma_{\tStar})}\le \BSConstant,
  \end{equation*}
  then $\RedShiftN$ is uniformly timelike with respect to the perturbed
  \KdS{} metric $g = g_b + \tilde{g}$,
  \begin{equation*}
    g(\RedShiftN, \RedShiftN) < 0. 
  \end{equation*}
\end{lemma}
\begin{proof}
  The result of the lemma then follows from
  \eqref{nonlinear:eq:RedShiftN-timelike} by choosing sufficiently small
  $\BSConstant$.
\end{proof}
The globally timelike nature of $\RedShiftN$ immediately gives the
following Gronwall-type estimate.
\begin{lemma}
  \label{nonlinear:lemma:Energy:Gronwall}
  Let $g = g_b + \tilde{g}$, where $g_b$ is a sufficiently
  slowly-rotating \KdS{} metric, and
  $\sup_{\tStar}e^{\BSDDecay\tStar}
  \norm{\tilde{g}}_{\FullHk{3}(\Sigma_{\tStar})} \le \BSConstant$,
  where $\BSConstant$ is sufficiently small.  Furthermore, let $h$ be
  a solution to the Cauchy problem
  \begin{equation}
    \label{nonlinear:eq:Energy:Gronwall-Cauchy-Prob}
    \begin{split}
      \LinEinsteinS_g h &= f,\\
      \gamma_0(h) &= (h_0, h_1),
    \end{split}    
  \end{equation}
  with $\LinEinsteinS_g$ as in
  \eqref{nonlinear:eq:EVE:quasilinear-LinEinstein-def}.  Then, there exists some
  $\GronwallExp$ such that the solution $h$ to
  \eqref{nonlinear:eq:Energy:Gronwall-Cauchy-Prob} satisfies the following
  energy estimate
  \begin{align}
    \sup_{\tStar\le \TStar} e^{-\GronwallExp\tStar} \norm*{h}_{\HkWithT{1}(\Sigma_{\tStar})}
    \lesssim {} &
                  \norm*{(h_0, h_1)}_{\LSolHk{1}(\Sigma_0)}
                  + \int_0^{\TStar}e^{-\GronwallExp\tStar}\norm*{f}_{\InducedLTwo(\Sigma_{\tStar})}\,d\tStar.   \label{nonlinear:eq:Energy:Gronwall-estimate}                   
  \end{align}
\end{lemma}

\begin{proof}                 
  Consider the energy on the spacelike slices
  $\Sigma_{\tStar}$,
  \begin{equation*}
    \RedShiftEnergy(\tStar)[h] = \int_{\Sigma_{\tStar}}\JCurrent{\RedShiftN, 0, 0}_g[h] \cdot n_{\Sigma_{\tStar}},
  \end{equation*}
  where $\JCurrent{\RedShiftN, 0, 0}_g[h]$ is as defined in
  \eqref{nonlinear:eq:div-them:J-K-currents}.  From Lemma
  \ref{nonlinear:lemma:Energy:N-uniformly-timelike}, $\RedShiftN$ is timelike on
  all of $\StaticRegionWithExtension$. As a result,
  \begin{equation}
    \label{nonlinear:eq:perturbed-redshift-coercivity}
    \RedShiftEnergy(\tStar)[h]
    \lesssim \norm*{\nabla h}_{L^2(\Sigma_{\tStar})}^2
    \lesssim \RedShiftEnergy(\tStar)[h].
  \end{equation}
  Now we apply the divergence theorem in Proposition
  \ref{nonlinear:prop:energy-estimates:spacetime-divergence-prop} with
  $X = \JCurrent{\RedShiftN, 0, 0}_g[h]$ on the region
  \begin{equation*}
    \DomainOfIntegration:= [0, \TStar]\times\Sigma,
  \end{equation*}
  where we recall the definition of $\Sigma$ in \eqref{nonlinear:eq:Sigma-def}. 
  
  Due to the timelike
  future-directed nature of $\RedShiftN$,  
  \begin{equation*}
    \int_{\Sigma_{\tStar} \bigcap \EventHorizonFuture_{-}}\JCurrent{\RedShiftN,0, 0}_g[h] \cdot n_{\EventHorizonFuture_{-}} \ge 0 ,\qquad
    \int_{\Sigma_{\tStar} \bigcap \CosmologicalHorizonFuture_{+} }\JCurrent{\RedShiftN,0, 0}_g[h] \cdot n_{\CosmologicalHorizonFuture_{+}} \ge 0.
  \end{equation*}
  To estimate the divergence term, we use
  \eqref{nonlinear:eq:div-them:J-K-currents} with $X=N, q=0$, i.e.
  \begin{equation*}
    \nabla_g\cdot \JCurrent{\RedShiftN, 0, 0}_g[h] = \Re\left[
      \RedShiftN\overline{h}\cdot \ScalarWaveOp[g]h
    \right]
    + \KCurrent{\RedShiftN, 0, 0}_g[h].
  \end{equation*}
  Using Corollary \ref{nonlinear:coro:BA-E:K-Current:perturb}, we can write
  \begin{equation*}
    \KCurrent{\RedShiftN, 0, 0}_g[h]
    = \KCurrent{\RedShiftN, 0, 0}_{g_b}[h] + \KCurrentPert{\RedShiftN, 0, 0}[h],
  \end{equation*}
  where
  \begin{equation*}
    \abs*{\KCurrentPert{\RedShiftN, 0, 0}[h]}\lesssim \BSConstant e^{-\BSDDecay\tStar}(|\nabla h|^2+ |h|^2). 
  \end{equation*}
  Since $\RedShiftN$ is uniformly timelike with respect to $g$
  on $\StaticRegionWithExtension$, we have that 
  \begin{equation*}
    \int_{\DomainOfIntegration}\KCurrent{\RedShiftN, 0, 0}_{g}[h]
    + \abs*{\Re\bangle*{\RedShiftN h, \SubPOp_g h}_{L^2(\DomainOfIntegration)}}
    + \abs*{\Re\bangle*{\RedShiftN h, \PotentialOp_g h}_{L^2(\DomainOfIntegration)}}
    \lesssim \int_0^{\TStar}\RedShiftEnergy(\tStar)[h]\,d\tStar + \norm{h}_{L^2(\DomainOfIntegration)}^2.
  \end{equation*}
  Thus, using the divergence theorem, there exists a constant
  $\GronwallExp$ sufficiently large, such that
  \begin{equation}
    \sup_{0\le \tStar\le\TStar}\RedShiftEnergy(\tStar)[h]
    \le \RedShiftEnergy(0)[h]
    + C\int_0^{\TStar}\norm*{\LinEinsteinS_gh}_{\InducedLTwo(\Sigma_{\tStar})}^2\,d\tStar
    + \GronwallExp \int_{0}^{\TStar}\RedShiftEnergy(\tStar)[h]\,d\tStar .
  \end{equation}
  Then, directly via Gronwall's Lemma, we have that
  \begin{equation*}
    e^{-\GronwallExp\tStar}\RedShiftEnergy(\tStar)[h] \lesssim \RedShiftEnergy(0)[h] + \int_0^{\TStar}e^{-\GronwallExp\tStar}\norm*{\LinEinsteinS_g h}_{\InducedLTwo(\Sigma_{\tStar})}^2\,d\tStar,
  \end{equation*}
  from which the conclusion follows immediately given
  \eqref{nonlinear:eq:perturbed-redshift-coercivity}.
\end{proof}

We also have the following analogue of Lemma
\ref{linear:lemma:ILED-near:D-t-minus-1-control} in
\cite{fang_linear_2021}, which follows from a simple
integration-by-parts argument. 
\begin{lemma}
  \label{nonlinear:lemma:BA-E:D-t-minus-1-control}
  Let $g = g_b + \tilde{g}$, where $g_b$ is a sufficiently
  slowly-rotating \KdS{} metric, and
  $\sup_{\tStar}e^{\BSDDecay\tStar}\norm{\tilde{g}}_{\FullHk{3}(\Sigma_{\tStar})}
  \le \BSConstant$, where $\BSConstant$ is sufficiently small, and let
  \begin{equation*}
    \DomainOfIntegration := [0, \TStar] \times \Sigma,
  \end{equation*}
  and let $h$ be a sufficiently regular function on
  $\DomainOfIntegration$ such that $h(\tStar, \cdot)$ is compactly
  supported on $\Sigma$ for all $\tStar>0$. Then
  \begin{equation*}
    \norm*{\bangle*{D_x}^{-1}D_{\tStar}h}_{L^2(\DomainOfIntegration)}
    \lesssim \norm*{\LinEinsteinS_{g}h}_{H^{-1}(\DomainOfIntegration)}
    + \norm*{h}_{L^2(\DomainOfIntegration)}
    + \norm*{h}_{H^1(\Sigma_{\TStar})}
    + \norm*{h}_{H^1(\Sigma_{0})}
    . 
  \end{equation*}
\end{lemma}

\begin{proof}
  Let $R_{-1}\in \Op S^{-1}(\TrappingNbhd)$ be a compactly supported
  self-adjoint operator. We use $R_{-1}$ as a Lagrangian
  multiplier. Integrating by parts, we have that 
  \begin{align*}
    &2\bangle*{(g^{\tStar\tStar})^{-1}\ScalarWaveOp[g_b]h, R_{-1}^2h}_{L^2(\DomainOfIntegration)}\\
    ={}& \norm*{R_{-1}D_{\tStar}h}_{L^2(\DomainOfIntegration)}^2
         + O\left(
         (1+\norm*{\tilde{g}}_{W^{1,\infty}(\DomainOfIntegration)})
         \left(\norm*{R_{-1}D_{\tStar}h}_{L^2(\DomainOfIntegration)}\norm*{h}_{L^2(\DomainOfIntegration)}
         + \norm*{h}_{L^2(\DomainOfIntegration)}^2\right)\right)\\        
    & + \evalAt*{\bangle*{n_{\TrappingNbhd}h, R_{-1}h}_{\InducedLTwo(\TrappingNbhd)}}_{\tStar=0}^{\tStar=\TStar}. 
  \end{align*}
  Thus, applying Cauchy-Schwarz,
  \begin{equation}
    \label{nonlinear:eq:ILED-near:D-t-minus-1-control:aux1}
    \norm*{R_{-1}\p_{\tStar}h}_{L^2(\DomainOfIntegration)}^2    
    \lesssim \norm*{\LinEinsteinS_{g}h}_{L^2(\DomainOfIntegration)}^2
    + \norm*{h}_{L^2(\DomainOfIntegration)}^2
    + \norm*{h}_{\HkWithT{1}(\Sigma_{\TStar})}^2
    + \norm*{h}_{\HkWithT{1}(\Sigma_{0})}^2. 
  \end{equation}
  as desired. 
\end{proof}

We now recall the following result from the linear theory:
\begin{prop}
  \label{nonlinear:prop:perturbed-ILED:lin-theory}
  Fix  $\delta>0$. Let $g_b$ be a fixed, sufficiently
  slowly-rotating \KdS{} metric, and let $\DomainOfIntegration = [0,
  \TStar]\times\Sigma$. 

  Then, there exist some vectorfield multipliers $\MorawetzOuterVF$
  and $\MorawetzInnerNTVF$, smooth function $\LagrangeCorrOuter$ and
  $\LagrangeCorrInnerNT$, one-form $\breve{m}$, elliptic zero-order
  operator $\PseudoSubPFixer$, and cutoff functions
  \begin{equation*}
    \mathring{\chi},\breve{\chi},\check{\chi}\in \Op S^0, 
  \end{equation*}
  such that the following properties hold.
  \begin{enumerate}
  \item $\mathring{\chi}$ and $\breve{\chi}$ satisfy the conditions
    that
    \begin{equation}
      \label{nonlinear:eq:perturbed-ILED:lin-theory:cutoff:def}
      \Sigma \subset \supp \mathring{\chi}\bigcup\supp  \breve{\chi}\bigcup\supp  \check{\chi},\qquad
      \supp \p\chi_1\bigcap\supp \p \chi_2 = \emptyset, \quad \chi_1, \chi_2 \in \{\mathring{\chi},\breve{\chi},\check{\chi}\}.
    \end{equation}
  \item The following inequalities hold,
    \begin{align}
      &\int_{\DomainOfIntegration}\KCurrent{\MorawetzOuterVF, \LagrangeCorrOuter, m}_{g_b}[\breve{\chi}h]
        - \bangle*{\SubPOp_{g_b}[\breve{\chi}h], \MorawetzOuterVF (\breve{\chi}h)}_{L^2(\DomainOfIntegration)}
        + \int_{\DomainOfIntegration}\KCurrent{\MorawetzInnerNTVF, \LagrangeCorrInnerNT, 0}_{g_b}[\breve{\chi}h]
        - \bangle*{\SubPOp_{g_b}[\check{\chi}h], \MorawetzInnerNTVF (\check{\chi}h)}_{L^2(\DomainOfIntegration)}
        \notag \\
      \gtrsim{}& \norm*{\breve{\chi}h}_{H^1(\DomainOfIntegration)}^2
                 + \norm*{\check{\chi}h}_{H^1(\DomainOfIntegration)}^2
                 +\norm*{C\nabla(\breve{\chi}h)}_{L^2(\DomainOfIntegration)}\norm*{\MorawetzOuterVF(\breve{\chi}h)}_{L^2(\DomainOfIntegration)}\notag \\
                & +\norm*{C\nabla(\check{\chi}h)}_{L^2(\DomainOfIntegration)}\norm*{\MorawetzOuterVF(\check{\chi}h)}_{L^2(\DomainOfIntegration)}
                 - C_1 \norm{h}_{L^2(\DomainOfIntegration)}^2 ,\label{nonlinear:eq:perturbed-ILED:lin-theory:non-trapping-bulk}
    \end{align}
    and 
    \begin{align}    
      \bangle*{\SubPConjOp_{g_b}[\mathring{\chi}h], \KillT(\mathring{\chi}h)}_{L^2(\DomainOfIntegration)}
      &\le \delta \norm{\mathring{\chi}h}_{H^1(\DomainOfIntegration)}^2  + C_1\norm{h}_{L^2(\DomainOfIntegration)}^2, \label{nonlinear:eq:perturbed-ILED:lin-theory:trapping-bulk}
    \end{align}
    where
    \begin{equation*}
      \SubPConjOp_{g_b} = \PseudoSubPFixer \SubPOp_{g_b} \PseudoSubPFixer^- + \PseudoSubPFixer \left[\ScalarWaveOp[g_b],\PseudoSubPFixer^-\right],
    \end{equation*}
    where we denote by $\PseudoSubPFixer^-$ the parametrix of
    $\PseudoSubPFixer$, and
    \begin{align}      
      \norm*{h}_{H^1(\DomainOfIntegration)}^2
      \lesssim{}& \norm*{h}_{L^2(\DomainOfIntegration)}^2
                  + \int_{\DomainOfIntegration}\KCurrent{\MorawetzOuterVF, \LagrangeCorrOuter, \breve{m}}_{g_b}[\breve{\chi}h]
                  - \bangle*{\SubPOp_{g_b}[\breve{\chi}h], \MorawetzOuterVF (\breve{\chi}h)}_{L^2(\DomainOfIntegration)}
                  + \int_{\DomainOfIntegration}\KCurrent{\MorawetzInnerNTVF, \LagrangeCorrInnerNT, 0}_{g_b}[\check{\chi}h]\notag\\
                  & - \bangle*{\SubPOp_{g_b}[\check{\chi}h], \MorawetzInnerNTVF (\check{\chi}h)}_{L^2(\DomainOfIntegration)}
                  + \int_{\DomainOfIntegration}\KCurrent{\KillT, 0, 0}_{g_b}[\mathring{\chi}h]
                  - \bangle*{\SubPConjOp_{g_b}[\mathring{\chi}h], \KillT (\mathring{\chi}h)}_{L^2(\DomainOfIntegration)}\notag\\
                  & + \bangle*{\left[\ScalarWaveOp[g_b], \breve{\chi}\right]h, \breve{\chi}\MorawetzOuterVF h}_{L^2(\DomainOfIntegration)}
                  + \bangle*{\left[\ScalarWaveOp[g_b], \check{\chi}\right]h, \check{\chi}\MorawetzInnerNTVF h}_{L^2(\DomainOfIntegration)}
                  + \bangle*{\left[\ScalarWaveOp[g_b], \mathring{\chi}\right]h, \mathring{\chi}\KillT h}_{L^2(\DomainOfIntegration)} 
                  \label{nonlinear:eq:perturbed-ILED:lin-theory:gluing}.
    \end{align}
  \item $\KillT$ is uniformly timelike with respect to $g_b$ on the support of
    $\mathring{\chi}$,
    \begin{equation}
      \label{nonlinear:eq:perturbed-ILED:lin-theory:KillT-timelike}
      g_b(\KillT, \KillT) < 0, \qquad \text{ on }\supp\mathring{\chi}.
    \end{equation}  
  \item $\JCurrent{X,q,m}_{g_b}[\breve{\chi}h]$ satisfies the following
    properties 
    along $\EventHorizonFuture_-$ and $\CosmologicalHorizonFuture_+$, 
    \begin{equation}
      \label{nonlinear:eq:perturbed-ILED:lin-theory:boundary-flux:nontrapping}
      \int_{\EventHorizonFuture_-}\JCurrent{\MorawetzOuterVF,\LagrangeCorrOuter,\breve{m}}_{g_b}[\breve{\chi}h]\cdot n_{\EventHorizonFuture_-} \ge 0, \qquad
      \int_{\CosmologicalHorizonFuture_+}\JCurrent{\MorawetzOuterVF,\LagrangeCorrOuter,\breve{m}}_{g_b}[\breve{\chi}h]\cdot n_{\CosmologicalHorizonFuture_+} \ge 0.
    \end{equation}
    On the other hand, $\mathring{\chi}h, \check{\chi}{h}$ vanish
    along $\EventHorizonFuture_-$, $\CosmologicalHorizonFuture_+$. 
  \end{enumerate}
\end{prop}
\begin{remark}
  We remark that $\PseudoSubPFixer$, $\mathring{\chi}$, and
  $\breve{\chi}$ are all pseudodifferential operators, although their
  exact construction is irrelevant here. 
\end{remark}
\begin{proof}
  Let
  \begin{equation*}
    \mathring{\chi}= \mathring{c}\mathring{\chi}\mathring{\chi}_\zeta, \qquad
    \check{\chi} = \check{c}\mathring{\chi}\breve{\chi}_\zeta,\qquad
    \breve{\chi} = \breve{\chi},
  \end{equation*}
  as in Section \ref{linear:sec:ILED} of \cite{fang_linear_2021},
  and
  \begin{equation*}
    (\MorawetzOuterVF, \LagrangeCorrOuter, \breve{m}) := (\MorawetzOuterVF + \dot{c}\dot{\chi}^2\RedShiftN, \LagrangeCorrOuter, \breve{m}),\qquad
    (\MorawetzInnerNTVF, \LagrangeCorrInnerNT, 0) := (\MorawetzInnerNTVF, \LagrangeCorrInnerNT, 0)
  \end{equation*}
  where the right-hand sides of the definitions are as defined in
  \eqref{linear:eq:ILED-nontrapping:XOuter-fOuter-def} and Lemma
  \eqref{linear:lemma:ILED-nontrapping:extension}  of
  \cite{fang_linear_2021}, and in
  \eqref{linear:eq:ILED-nontrapping-freq:op-defs} of
  \cite{fang_linear_2021} respectively.

  The bulk estimate in
  \eqref{nonlinear:eq:perturbed-ILED:lin-theory:non-trapping-bulk} is contained
  in Theorem \ref{linear:thm:ILED-redshift:main}, Lemma
  \ref{linear:lemma:ILED-nontrapping:bulk-positivity}, and Lemma
  \ref{linear:lemma:ILED-nontrapping-freq:bulk-positivity} in
  \cite{fang_linear_2021}.

  The combined bulk estimate containing the commutation error terms in
  \eqref{nonlinear:eq:perturbed-ILED:lin-theory:gluing} is directly implied by   
  Lemma \ref{linear:lemma:ILED-full:Trapping-NonTrapping:case-2} in
  \cite{fang_linear_2021}.  
  
  The flux estimate in
  \eqref{nonlinear:eq:perturbed-ILED:lin-theory:boundary-flux:nontrapping} is
  contained in
  \eqref{linear:eq:ILED-nontrapping:extension:boundary-flux} in
  \cite{fang_linear_2021}.

  The remaining statements follow by construction of
  $\mathring{\chi}$, $\check{\chi}$, $\breve{\chi}$, and the definition of $g_b$. 
\end{proof}

These properties are in fact preserved under
perturbation by an exponentially decaying perturbation $\tilde{g}$. 

\begin{lemma}
  \label{nonlinear:lemma:perturbed-ILED:perturbed-bulk}
  Fix some $C>0$,  $\delta>0$. Let $g_b$ be a fixed, sufficiently
  slowly-rotating \KdS{} metric. Then let the spacetime region
  $\DomainOfIntegration$, vectorfield multipliers $\MorawetzOuterVF$
  and $\MorawetzInnerNTVF$, smooth function $\LagrangeCorrOuter$ and
  $\LagrangeCorrInnerNT$, one-form $\breve{m}$, elliptic zero-order
  operator $\PseudoSubPFixer$, and cutoffs
  $\mathring{\chi}, \check{\chi}$ and $\breve{\chi}$ be those of
  Proposition \ref{nonlinear:prop:perturbed-ILED:lin-theory}. Furthermore, let
  $h$ be a function compactly supported on $\DomainOfIntegration$. If
  \begin{equation*}
    g = g_b + \tilde{g}, \qquad \sup_{\tStar>0}e^{\BSDDecay\tStar}\norm{\tilde{g}}_{\HkWithT{3}(\Sigma_{\tStar})} \le \BSConstant,
  \end{equation*}
  then for sufficiently small $\BSConstant$, the following properties
  hold.
  \begin{enumerate}
  \item There exists some $C_1>0$ sufficiently large so that 
    \begin{align}
      &\int_{\DomainOfIntegration}\KCurrent{\MorawetzOuterVF, \LagrangeCorrOuter, \breve{m}}_{g}[\breve{\chi}h]
        - \bangle*{\SubPOp_{g}[\breve{\chi}h], \MorawetzOuterVF (\breve{\chi}h)}_{L^2(\DomainOfIntegration)}                
        + \int_{\DomainOfIntegration}\KCurrent{\MorawetzInnerNTVF, \LagrangeCorrInnerNT, 0}_{g}[\breve{\chi}h]
        - \bangle*{\SubPOp_{g}[\check{\chi}h], \MorawetzInnerNTVF (\check{\chi}h)}_{L^2(\DomainOfIntegration)}
        \notag \\
      \gtrsim{}& \norm*{\breve{\chi}h}_{H^1(\DomainOfIntegration)}^2
                 + \norm*{\check{\chi}h}_{H^1(\DomainOfIntegration)}^2
                 +\norm*{C\nabla(\breve{\chi}h)}_{L^2(\DomainOfIntegration)}\norm*{\MorawetzOuterVF(\breve{\chi}h)}_{L^2(\DomainOfIntegration)}\notag \\
                & +\norm*{C\nabla(\check{\chi}h)}_{L^2(\DomainOfIntegration)}\norm*{\MorawetzOuterVF(\check{\chi}h)}_{L^2(\DomainOfIntegration)}
             - C_1 \norm{h}_{L^2(\DomainOfIntegration)}^2,    \label{nonlinear:eq:perturbed-ILED:non-trapping-bulk} 
    \end{align}
    and
    \begin{align}        
      \int_{\DomainOfIntegration}\KCurrent{\KillT, 0, 0}_{g}[\mathring{\chi}h]
      -\bangle*{\SubPConjOp_{g}[\mathring{\chi}h], \KillT (\mathring{\chi}h)}_{L^2(\DomainOfIntegration)}
      &\le \delta \norm{\mathring{\chi}h}_{H^1(\DomainOfIntegration)}^2
        + C_1\norm{\mathring{\chi}h}_{L^2(\DomainOfIntegration)}^2, \label{nonlinear:eq:perturbed-ILED:trapping-bulk}
    \end{align}
    where
    \begin{equation}
      \label{nonlinear:eq:perturbed-ILED:tilde-S:def}
      \SubPConjOp_{g} = \SubPConjOp_{g_b} + \widetilde{\SubPOp}_{g,g_b},\qquad
      \widetilde{\SubPOp}_{g,g_b} := \PseudoSubPFixer \left[\ScalarWaveOp[g] - \ScalarWaveOp[g_b], \PseudoSubPFixer^{-}\right] + \PseudoSubPFixer \left(\SubPOp_g-\SubPOp_{g_b}\right)\PseudoSubPFixer^{-} ,
    \end{equation}
    and
    \begin{align}      
      \norm*{h}_{H^1(\DomainOfIntegration)}^2
      \lesssim{}& \norm*{h}_{L^2(\DomainOfIntegration)}^2
                  + \int_{\DomainOfIntegration}\KCurrent{\MorawetzOuterVF, \LagrangeCorrOuter, \breve{m}}_{g}[\breve{\chi}h]
                  - \bangle*{\SubPOp_{g}[\breve{\chi}h], \MorawetzOuterVF (\breve{\chi}h)}_{L^2(\DomainOfIntegration)}
                  + \int_{\DomainOfIntegration}\KCurrent{\MorawetzInnerNTVF, \LagrangeCorrInnerNT, 0}_{g}[\check{\chi}h]\notag\\
                  & - \bangle*{\SubPOp_{g}[\check{\chi}h], \MorawetzInnerNTVF (\check{\chi}h)}_{L^2(\DomainOfIntegration)}
                  + \int_{\DomainOfIntegration}\KCurrent{\KillT, 0, 0}_{g}[\mathring{\chi}h]
                  - \bangle*{\SubPConjOp_{g}[\mathring{\chi}h], \KillT (\mathring{\chi}h)}_{L^2(\DomainOfIntegration)}\notag\\
                  & + \bangle*{\left[\ScalarWaveOp[g], \breve{\chi}\right]h, \breve{\chi}\MorawetzOuterVF h}_{L^2(\DomainOfIntegration)}
                  + \bangle*{\left[\ScalarWaveOp[g], \check{\chi}\right]h, \check{\chi}\MorawetzInnerNTVF h}_{L^2(\DomainOfIntegration)}
                  + \bangle*{\left[\ScalarWaveOp[g], \mathring{\chi}\right]h, \mathring{\chi}\KillT h}_{L^2(\DomainOfIntegration)} 
                  \label{nonlinear:eq:perturbed-ILED:gluing}.
    \end{align}
      
  \item $\KillT$ is uniformly timelike with respect to $g$ on the
    support of $\mathring{\chi}$, i.e.
    \begin{equation}
      \label{nonlinear:eq:perturbed-ILED:KillT-timelike}
      g(\KillT, \KillT) < 0,\quad \text{ on }\supp \mathring{\chi}.
    \end{equation}
  \item $\JCurrent{\MorawetzOuterVF,\LagrangeCorrOuter,\breve{m}}_{g}[\breve{\chi}h]$ has the following
    properties along $\EventHorizonFuture_-$ and
    $\CosmologicalHorizonFuture_+$,
    \begin{equation}
      \label{nonlinear:eq:perturbed-ILED:boundary-flux:nontrapping}
      \int_{\EventHorizonFuture_-}\JCurrent{\MorawetzOuterVF,\LagrangeCorrOuter,\breve{m}}_{g}[\breve{\chi}h]\cdot n_{\EventHorizonFuture_-} \ge 0, \qquad
      \int_{\CosmologicalHorizonFuture_+}\JCurrent{\MorawetzOuterVF,\LagrangeCorrOuter,\breve{m}}_{g}[\breve{\chi}h]\cdot n_{\CosmologicalHorizonFuture_+} \ge 0.
    \end{equation}
    On the other hand, $\mathring{\chi}h$ and $\check{\chi}h$ vanish along
    $\EventHorizonFuture_-$, $\CosmologicalHorizonFuture_+$.
  \end{enumerate}  
\end{lemma}
\begin{proof}
  Observe that in light of Proposition
  \ref{nonlinear:prop:perturbed-ILED:lin-theory} and
  \ref{nonlinear:lemma:BA-E:Deform-Tensor:perturb}, the first conclusion in
  \eqref{nonlinear:eq:perturbed-ILED:non-trapping-bulk} follows immediately.

  For the second and third conclusion in
  \eqref{nonlinear:eq:perturbed-ILED:gluing} and
  \eqref{nonlinear:eq:perturbed-ILED:trapping-bulk} respectively, we use
  following observation. For any zero-order pseudo-differential
  operator $P_0\in \Op S^0(\Sigma)$,
  \begin{equation*}
    \left[\ScalarWaveOp[g] - \ScalarWaveOp[g_b], P_0\right]
    = \tilde{g}^{\alpha\beta} \left[\p_\alpha\p_\beta, P_0\right]
    +\left[\tilde{g}^{\alpha\beta}, P_0\right]\p_\alpha\p_\beta
    + \norm*{\tilde{g}}_{W^{1,\infty}(\DomainOfIntegration)}\Op S^0(\Sigma)
    .
  \end{equation*}
  Then since $\tilde{g}=g-g_b$ is a scalar function, we have by the
  classical Coifman-Meyer commutator estimates in
  Proposition \ref{nonlinear:prop:Coifman-Meyer}  
  \begin{equation}
    \label{nonlinear:eq:perturbed-ILED:zero-order-commutation-control}
    \norm*{\tilde{g}^{\alpha\beta} \left[\p_\alpha\p_\beta, P_0\right]h 
      + \left[\tilde{g}^{\alpha\beta}, P_0\right]\p_\alpha\p_\beta h}_{L^2(\DomainOfIntegration)}
    \lesssim \norm{\tilde{g}}_{W^{1,\infty}(\DomainOfIntegration)}\norm{h}_{H^1(\DomainOfIntegration)}. 
  \end{equation}  
  
  To prove the second conclusion in
  \eqref{nonlinear:eq:perturbed-ILED:trapping-bulk}, we first observe that since
  $\KillT$ is Killing on $g_b$, applying Lemma
  \ref{nonlinear:lemma:BA-E:Deform-Tensor:perturb} immediately yields that
  \begin{equation}
    \label{nonlinear:eq:perturbed-ILED:K-t-on-g:aux}
    \abs*{\int_{\DomainOfIntegration}\KCurrent{\KillT, 0, 0}_{g}[\mathring{\chi}h]}
    \lesssim \norm{\tilde{g}}_{W^{1,\infty}(\DomainOfIntegration)}\norm{\mathring{\chi}h}_{H^1(\DomainOfIntegration)}^2.
  \end{equation}
  Applying, \eqref{nonlinear:eq:perturbed-ILED:zero-order-commutation-control}
  with $P_0 = \PseudoSubPFixer^{-}$, and using Proposition
  \ref{nonlinear:prop:perturbed-ILED:lin-theory} we observe that there exists
  some $C_1$ such that
  \begin{equation*}
    \abs*{\bangle*{\widetilde{\SubPOp}_{g,g_b}(\mathring{\chi}h), \KillT(\mathring{\chi}h)}_{L^2(\DomainOfIntegration)}}
    \lesssim \norm{\tilde{g}}_{W^{1,\infty}(\DomainOfIntegration)}\norm{h}_{H^1(\DomainOfIntegration)}^2
    + C_1\norm{h}_{L^2(\DomainOfIntegration)}^2.
  \end{equation*}
  The inequality in \eqref{nonlinear:eq:perturbed-ILED:non-trapping-bulk} then
  follows by using Proposition \ref{nonlinear:prop:perturbed-ILED:lin-theory}
  and combining with \eqref{nonlinear:eq:perturbed-ILED:K-t-on-g:aux}, taking
  $\BSConstant$ sufficiently small.

  To prove the third conclusion in \eqref{nonlinear:eq:perturbed-ILED:gluing},
  it suffices to show that 
  \begin{equation}
    \label{nonlinear:eq:perturbed-ILED:bulk:aux1}
    \abs*{\bangle*{\left[\ScalarWaveOp[g] - \ScalarWaveOp[g_b], \breve{\chi}\right]h, \MorawetzVF(\breve{\chi}h)}_{L^2(\DomainOfIntegration)}}
    + \abs*{\bangle*{\left[\ScalarWaveOp[g] - \ScalarWaveOp[g_b], \mathring{\chi}\right]h, \KillT(\mathring{\chi}h)}_{L^2(\DomainOfIntegration)}}
    \lesssim \norm*{\tilde{g}}_{W^{1,\infty}(\DomainOfIntegration)}\norm*{h}_{H^1(\DomainOfIntegration)}. 
  \end{equation}
  But \eqref{nonlinear:eq:perturbed-ILED:bulk:aux1} follows
  directly from
  \eqref{nonlinear:eq:perturbed-ILED:zero-order-commutation-control} with
  $P_0 = \breve{\chi}$, and $P_0 = \mathring{\chi}$.
  Then \eqref{nonlinear:eq:perturbed-ILED:gluing} follows from 
  \eqref{nonlinear:eq:perturbed-ILED:lin-theory:non-trapping-bulk}
  \eqref{nonlinear:eq:perturbed-ILED:lin-theory:boundary-flux:nontrapping},  
  \eqref{nonlinear:eq:perturbed-ILED:lin-theory:gluing}, \eqref{nonlinear:eq:perturbed-ILED:lin-theory:trapping-bulk},
  where we take $\BSConstant$ to be sufficiently small to conclude.

  The remaining properties \eqref{nonlinear:eq:perturbed-ILED:KillT-timelike},
  \eqref{nonlinear:eq:perturbed-ILED:boundary-flux:nontrapping}, follow directly
  from \eqref{nonlinear:eq:perturbed-ILED:lin-theory:KillT-timelike},
  \eqref{nonlinear:eq:perturbed-ILED:lin-theory:boundary-flux:nontrapping}, by
  taking $\BSConstant$ sufficiently small.
\end{proof}

\begin{lemma}
  \label{nonlinear:lemma:perturbed-ILED:main-components}
  Let $g_b$ be a fixed, slowly-rotating \KdS{} metric, and fix
  $\delta_0>0$, $\TStar>0$, and denote
  \begin{equation*}
    \DomainOfIntegration = [0,\TStar]\times \Sigma.
  \end{equation*}
  Let $\phi = e^{-\delta_0 \tStar}h$.
  Then there exists a choice of elliptic zero-order operator
  $\PseudoSubPFixer$, and cutoff operators $\mathring{\chi}$ and
  $\breve{\chi}$ satisfying
  \eqref{nonlinear:eq:perturbed-ILED:lin-theory:cutoff:def} such that for 
  \begin{align}    
    \norm{\breve{\chi}\phi}_{H^1(\DomainOfIntegration)}
    + \norm*{\check{\chi}\phi}_{H^1(\DomainOfIntegration)}
    \lesssim{}&  \norm{e^{-\delta_0\tStar}\LinEinsteinS_g (\breve{\chi}h)}_{L^2(\DomainOfIntegration)}
      + \norm{e^{-\delta_0\tStar}\LinEinsteinS_g (\check{\chi}h)}_{L^2(\DomainOfIntegration)}\notag \\
      & + \norm*{\phi}_{L^2(\DomainOfIntegration)}
      + \norm*{\phi}_{\HkWithT{1}(\Sigma_{\TStar})}
      + \norm*{\phi}_{\HkWithT{1}(\Sigma_{0})},
      \label{nonlinear:eq:perturbed-ILED:main-components:non-trapping} \\
      % \sup_{\tStar<\TStar}\norm{\mathring{\chi}\phi}_{\HkWithT{1}(\Sigma_{\tStar})}
      % + 
    \norm*{\mathring{\chi}\phi}_{H^1(\DomainOfIntegration)}
    \lesssim{}& \norm{e^{-\delta_0\tStar}\LinEinsteinS_{g}(\mathring{\chi}h)}_{L^2(\DomainOfIntegration)}
      + \norm*{\phi}_{L^2(\DomainOfIntegration)}
      + \norm*{\phi}_{\HkWithT{1}(\Sigma_{\TStar})}
      + \norm*{\phi}_{\HkWithT{1}(\Sigma_{0})}, \label{nonlinear:eq:perturbed-ILED:main-components:trapping}
  \end{align}
  where
  \begin{equation*}
    \LinEinsteinConjS_g := \PseudoSubPFixer \LinEinsteinS_g \PseudoSubPFixer^-
  \end{equation*}
\end{lemma}
\begin{proof}
  Observe that for $\phi = e^{-\delta_0\tStar} h$, we have that
  \begin{equation*}
    \LinEinsteinS_{g}\phi
    + \GInvdtdt_g^{-1}\left(
      2\delta_0\KillT + \delta_0^2\right)\phi
    = e^{-\delta_0\tStar}\LinEinsteinS_gh. 
  \end{equation*}
  Both of the conclusions in
  \eqref{nonlinear:eq:perturbed-ILED:main-components:non-trapping} and
  \eqref{nonlinear:eq:perturbed-ILED:main-components:trapping} are direct
  applications of the vectorfield multipliers, although the proof of
  \eqref{nonlinear:eq:perturbed-ILED:main-components:trapping} is slightly more
  technical due to the inclusion of the pseudo-differential
  conjugation by $\PseudoSubPFixer$. 
  
  We begin with the first conclusion,
  \eqref{nonlinear:eq:perturbed-ILED:main-components:non-trapping}. Applying the
  divergence theorem in Proposition
  \ref{nonlinear:prop:energy-estimates:spacetime-divergence-prop} with
  $\JCurrent{\MorawetzOuterVF,\LagrangeCorrOuter,\breve{m}}_g[\breve{\chi}h]$,
  we have that
  \begin{align*}
    & - \Re\bangle*{e^{-\delta_0\tStar}\LinEinsteinS_g(\breve{\chi}h), (\MorawetzOuterVF+\LagrangeCorrOuter)(\breve{\chi}\phi)}_{L^2(\DomainOfIntegration)}\\
    ={}& \int_{\DomainOfIntegration} \KCurrent{\MorawetzOuterVF, \LagrangeCorrOuter, \breve{m}}_g[\breve{\chi}\phi]
         - \Re\bangle*{\SubPOp_g(\breve{\chi}\phi), (\MorawetzOuterVF+\LagrangeCorrOuter) (\breve{\chi}\phi)}_{L^2(\DomainOfIntegration)}
         + \Re\bangle*{\PotentialOp_g(\breve{\chi}\phi), (\MorawetzOuterVF+\LagrangeCorrOuter)(\breve{\chi}\phi)}_{L^2(\DomainOfIntegration)}\\
    &- \Re\bangle*{\SubPOp_{g}(\breve{\chi}\phi), \LagrangeCorrOuter\breve{\chi}\phi}_{L^2(\DomainOfIntegration)}
    + \bangle*{\left(-\delta_0\SubPOp_0 + \GInvdtdt_g^{-1}(2\delta_0\KillT  +\delta_0^2)\right)\breve{\chi}\phi, (\MorawetzOuterVF+\LagrangeCorrOuter)(\breve{\chi}\phi)}_{L^2(\DomainOfIntegration)}\\
    &+ \int_{\EventHorizonFuture_-}\JCurrent{\MorawetzOuterVF,\LagrangeCorrOuter,\breve{m}}_{g}[\breve{\chi}\phi]\cdot n_{\EventHorizonFuture_-} 
    +  \int_{\CosmologicalHorizonFuture_+}\JCurrent{\MorawetzOuterVF,\LagrangeCorrOuter,\breve{m}}_{g}[\breve{\chi}\phi]\cdot n_{\CosmologicalHorizonFuture_+},
  \end{align*}
  where
  \begin{equation*}
    \SubPOp = \SubPOp_0\p_{\tStar} + \SubPOp_1,
  \end{equation*}
  where $\SubPOp_i$ is an $i$-th order differential operator that does
  not involve $\p_{\tStar}$.
  Similarly, applying the divergence theorem in Proposition
  \ref{nonlinear:prop:energy-estimates:spacetime-divergence-prop} with
  $\JCurrent{\MorawetzInnerNTVF,\LagrangeCorrInnerNT,0}_g[\check{\chi}h]$,
  we have that
  \begin{align*}
    & - \Re\bangle*{e^{-\delta_0\tStar}\LinEinsteinS_g(\check{\chi}h), (\MorawetzInnerNTVF+\LagrangeCorrInnerNT)(\check{\chi}\phi)}_{L^2(\DomainOfIntegration)}\\
    ={}& \int_{\DomainOfIntegration} \KCurrent{\MorawetzInnerNTVF, \LagrangeCorrInnerNT, 0}_g[\check{\chi}\phi]
         - \Re\bangle*{\SubPOp_g(\check{\chi}\phi), (\MorawetzInnerNTVF+\LagrangeCorrInnerNT) (\check{\chi}\phi)}_{L^2(\DomainOfIntegration)}
         + \Re\bangle*{\PotentialOp_g(\check{\chi}\phi), (\MorawetzInnerNTVF+\LagrangeCorrInnerNT)(\check{\chi}\phi)}_{L^2(\DomainOfIntegration)}\\
    &- \Re\bangle*{\SubPOp_{g}(\check{\chi}\phi), \LagrangeCorrInnerNT\check{\chi}\phi}_{L^2(\DomainOfIntegration)}
      + \bangle*{\left(-\delta_0\SubPOp_0 + \GInvdtdt_g^{-1}(2\delta_0\KillT  +\delta_0^2)\right)\check{\chi}\phi, (\MorawetzInnerNTVF+\LagrangeCorrInnerNT)(\check{\chi}\phi)}_{L^2(\DomainOfIntegration)}\\
    &+ \int_{\EventHorizonFuture_-}\JCurrent{\MorawetzInnerNTVF,\LagrangeCorrInnerNT,0}_{g}[\check{\chi}\phi]\cdot n_{\EventHorizonFuture_-} 
      +  \int_{\CosmologicalHorizonFuture_+}\JCurrent{\MorawetzInnerNTVF,\LagrangeCorrInnerNT,0}_{g}[\check{\chi}\phi]\cdot n_{\CosmologicalHorizonFuture_+}.
  \end{align*}

  Choosing $C$ sufficiently large, we have from Lemma
  \ref{nonlinear:lemma:perturbed-ILED:perturbed-bulk}, we then have that up to
  lower-order terms,
  \begin{align*}
    &\norm*{\breve{\chi}\phi}_{H^1(\DomainOfIntegration)}^2
    + \norm*{\check{\chi}\phi}_{H^1(\DomainOfIntegration)}^2\notag \\
    \lesssim{}&
                \int_{\DomainOfIntegration} \KCurrent{\MorawetzOuterVF,\LagrangeCorrOuter,\breve{m}}_g[\breve{\chi}\phi]
                - \Re\bangle*{\SubPOp_g(\breve{\chi}\phi), \MorawetzOuterVF (\breve{\chi}\phi)}_{L^2(\DomainOfIntegration)}
                + \bangle*{ \GInvdtdt_g^{-1}(2\delta_0\KillT  +\delta_0^2)\breve{\chi}\phi, \MorawetzOuterVF(\breve{\chi}\phi)}_{L^2(\DomainOfIntegration)}\notag\\
               & + \int_{\DomainOfIntegration} \KCurrent{\MorawetzInnerNTVF,\LagrangeCorrInnerNT,0}_g[\check{\chi}\phi]
                - \Re\bangle*{\SubPOp_g(\check{\chi}\phi), \MorawetzInnerNTVF (\check{\chi}\phi)}_{L^2(\DomainOfIntegration)}
                 + \bangle*{ \GInvdtdt_g^{-1}(2\delta_0\KillT  +\delta_0^2)\check{\chi}\phi, \MorawetzInnerNTVF(\check{\chi}\phi)}_{L^2(\DomainOfIntegration)}
                .
  \end{align*}
  Then, we conclude using the control of the boundary fluxes in
  \eqref{nonlinear:eq:perturbed-ILED:boundary-flux:nontrapping} and applying
  Cauchy-Schwarz, to control the lower-order terms, and using the bulk
  control in \eqref{nonlinear:eq:perturbed-ILED:non-trapping-bulk}. 
  
  We now prove the second conclusion,
  \eqref{nonlinear:eq:perturbed-ILED:main-components:trapping}. In this case, we
  will first use the multiplier argument with the conjugated operator
  $\LinEinsteinConjS_g $, and then show that this is sufficient to
  conclude. First observe that we can write
  \begin{equation*}
    \LinEinsteinConjS_g
    = \ScalarWaveOp_g 
    + \SubPConjOp_{g}    
    + \PotentialConjOp_g,
  \end{equation*}
  where $\SubPConjOp_{g}$ is as defined in
  \eqref{nonlinear:eq:perturbed-ILED:tilde-S:def} and
  \begin{equation*}
    \label{nonlinear:eq:perturbed-ILED:VConj-def}
    \PotentialConjOp_g = \PseudoSubPFixer \PotentialOp_{g} \PseudoSubPFixer^{-}.
  \end{equation*}

  Then, we can apply the divergence theorem in Proposition
  \ref{nonlinear:prop:energy-estimates:spacetime-divergence-prop} with the
  vectorfield multiplier
  $\JCurrent{\KillT, \delta_1, 0}_g[\mathring{\chi}\phi]$, on the region
  $\DomainOfIntegration$, to yield that
  \begin{align*}
    &\widetilde{\EnergyKill}[\TStar](\mathring{\chi}\phi)
      + \int_{\DomainOfIntegration}\KCurrent{\KillT, 0, 0}_g[\mathring{\chi}\phi]
      + 2\delta_0\norm*{\KillT \mathring{\chi}\phi}_{L^2(\DomainOfIntegration)}^2
    + \delta_1\int_{\DomainOfIntegration} g^{\alpha\beta}\p_\alpha( \mathring{\chi}\phi) \cdot \p_\beta(\overline{\mathring{\chi}\phi})\\
    ={}& \widetilde{\EnergyKill}[0](\mathring{\chi}\phi)
    - \bangle*{e^{-\delta_0\tStar}\LinEinsteinConjS_g(\mathring{\chi}h), (\KillT+\delta_1) (\mathring{\chi}\phi)}_{L^2(\DomainOfIntegration)}
    + \bangle*{\SubPConjOp_g(\mathring{\chi}\phi), (\KillT+\delta_1) (\mathring{\chi}\phi)}_{L^2(\DomainOfIntegration)}\\
         &+ \bangle*{\PotentialConjOp_g(\mathring{\chi}\phi), (\KillT+\delta_1) (\mathring{\chi}\phi)}_{L^2(\DomainOfIntegration)}
           - \delta_0^2\bangle*{\mathring{\chi}\phi, (\KillT+\delta_1) (\mathring{\chi}\phi)}_{L^2(\DomainOfIntegration)}
           + \bangle*{\delta_0\SubPOp_0(\mathring{\chi}\phi), (\KillT+\delta_1) (\mathring{\chi}\phi)}_{L^2(\DomainOfIntegration)}
           ,
  \end{align*}
  where the boundary fluxes vanish since $\mathring{\chi}h$ vanishes
  on the boundaries, and
  \begin{equation*}
    \widetilde{\EnergyKill}[\tStar](\mathring{\chi}\phi) := \int_{\Sigma_{\tStar}}\JCurrent{\KillT, \delta_1, 0}_g[\mathring{\chi}\phi]\cdot n_{\Sigma_{\tStar}}, 
  \end{equation*}
  which we emphasize may not be positive.
  
  First observe that for $\delta_1$ appropriately small with respect
  to $\delta_0$, 
  \begin{equation*}
    2\delta_0\norm*{\sqrt{A_g}^{-1}\KillT (\mathring{\chi}\phi)}_{L^2(\DomainOfIntegration)}^2
    + \delta_1\int_{\DomainOfIntegration} (g_b)^{\alpha\beta}\p_\alpha (\mathring{\chi}\phi)\cdot \p_\beta (\overline{\mathring{\chi}\phi})
    \gtrsim \delta_0\norm*{\mathring{\chi}\phi}_{H^1(\DomainOfIntegration)}^2. 
  \end{equation*}
  Then, using Lemma \ref{nonlinear:lemma:perturbed-ILED:perturbed-bulk} and
  \eqref{nonlinear:eq:perturbed-ILED:lin-theory:trapping-bulk} we see
  that for $\BSConstant$ sufficiently small and $\delta$ sufficiently
  small, we can guarantee that  
  \begin{align*}
    \norm*{\mathring{\chi}\phi}_{H^1(\DomainOfIntegration)}^2
    \lesssim{}& 2\delta_0\norm*{\sqrt{A_g}^{-1}\KillT (\mathring{\chi}\phi)}_{L^2(\DomainOfIntegration)}^2
                + \int_{\DomainOfIntegration}\KCurrent{\KillT, 0, 0}_g[\mathring{\chi}\phi]\\
              &+ \delta_1\int_{\DomainOfIntegration} g^{\alpha\beta}\p_\alpha (\mathring{\chi}\phi)\cdot \p_\beta(\overline{\mathring{\chi}\phi})
                - \bangle*{\SubPConjOp_g(\mathring{\chi}\phi), \KillT (\mathring{\chi}\phi)}_{L^2(\DomainOfIntegration)}.
  \end{align*}  
  Then, by a Cauchy-Schwarz argument we have that
  \begin{equation}
    \label{nonlinear:eq:perturbed-ILED:main-components:trapping:conjugated}
    \norm*{\mathring{\chi}\phi}_{H^1(\DomainOfIntegration)}^2
    \lesssim 
    \norm*{e^{-\delta_0\tStar}\LinEinsteinConjS_g h}_{L^2(\DomainOfIntegration_{\sStar})}^2    
    + C_1 \norm{\phi}_{L^2(\DomainOfIntegration_{\sStar})}^2
    + \norm*{\phi}_{\HkWithT{1}(\Sigma_{\TStar})}^2
    + \norm*{\phi}_{\HkWithT{1}(\Sigma_{0})}^2
    , 
  \end{equation}
  as desired.

  We now show that we can recover
  \eqref{nonlinear:eq:perturbed-ILED:main-components:trapping} from
  \eqref{nonlinear:eq:perturbed-ILED:main-components:trapping:conjugated}.
  We first observe that since
  $\PseudoSubPFixer, \PseudoSubPFixer^-\in \Op S^0(\Sigma)$, we have
  by combining standard elliptic estimates with Lemma
  \ref{nonlinear:lemma:BA-E:D-t-minus-1-control} that
  \begin{equation*}
    \norm*{\mathring{\chi}\phi}_{H^1(\DomainOfIntegration)}
    \lesssim \norm*{\PseudoSubPFixer (\mathring{\chi}\phi)}_{H^1(\DomainOfIntegration)}
    + \norm*{e^{-\delta_0\tStar}\LinEinsteinS_g (\mathring{\chi}h)}_{L^2(\DomainOfIntegration)}
    + \norm*{\mathring{\chi}\phi}_{L^2(\DomainOfIntegration)}
    + \norm*{\mathring{\chi}\phi}_{\HkWithT{1}(\Sigma_{0})}
    + \norm*{\mathring{\chi}\phi}_{\HkWithT{1}(\Sigma_{\TStar})}. 
  \end{equation*}
  Moreover, using
  \eqref{nonlinear:eq:perturbed-ILED:main-components:trapping:conjugated}, we
  have that
  \begin{equation*}
    \norm*{\PseudoSubPFixer (\mathring{\chi}\phi)}_{H^1(\DomainOfIntegration)}
    \lesssim \norm*{e^{-\delta_0\tStar}\PseudoSubPFixer \LinEinsteinS_g (\mathring{\chi}h)}_{L^2(\DomainOfIntegration)}
    + \norm*{e^{-\delta_0\tStar}\PseudoSubPFixer \LinEinsteinS_g\circ R_{-\infty}(\mathring{\chi}h)}_{L^2(\DomainOfIntegration)},
  \end{equation*}
  where $R_{-\infty} = 1 - \PseudoSubPFixer^-\PseudoSubPFixer \in \Op S^{-\infty}(\Sigma)$.
  Then,
  \begin{equation*}
    \norm*{e^{-\delta_0\tStar}\PseudoSubPFixer \LinEinsteinS_g \circ R_{-\infty}(\mathring{\chi}h)}_{L^2(\DomainOfIntegration)}
    \lesssim \norm*{e^{-\delta_0\tStar}\LinEinsteinS_g(\mathring{\chi}h)}_{L^2(\DomainOfIntegration)}
    + \norm*{e^{-\delta_0\tStar}D_{\tStar}(\mathring{\chi}h)}_{H^{-1}(\DomainOfIntegration)}.
  \end{equation*}
  We conclude the proof of
  \eqref{nonlinear:eq:perturbed-ILED:main-components:trapping}, and with it the
  proof of Lemma \ref{nonlinear:lemma:perturbed-ILED:main-components}, using the
  control of
  $\norm*{D_{\tStar}(\mathring{\chi}h)}_{H^{-1}(\DomainOfIntegration)}$
  from Lemma \ref{nonlinear:lemma:BA-E:D-t-minus-1-control}.
\end{proof}

\subsection{Proof of Proposition \ref{nonlinear:prop:BA-E:main-energy-estimate}}
\label{nonlinear:sec:proof:prop:BA-E:main-energy-estimate}

We are now ready to prove Proposition
\ref{nonlinear:prop:BA-E:main-energy-estimate}.
\begin{proof}[Proof of Proposition
  \ref{nonlinear:prop:BA-E:main-energy-estimate}]
  We observe that the estimates in
  \eqref{nonlinear:eq:perturbed-ILED:main-components:non-trapping} and
  \eqref{nonlinear:eq:perturbed-ILED:main-components:trapping} are almost enough
  to conclude immediately. The only complication is the presence of
  the cutoff functions inside of $\LinEinsteinS_g$. To commute the
  cutoffs, we use the fact that up to lower-order terms, 
  \begin{align*}
     &\bangle*{\LinEinsteinS_g(\breve{\chi}h), (X+q)(\breve{\chi}h)}_{L^2(\DomainOfIntegration)}
    + \bangle*{\LinEinsteinS_g(\mathring{\chi}h), \KillT(\mathring{\chi}h)}_{L^2(\DomainOfIntegration)}\\
    ={}&  \bangle*{\breve{\chi}\LinEinsteinS_gh, (X+q)(\breve{\chi}h)}_{L^2(\DomainOfIntegration)}
    + \bangle*{\mathring{\chi}\LinEinsteinS_gh, \KillT(\mathring{\chi}h)}_{L^2(\DomainOfIntegration)}
    + \bangle*{\left[\ScalarWaveOp[g], \breve{\chi}\right]h, \breve{\chi}(X+q)h}_{L^2(\DomainOfIntegration)}\\
    &+ \bangle*{\left[\ScalarWaveOp[g], \mathring{\chi}\right]h, \mathring{\chi}\KillT h}_{L^2(\DomainOfIntegration)}. 
  \end{align*}
  Then, using \eqref{nonlinear:eq:perturbed-ILED:gluing}, we see that the extra
  commutation terms can be absorbed into the positive bulk, and the
  rest of the argument proceeds exactly as above for Lemma
  \ref{nonlinear:lemma:perturbed-ILED:main-components}.   
\end{proof}

\printbibliography[heading=bibintoc]

\end{document}
%%% Local Variables:
%%% mode: latex
%%% TeX-master: t
%%% End: